\documentclass[11pt]{amsart}
\usepackage{amsmath, amsthm, mathabx,amscd, amsfonts, amssymb, hyperref, mathrsfs, textcomp, epsfig, a4wide, graphicx, IEEEtrantools, tikz, verbatim, xypic}
\usepackage[mathscr]{euscript}

\usetikzlibrary{shapes,arrows}
\usetikzlibrary{matrix,decorations.pathreplacing,angles,quotes,cd}

\newtheorem{teor}{Theorem}
\newtheorem{thm}{Theorem}[section]
\newtheorem{prop}[thm]{Proposition}
\newtheorem{lemma}[thm]{Lemma}
\newtheorem*{Quest}{Question}
\newtheorem{quest}{Question}
\newtheorem{cor}[thm]{Corollary}
\newcommand{\spin}{\mathfrak{s}}
\theoremstyle{definition}

\theoremstyle{remark}
\newtheorem{remark}{Remark}[section]

     \RequirePackage{rotating}                   
    \def\HSt{%
       \setbox0=\hbox{$\widehat{\mathit{HS}}$}
       \setbox1=\hbox{$\mathit{HS}$}
       \dimen0=1.1\ht0
       \advance\dimen0 by 1.17\ht1
       \smash{\mskip2mu\raise\dimen0\rlap{%
          \begin{turn}{180}
              {$\widehat{\phantom{\mathit{HS}}}$}
           \end{turn}} \mskip-2mu    
                \mathit{HS}
    }{\vphantom{\widehat{\mathit{HS}}}}{}}

     \RequirePackage{rotating}                   
    \def\HMt{%
       \setbox0=\hbox{$\widehat{\mathit{HM}}$}
       \setbox1=\hbox{$\mathit{HM}$}
       \dimen0=1.1\ht0
       \advance\dimen0 by 1.17\ht1
       \smash{\mskip2mu\raise\dimen0\rlap{%
          \begin{turn}{180}
              {$\widehat{\phantom{\mathit{HM}}}$}
           \end{turn}} \mskip-2mu    
                \mathit{HM}
    }{\vphantom{\widehat{\mathit{HM}}}}{}}

\newcommand{\vol}{\mathrm{vol}}
\newcommand{\C}{\mathbb{C}}
\newcommand{\R}{\mathbb{R}}

\begin{document}

\title[Monopoles on hyperbolic three-manifolds]{The Seiberg-Witten equations and the length spectrum of hyperbolic three-manifolds} 

\author{Francesco Lin}
\address{Department of Mathematics, Columbia University} 
\email{flin@math.columbia.edu}

\author{Michael Lipnowski}
\address{Department of Mathematics and Statistics, McGill University} 
\email{michael.lipnowski@mcgill.ca}
\begin{abstract}
We exhibit the first examples of hyperbolic three-manifolds for which the Seiberg-Witten equations do not admit any irreducible solution. Our approach relies on hyperbolic geometry in an essential way; it combines an explicit upper bound for the first eigenvalue on coexact $1$-forms $\lambda_1^*$ on rational homology spheres which admit irreducible solutions together with a version of the Selberg trace formula relating the spectrum of the Laplacian on coexact $1$-forms with the volume and complex length spectrum of a hyperbolic three-manifold. Using these relationships, we also provide precise numerical bounds on $\lambda_1^*$ for several hyperbolic rational homology spheres.
\end{abstract}
\maketitle

In the last three decades, both hyperbolic geometry and Floer homology have played a central role in the study of the geometry and topology of three-dimensional manifolds (see for example \cite{Agol}, \cite{Gabai}, \cite{KM1}, \cite{KM2}, \cite{Tau}). Despite this, and even though both subjects have by now reached their maturity, their mutual interaction (if any) remains extremely mysterious. For example, while the computation of the Floer homology for the Seifert fibered case is very well-understood in explicit, geometric terms \cite{FS}, \cite{MOY}, the Floer homology of hyperbolic manifolds (i.e. admitting a metric with constant sectional curvature $-1$) has eluded similar descriptions. Because Mostow rigidity implies that the geometric invariants of a hyperbolic metric are indeed topological invariants, the following is a very natural yet outstanding problem one encounters.
\begin{Quest}
For a hyperbolic three-manifold $Y$, is there any relationship between the topological invariants arising from the hyperbolic geometry of $Y$ (e.g. the volume, injectivity radius, lengths of geodesics, etc.) and the invariants arising from Floer homology?
\end{Quest}

In the present paper we discuss, for a hyperbolic-three manifold $Y$ with $b_1(Y)=0,$ a relationship between the existence of irreducible solutions to the Seiberg-Witten equations on $Y$ and the hyperbolic geometry of $Y.$ As a testing ground, we explore this relationship for the first $50$ manifolds in the Hodgson-Weeks census, which is a good approximation to the complete list of hyperbolic three-manifolds with volume $<6.5$ and injectivity radius $>0.15$ \cite{HW}. Our main application is the following.
\begin{teor}\label{Thm1}
Let $Y$ be one of the hyperbolic three-manifolds from the Hodgson-Weeks census listed in Table \ref{table1}. Then, for the hyperbolic metric, for any spin$^c$ structure the Seiberg-Witten equations on $Y$ (for sufficiently small perturbations) do not admit irreducible solutions\footnote{This result (and the following Theorem $2$) takes as input the computations of the length spectrum provided by the \texttt{length\_spectrum()} method of SnapPy version $2.6.1$ \cite{CDGW}. These are very accurate (especially for the small manifolds we are dealing with in the paper), but are not yet certified using interval arithmetic in the current version. There is promising work towards this end \cite{Trnkova2017} using the certified hyperbolic structure produced in \cite{HIKMOT2016}.}.
\end{teor}
The only previously known examples of Riemannian rational homology three-spheres with no irreducible solutions were provided by manifolds with positive scalar curvature, and the Hantzsche-Wendt manifold (the only rational homology three-sphere with a flat metric), \cite{KM}. In this sense, the manifolds in Table \ref{table1} are also the first examples of hyperbolic three-manifolds for which the set of solutions to the Seiberg-Witten equations is determined explicitly.

\begin{table}[h!]
 \begin{tabular}{|c| c| c|} 
 \hline
 Census label & Volume & Injectivity radius \\ 
 \hline\hline
 0 &  0.94270\ldots &  0.29230\ldots \\ 
 \hline
 2 & 1.01494\ldots & 0.41572\ldots \\
 \hline
 3 & 1.26371\ldots & 0.28753\ldots \\
 \hline
 8 & 1.42361\ldots  & 0.17618\ldots \\
 \hline
 12 & 1.54356\ldots & 0.16768\ldots\\
 \hline
  13 & 1.54356\ldots& 0.28903\ldots\\
 \hline
14 & 1.58316\ldots & 0.27889\ldots \\
 \hline
15 & 1.58316\ldots &  0.38874\ldots \\
 \hline
  16 & 1.58864\dots & 0.26727\ldots\\
 \hline
22 & 1.83193\ldots & 0.26532\ldots\\
 \hline
 25 & 1.83193\ldots &  0.26531\ldots \\
 \hline
  28 & 1.88541\ldots& 0.29230\ldots\\
 \hline
   29 & 1.88541\ldots & 0.19853\ldots \\
 \hline
   30 & 1.88541\ldots& 0.19853\ldots\\     
 \hline
  31 & 1.88541\ldots & 0.29230\ldots \\
  \hline
  32 & 1.88591\ldots & 0.20593\ldots\\
      \hline
 33 & 1.91084\ldots & 0.22107\ldots\\
 \hline
39 & 1.96274\ldots & 0.21576\ldots  \\ 
   \hline
    40 & 1.96274\ldots& 0.28904\ldots\\
 \hline
   42 & 2.02395\ldots& 0.17922\ldots\\
   \hline
44 & 2.02988\ldots & 0.43127\ldots \\
    \hline
   46 & 2.02988\ldots& 0.27177\ldots\\
     \hline
49 &2.02988\ldots& 0.21564\ldots\\
     \hline
\end{tabular}

\caption{The hyperbolic manifolds of Theorem \ref{Thm1}.}
\label{table1}
\end{table}

As a direct consequence of Theorem \ref{Thm1}, the manifolds in Table \ref{table1} are \textit{$L$-spaces}, i.e. their reduced Floer homology groups $\mathit{HM}$ vanishes. This had been previously shown by Dunfield \cite{Dun} in the setting of Heegaard Floer homology; the latter is known to be isomorphic to monopole Floer homology (see \cite{KLT}, \cite{CGH} and subsequent papers). In fact, he has determined exactly which spaces in the Hodgson-Weeks census are $L$-spaces; in this regard, Table \ref{table1} comprises $23$ of the $28$ $L$-spaces with label less than $49$.
\begin{remark}
As a matter of nomenclature, we will refer to rational homology spheres admitting a metric with no irreducible solutions as \textit{minimal $L$-spaces}.
\end{remark}
As mentioned above, the proof of Theorem \ref{Thm1} exploits in an essential way the fact that the underlying manifold is equipped with a hyperbolic metric. While we do not know a direct way to relate the latter to the Seiberg-Witten equations, we use as stepping stone the spectral geometry of the Hodge Laplacian acting on coexact $1$-forms. Recall that on a Riemannian $3$-manifold with $b_1(Y)=0$ the Hodge decomposition implies the direct sum decomposition of $1$-forms
\begin{equation*}
\Omega^1(Y)=d\Omega^0(Y)\oplus d^* \Omega^2(Y)
\end{equation*}
into exact and coexact ones; the Hodge Laplacian $\Delta=(d+d^*)^2$ preserves the latter decomposition. We denote the spectrum of $\Delta$ acting on coexact $1$-forms $d^* \Omega^2(Y)$ by
\begin{equation*}
0<\lambda_1^*\leq \lambda_2^*\leq \lambda_3^*\leq \dots.
\end{equation*}
In the present paper, we will be mostly interested in the first eigenvalue $\lambda_1^*$. While not much is known in general about this quantity, it has recently attracted attention due to its relationship with a deep conjecture of Bergeron and Venkatesh \cite{BV} regarding the growth of torsion in the cohomology of arithmetic hyperbolic three-manifolds under towers of coverings. In our setting, its appearance as a stepping stone is synthesized in the following diagram. 

\begin{center}
\begin{tikzpicture}[auto,
    block_center/.style ={rectangle, draw=black, thick, fill=white,
      text width=11em, text centered,
      minimum height=4em},
      block_noborder/.style ={rectangle, draw=none, thick, fill=none,
      text width=11em, text centered, minimum height=1em},
      line/.style ={draw, thick, -latex', shorten >=0pt}]]
\matrix [column sep=5mm,row sep=3mm] {
      \node [block_center] (SW) {Non-existence of irreducible solutions to the Seiberg-Witten equations};
      && \node [block_center] (hyp) {Hyperbolic geometry};\\
      \node [block_noborder] (est) {lower bounds};&&\node [block_noborder] (rep) {representation theory};\\
      &\node [block_center] (spec) {First eigenvalue on coexact $1$-forms $\lambda_1^*$};\\
};

\begin{scope}[every path/.style=line]
\path (spec)   -| (est);
\path (rep)   |- (spec);
\path (est)   -- (SW);
\path (hyp)   -- (rep);
\end{scope}

\end{tikzpicture}
\end{center}
\begin{remark}It is an interesting question whether there is any relation between Floer homology and the spectrum of the Laplacian on functions (or, equivalently, the spectrum on exact $1$-forms), as the latter is a much better understood quantity. While the appearance of $\lambda_1^*$ in our context stems naturally from the geometry of the Seiberg-Witten equations, we do not know of any interesting relation between the latter and the spectrum of the Laplacian of functions.
\end{remark}

Before discussing the relationship between $\lambda_1^\ast$ and the Seiberg-Witten equations, let us point out another of its applications.
\begin{teor}\label{Thm2}
For the hyperbolic three-manifolds from the Hodgson-Weeks census listed in Table \ref{table2}, the precise numerical bounds for $\lambda_1^*$ enumerated in that table hold true.  
\end{teor}
\begin{remark}
One should compare these computations with the case of the first eigenvalue on functions. While there are some numerical results in the latter case (especially in the astrophysics literature, see \cite{Ino}, \cite{CS}), these are based on heuristic computations. As we will see, a key input in the proof of Theorem \ref{Thm2} is given by the computations of the topological invariants arising from Seiberg-Witten theory.
\end{remark}
\begin{table}[h!]
 \begin{tabular}{|c| c| c| c| c|} 
 \hline
 Census label & $\lambda_1^\ast$ lower bound & $\lambda_1^\ast$ upper bound & Volume & Injectivity radius \\ 
 \hline\hline
 1 & 0.33749 & 0.33983 & 0.98136\ldots & 0.28904\ldots \\ 
 \hline
 4 & 0.61613 & 0.64594 & 1.28448\ldots & 0.24015\ldots  \\
 \hline
 6 & 0.58541  & 0.60133  &  1.41406\ldots & 0.39706\ldots \\
\hline
 7 & 0.27882 & 0.28224 & 1.41406\ldots & 0.18244\ldots \\
 \hline
 9 & 0.43598 & 0.97651 & 1.44070\ldots& 0.18076\ldots  \\
\hline
 19 & 0.68344 & 0.82304  & 1.75712\ldots & 0.35268\ldots \\
\hline
 23 & 0.50310 & 0.51433  & 1.83193\ldots & 0.24060\ldots \\
\hline
 24 & 0.31571 & 0.32022 & 1.83193\ldots & 0.26531\ldots \\
\hline
34  & 0.00131 & 0.00537 & 1.91221\ldots & 0.24958\ldots \\
\hline
45  & 0.60516 & 0.76929 & 2.02988\ldots & 0.27176\ldots \\
\hline
47  & 0.37043 & 0.38036  & 2.02988\ldots & 0.21563\ldots \\
\hline
48  & 0.28543 & 0.29030  & 2.02988\ldots & 0.27176\ldots \\
\hline
\end{tabular}
\caption{Bounds for $\lambda_1^\ast(Y)$ for the hyperbolic manifolds $Y$ of Theorem \ref{Thm2}.}
\label{table2}
\end{table}

The passage from lower bounds on $\lambda_1^\ast$ to the Seiberg-Witten equations relies upon the following refinement of the main theorem of \cite{Lin} as one key input.  See also \S \ref{SW} for a more detailed discussion.

\begin{teor}\label{Thm3}
Let $Y$ be rational homology three-sphere equipped with a hyperbolic metric. If $\lambda_1^*> 2$, then for any spin$^c$ structure the Seiberg-Witten equations (for sufficiently small perturbations) do not admit irreducible solutions.
\end{teor}
On the other hand, the relationship between hyperbolic geometry and spectral geometry is provided by a specialization of the celebrated Selberg trace formula, which provides, for a Lie group $G$ and a lattice $\Gamma$ in it, a link between geometry and spectral theory (which can be thought as a non-abelian generalization of the classical Poisson summation formula). For simplicity, consider first the more familiar context of a closed hyperbolic surface $X,$ which corresponds to $G=\mathrm{PSL}_2(\mathbb{R})$ and $\Gamma=\pi_1(X)$. In this case, it was proven by Selberg (see \cite{Hej}) that once we label the eigenvalues with repetitions
\begin{equation*}
0 =: \lambda_0 <\lambda_1\leq\lambda_2\leq\lambda_3\leq\dots
\end{equation*}
of the Laplacian $\Delta$ acting on functions on $X$ as $\lambda_j=r_j^2+1/4$ with $r_j\in \mathbb{R}_{\geq0}\cup [0,1/2]\sqrt{-1}$, the following identity holds for every $g\in{C}^{\infty}_c(\mathbb{R})^{\mathrm{even}}$ :
\begin{equation}\label{tracesurface}
\sum_{j=0}^{\infty}\widehat{g}(r_j)=\frac{\mathrm{vol}(X)}{4\pi}\int_{-\infty}^{\infty}\widehat{g}(r)r\tanh(\pi r)dr+\sum_{\gamma\neq 1}\frac{\ell(\gamma_0)}{2 \sinh \left( \frac{\ell(\gamma)}{2} \right)} g(\ell(\gamma)).
\end{equation}
Here, $\widehat{g}(t) := \int_\mathbb{R} g(x) e^{-ix \cdot t} dx$ is the Fourier transform of $g$, and the sum on the right hand side runs over all non-trivial closed geodesics $\gamma$ on $X$. These correspond to non-trivial conjugacy classes in $\pi_1(X)$.  We denote by $\ell(\gamma)$ the length of $\gamma$, and $\gamma_0$ is the primitive geodesic of which $\gamma$ is a multiple. 

\begin{remark}\label{onequart}
The value $1/4$ plays a key role in the spectral theory of hyperbolic surfaces, because it is the bottom of the $L^2$-spectrum of the Laplacian on $\mathbb{H}^2$, see \cite{McKeanquarter} for a nice proof. Eigenvalues less than $1/4$ (i.e. for which the parameter $r_j$ is imaginary) are called \textit{small} \cite[Chapter 8]{Bus}, and are the protagonist of the famous Selberg $1/4$ conjecture \cite{SarSel}.
\end{remark}

The Selberg trace formula is a very powerful tool as it allows to extract seemingly inaccessible information regarding spectral geometry of $X$ via the understanding of the lengths of its geodesics; the latter quantities are directly computable from the traces of the elements $\pi_1(X)\subset \mathrm{PSL}_2(\mathbb{R})$.
\\
\par
In the present paper, we will derive a specialization of the general Selberg trace formula that relates, for a closed oriented hyperbolic three-manifold $Y,$ the following quantities:
\begin{itemize}
\item on the spectral side, the square roots of the eigenvalues of the Laplacian (with repetitions) on coexact $1$-forms $t_j=\sqrt{\lambda_j^*}$;
\item on the geometric side, the volume $\mathrm{vol}(Y)$ and the complex lengths $\mathbb{C}\ell(\gamma)$ of the closed geodesics $\gamma$ of $Y$.
\end{itemize}
Recall here that for a closed geodesic $\gamma$ in a hyperbolic three-manifold there is a notion of \textit{holonomy} $\mathrm{hol}(\gamma)$, namely how an orthonormal framing for the normal bundle of $\gamma$ is rotated under parallel transport along $\gamma$. The complex length of $\gamma$ is then given by
\begin{equation*}
\mathbb{C}\ell(\gamma) := \ell(\gamma) + i \; \mathrm{hol}(\gamma) \in \mathbb{R} + i \; (\mathbb{R} / 2\pi \mathbb{Z}).
\end{equation*}
As in the case of surfaces, these are directly computable in term of the traces of elements $\pi_1(Y)\subset \mathrm{PSL}_2(\mathbb{C})$. The formula is then the following.
\begin{teor}[Explicit Selberg trace formula for coexact 1-forms on closed hyperbolic 3-manifolds]\label{Thm4}
Let $Y$ be a closed oriented hyperbolic three-manifold, and let $H$ be an even, smooth, compactly supported, $\R$-valued function on $\R.$  Then the following identity holds:
\begin{align*}
\left( \frac{1}{2} b_1(Y) -\frac{1}{2} \right) \widehat{H}(0)+\frac{1}{2}\sum_{j=0}^{\infty}\widehat{H}(t_j) &= \frac{\mathrm{vol}(Y)}{2\pi}\cdot \left(H(0)-H''(0) \right)\\
&+\sum_{[\gamma]\neq 1}\ell(\gamma_0) \cdot \frac{\mathrm{cos}(\mathrm{hol}(\gamma))}{|1-e^{\mathbb{C}\ell(\gamma)}| \cdot |1-e^{-\mathbb{C}\ell(\gamma)}|}H \left( \ell(\gamma) \right),
\end{align*}
where $\widehat{H}(t) := \int_\mathbb{R} H(x) e^{-ix \cdot t} dx$ is the Fourier transform of $H.$
In fact, the identity also holds under weaker assumptions on $H$ (see Theorem \ref{geometrictraceformulallimitnew} for a more precise statement).  
\end{teor}
\begin{remark}\label{literature}While we could find other specializations of the Selberg trace formula to differential forms on hyperbolic three-manifolds in the literature (see \cite[Theorem 2]{Fried} and \cite[Chapter II, \S 1]{Millson}), none of them were sufficiently explicit for our purposes. We refer to Section \ref{traceformulanew} for a more detailed discussion.
\end{remark}

Taking as input computations of volume and length spectrum of $Y$ provided by SnapPy \cite{CDGW}, this formula can be used to show that for a given value $t\in\mathbb{R}_{\geq 0}$, $t^2$ is \textit{not} an eigenvalue of the Laplacian on coexact $1$-forms on $Y$. The specific procedure we use, inspired by the work of Booker and Strombergsson related to Selberg's $\frac{1}{4}$-conjecture \cite{BS}, is discussed thoroughly in \S \ref{bookermethod}. Granted this, let us discuss the logic behind the proof of our main results:
\begin{itemize}
\item 
For the spaces in Theorem \ref{Thm1}, we will use the Selberg trace formula to show that $t^2$ is not a coexact, 1-form eigenvalue for any $t^2 \in [0,2].$  Combined with Theorem \ref{Thm3}, this implies that there are no irreducible solutions to the Seiberg-Witten equations.
\item 
For the spaces in Theorem \ref{Thm2}, it is known that their reduced Floer homology $\mathit{HM}$ is non-vanishing. This implies that for an arbitrarily small regular perturbation, the Seiberg-Witten equations admit irreducible solutions, so that $\lambda_1^*\in(0,2]$. On the other hand, using Theorem \ref{Thm4} one can give in these examples a precise constraint on which elements in $[0,\sqrt{2}]$ can possibly be square roots of eigenvalues.  Combining with the existence result from Seiberg-Witten theory implies the precise bounds in Table \ref{table2}.
\end{itemize}

We conclude this introduction by discussing the two simplest examples in which our main results apply, see Figure \ref{diagrams}. \begin{figure}
  \centering
\def\svgwidth{0.8\textwidth}
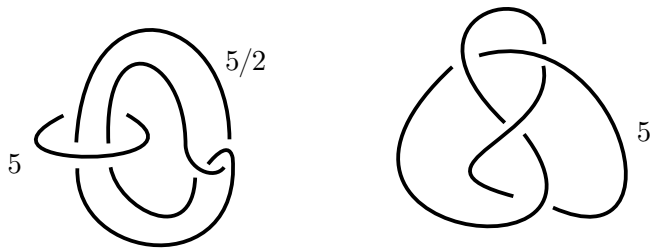
    \caption{Surgery diagrams for the Weeks and Meyerhoff manifolds, which are respectively the manifolds labeled $0$ and $1$ in the Hodgson-Weeks census. The link on the left is the Whitehead link, while the knot on the right is the figure-eight knot.}
    \label{diagrams}
\end{figure} 
The manifolds in the picture represent the ones labeled $0$ and $1$ in the Hodgson-Weeks census.  Both of these manifolds play a special role in hyperbolic geometry; the manifold on the left, the \textit{Weeks manifold}, is known to have the smallest volume $0.94...$ among closed, orientable hyperbolic three-manifolds \cite{GMM}, while the one on the right, the \textit{Meyerhoff manifold}, which has volume $0.98...$ was believed to have smallest volume for a long time. Using the surgery diagrams in Figure \ref{diagrams}, one can determine their Floer homology $\mathit{HM}$ and show, in particular, that the Weeks manifold is an $L$-space while the Meyerhoff manifold is not. Such a drastic difference is not reflected in basic quantities that are studied in hyperbolic geometry, as for example these manifolds have very similar volume and injectivity radius. On the other hand, these manifolds are drastically different from the point of view of the spectral geometry of coexact $1$-forms, as for the Weeks manifold $\lambda_1^*>8.9$, while for the Meyerhoff manifold $\lambda_1^*\sim 0.33$. We will provide a qualitative discussion of this drastic difference, in these and in more general examples, in \S \ref{further}.

\vspace{0.3cm}
\textit{Plan of the paper.} In \S\ref{SW} we provide some background material on monopole Floer homology, and discuss its relation with spectral geometry and in particular Theorem \ref{Thm3}. In \S\ref{traceformulanew} we provide the precise statement of Theorem \ref{Thm4} and discuss its significance and place in the existing literature. In \S\ref{bookersec} we discuss the computational technique of Booker and Strombergsson, and in \S\ref{computations} the outputs are presented. Finally, in \S\ref{further} we discuss the limitations of our method and some natural questions that arise.

\vspace{0.3cm}
\textit{Acknowledgements.} We would like to thank Matthew de Courcy-Ireland for suggesting this collaboration and the IAS for making it possible. We would like to thank Nathan Dunfield for thoroughly answering many of our questions, both regarding SnapPy and his work on $L$-spaces. We would also like to thank Akshay Venkatesh for some helpful conversations, and the anonymous referee thoroughly reading the manuscript and offering many suggestions that greatly helped to improve the exposition. The first author was partially supported by NSF grant DMS-1807242.

\vspace{0.5cm}

\section{The Seiberg-Witten equations and monopole Floer homology}\label{SW}
In this section we review the basic setup of Seiberg-Witten theory on a (closed, oriented, connected) three-manifold $Y$. We refer the reader to \cite{Lin2} for a more thorough introduction and to \cite{KM} for the quintessential reference.
\\
\par
\subsection{The geometric setup. }Consider on $Y$ a Riemannian metric and a spin$^c$ structure $\spin$. For our purposes, the best way to think about the latter is a rank two hermitian bundle $S\rightarrow Y$ together with a bundle map
\begin{equation*}
\rho: TY\rightarrow \mathrm{Hom}(S,S),
\end{equation*}
called Clifford multiplication, satisfying $\rho(v)^2=-|v|^2 1_S$. In coordinates, this means that for any oriented frame $e_1,e_2,e_3$ at a point $y$, we can find a basis of $S_y$ so that $\rho(e_i)$ is the Pauli matrix $\sigma_i$:
\begin{equation*}
\sigma_1=\begin{pmatrix}
i & 0\\
0 & -i
\end{pmatrix}, \quad
\sigma_2=\begin{pmatrix}
0 & -1\\
1 & 0
\end{pmatrix},\quad
\sigma_3=\begin{pmatrix}
0 & i\\
i & 0
\end{pmatrix}.
\end{equation*}
The Clifford multiplication can be extended to $1$-forms via the musical isomorphism, and to forms using the formula
\begin{equation*}
\rho(\alpha\wedge \beta)=\frac{1}{2}\left(\rho(\alpha)\rho(\beta)+(-1)^{|\alpha||\beta|}\rho(\beta)\rho(\alpha)\right).
\end{equation*}
In particular, if $\alpha$ is a $1$-form we have
\begin{equation}\label{cliffast}
\rho(\ast\alpha)=-\rho(\alpha).
\end{equation}
We consider the configuration space $\mathcal{C}(Y,\spin)$ consisting of pairs $(B,\Psi)$ where:
\begin{itemize}
\item $\Psi$ is a \textit{spinor}, i.e. a section of $\Gamma(S)$.
\item $B$ is a \textit{spin$^c$ connection} on $S$, i.e. a unitary connection for which $\rho$ is parallel, or, equivalently
\begin{equation}\label{parallel}
\nabla_B(\rho(X)\Psi)=\rho(\nabla X)\Psi+\rho(X)\nabla_B\Psi
\end{equation}
for any vector field $X$ and spinor $\Psi$. Here $\nabla$ is the Levi-Civita connection and $\nabla_B$ the covariant derivative associated to $B$.
\end{itemize}
The condition (\ref{parallel}) implies that the $\mathrm{SO}_3$-part of a spin$^c$ connection $B$ is determined by the Levi-Civita connection; as a consequence, $B$ is determined by the connection $B^t$ induced on the determinant line bundle $\mathrm{det}(S)$. In particular, the space of spin$^c$ connections is an affine space over $\Omega^1(Y;i\mathbb{R})$. 
\par
The space of configurations $\mathcal{C}(Y,\spin)$ is acted on by the group of automorphisms of the spin$^c$ structure, i.e. the \textit{gauge group} $\mathcal{G}(Y,\spin)=\mathrm{Maps}(Y,S^1)$, via
\begin{equation*}
u\cdot(B,\Psi)=(u^*B,u\cdot\Psi),
\end{equation*}
where $u^\ast B=B-u^{-1}du$ is the pullback connection.
\par
The stabilizer under the gauge group of the configuration $(B,\Psi)$ is trivial when $\Psi$ is not identically zero. On the other hand the stabilizer of a configuration of the form $(B,0)$ is given by the constant gauge transformations, so it is identified with $S^1$. We call the configurations of the first kind \textit{irreducible}, while the configurations of the second kind \textit{reducible}.
\par
For a fixed base connection $B_0$, the \textit{Chern-Simons-Dirac functional}
\begin{equation*}
\mathcal{L}:\mathcal{C}(Y,\spin)\rightarrow \mathbb{R}
\end{equation*}
is defined to be
\begin{equation*}
\mathcal{L}(B,\Psi)=-\frac{1}{8}\int_Y (B^t-B^t_0)\wedge( F_{B^t}+F_{B^t_0})+\frac{1}{2}\int_Y \langle D_B\Psi,\Psi\rangle \mathrm{d}vol.
\end{equation*}
Here $F_{B^t}$ denotes the curvature of the connection $B^t$ (hence a imaginary valued $2$-form), and $D_B$ is the Dirac operator associated to the connection $B$, i.e. the composition
\begin{equation*}
\Gamma(S)\stackrel{\nabla_B}{\longrightarrow}\Gamma( T^*X\otimes S)\stackrel{\rho}{\longrightarrow}\Gamma(S).
\end{equation*}
While the functional is invariant only under the action of the connected component of the gauge group, it descends to a well defined functional 
\begin{equation*}
\mathcal{L}:\mathcal{C}(Y,\spin)/\mathcal{G}(Y,\spin)\rightarrow \mathbb{R}/(2\pi^2\mathbb{Z})
\end{equation*}
on the moduli space of configurations. The critical points of the Chern-Simons-Dirac functional are given by the solutions $(B,\Psi)$ of the system
\begin{align*}
\frac{1}{2}\rho(\ast F_{B^t})+(\Psi\Psi^*)_0&=0\\
D_B\Psi&=0
\end{align*}
to which we refer to as the \textit{Seiberg-Witten equations} on $Y$. Here $(\Psi\Psi^*)_0$ is the traceless part fo $\Psi\Psi^*\in i\mathfrak{su}(S)$; in coordinates, if $\Psi=(\alpha,\beta)$, we have
\begin{equation}\label{quadraticSW}
(\Psi\Psi^*)_0=\begin{bmatrix}
\frac{1}{2}(|\alpha|^2-|\beta|^2) & \alpha\bar{\beta}\\
\bar{\alpha}\beta& \frac{1}{2}(|\beta|^2-|\alpha|^2)
\end{bmatrix}.
\end{equation}
\begin{remark}
The first equation often appears in the literature as $\frac{1}{2}\rho( F_{B^t})-(\Psi\Psi^*)_0=0$; these are equivalent because of Equation (\ref{cliffast}). Our choice makes more apparent the role of the coclosed $1$-form $\ast F_{B^t}$.
\end{remark}

\subsection{Monopole Floer homology and its applications. }One can apply the ideas of Morse homology to the Chern-Simons-Dirac functional $\mathcal{L}$ on the moduli space of configurations in order to define homological invariants of three-manifolds which are topological (i.e. independent of the initial choice of the metric). The final result is a package of invariants called \textit{monopole Floer homology} \cite{KM} (see also \cite{Fro}, \cite{MarWang} for alternative constructions). There are several complications to be handled, most notably the need to introduce a suitable space of \textit{regular} perturbations to the equations in order to achieve transversality, and the $S^1$-symmetry of the functional. In the setup of \cite{KM}, the latter is dealt with suitably blowing up the configurations space, and leads to the construction of $S^1$-equivariant Floer homology group.
\par
The simplest invariant arising from this construction is the \textit{reduced monopole Floer homology group} $\mathit{HM}(Y,\spin)$, which plays a central role when studying gluing formulas for the $4$-dimensional Seiberg-Witten invariants (see the classical reference \cite{Mor} for the latter).  It has also recently gained attention as it contains significant information regarding three-dimensional geometric structures; from this perspective it is convenient to consider the direct sum
\begin{equation*}
\mathit{HM}(Y)=\bigoplus_{\spin}\mathit{HM}(Y,\spin).
\end{equation*}
Objects of central study in three-dimensional geometry are coorientable taut foliations, i.e. coorientable $2$-dimensional foliations $\mathcal{F}$ equipped with a closed $2$-form $\omega$ which is positive on the leaves of $\mathcal{F}$. While criteria for the existence of such foliations has been provided by Gabai for manifolds with $b_1(Y)>0$ \cite{Gabai1}, a general characterization in the case of rational homology spheres is missing. In this sense, the following Floer theoretic obstruction from \cite{KMOS} plays a central role.
\begin{thm}[Theorem $2.1$ of \cite{KMOS}]\label{lspace}
Suppose $Y$ has $b_1(Y)=0$. If it admits a coorientable taut foliation, then $\mathit{HM}(Y)\neq 0$.
\end{thm}
This highlights the class of \textit{$L$-spaces}, i.e. three-manifolds $Y$ satisfying $b_1(Y)=0$ and $\mathit{HM}(Y)= 0.$ This notion corresponds to the analogous notion of $L$-space in Heegaard Floer homology (i.e. spaces for which $HF_{\mathrm{red}}(Y)=0$) via the isomorphism between the theories (see \cite{KLT}, \cite{CGH}, and subsequent papers). 
In fact, it was conjectured by Ozsv\'ath and Szab\' o that, under the assumption that $Y$ is irreducible, the converse of Theorem \ref{lspace} holds. Furthermore, the concepts of $L$-spaces and taut foliations are also conjecturally related to the existence of left-invariant orders on the fundamental group of $Y$ \cite{BGW}. Such conjectures have been proved in the case of graph manifolds \cite{HRRW}, and have been verified in some families of hyperbolic three-manifolds \cite{Dun}.
\par
Even though the definition of the invariant $\mathit{HM}(Y)$ involves the solution of certain non-linear PDEs, its computation can be carried over in several cases (including those in Figure \ref{diagrams}) using topological techniques, most notably the surgery exact triangle \cite{KMOS}. It can also be computed in a (practically infeasible) purely combinatorial fashion \cite{SW}.
\\
\par
\subsection{Relation with spectral geometry.} 
We will focus from now on the case of a rational homology sphere $Y$. If $Y$ admits a metric such that for all spin$^c$ structures (suitable small perturbations of) the Seiberg-Witten equations do not admit \textit{irreducible} solutions, then $\mathit{HM}(Y)= 0$. On the other hand, very little is known in general about the set of solutions to the Seiberg-Witten equations itself other than on manifolds which have positive scalar curvature or are flat (see \cite{KM}). In this case, one can show that the equations do not admit irreducible solutions for suitable small perturbations by means of a Bochner type argument involving the Weitzenb\"ock formula. The case of Seifert manifolds can be understood if one studies a different set of equations where the Levi-Civita connection is replaced by a non standard reducible one \cite{MOY}. As a refinement of argument in the first case, we will now discuss the following.
\begin{thm}\label{spectralnew}
Let $Y$ be a rational homology sphere equipped with a Riemannian metric $g$. Let $\lambda_1^*$ be the least eigenvalue of the Laplacian on coexact $1$-forms, and $\tilde{s}(p)$ the sum of the two least eigenvalues of the Ricci curvature at $p$. If $\lambda_1^*>-\mathrm{inf}_{p\in Y} \tilde{s}(p)/2$ then for all spin$^c$ structures on $(Y,g)$ the Seiberg-Witten equations (for sufficiently small perturbations) have no irreducible solutions.
\end{thm}
Theorem \ref{spectralnew} is a slight refinement of the main result of \cite{Lin}, for which the stronger assumption $\lambda_1^*>-\mathrm{inf}_{p\in Y} \tilde{s}(p)$ is required. As for a hyperbolic metric $\tilde{s}=-4$ everywhere, Theorem \ref{Thm3} follows. 

While there are qualitative results on the behavior of $\lambda_1^*$ for hyperbolic three-manifolds \cite{Jam},\cite{McG}, the goal of this paper is to find examples of hyperbolic three-manifolds for which the explicit bound $\lambda_1^*>2$ holds. In fact, the slight improvement on the main theorem from \cite{Lin} provided by the inequality in Theorem \ref{spectralnew} will be crucial for drawing conclusions in many of the examples of Theorem \ref{Thm1}.
\\
\par
The main theorem of \cite{Lin} uses, at one important step, the inequality 
\begin{equation}\label{weak}
|\nabla\xi|^2\leq |\Psi|^2|\nabla_B\Psi|^2 \text{ for } \xi=\rho^{-1}(\Psi\Psi^*)_0;
\end{equation}
this holds for any configuration $(B,\Psi),$ not necessarily solving the Seiberg-Witten equations.  The key observation behind the improvement in Theorem \ref{spectralnew} is the following refinement for a configuration $(B,\Psi)$ which \emph{does} solve the Seiberg-Witten equations.  
\begin{prop}
Let $(B,\Psi)$ be a solution to the Seiberg-Witten equations, and $\xi=\rho^{-1}(\Psi\Psi^*)_0.$  Then the pointwise identity
\begin{equation*}\label{equal}
|\nabla\xi|^2+|d\xi|^2= |\Psi|^2|\nabla_B\Psi|^2
\end{equation*}
holds.
\end{prop}
\begin{proof}
We will prove this identity at a point $p$ by a computation in coordinates. Fix an oriented orthonormal frame $e_1,e_2,e_3$ which we can assume to be syncronous at $p$ (i.e. $\nabla_{e_i}e_j(p)=0$), and consider a basis of the spinor bundle $S$ for which $\rho$ is represented by the Pauli matrices. We will write in this basis $\Psi=(\alpha,\beta)$. Locally, the covariant derivative $\nabla_B$ is obtained from the spin covariant derivative by adding an imaginary valued one form, see Equation ($3.2$) in \cite{Mor}; as the frame is syncronous at $p$, the spin covariant derivative on $S$ is just the standard derivative $\nabla$ at $p$, and after a gauge transformation, we can assume that the covariant derivative $\nabla_B$ on $S$ is also the standard derivative $\nabla$ at the point $p$. We will therefore write $\nabla_B\Psi=(\nabla\alpha,\nabla\beta)$ at $p$ as
\begin{equation*}
\nabla\alpha=\sum (\nabla_i\alpha) e^i,\quad \nabla\beta=\sum (\nabla_i\beta) e^i.
\end{equation*}
From Equation (\ref{quadraticSW}), we have
\begin{equation}\label{xi}
\xi=-i\left(\frac{1}{2}(|\alpha|^2-|\beta|^2)e^1-\mathrm{Im}(\bar{\alpha}\beta)e^2+\mathrm{Re}(\bar{\alpha}\beta)e^3\right).
\end{equation}
We have
\begin{equation}\label{finalid}
|\Psi|^2|\nabla_B\Psi|^2=(|\alpha|^2+|\beta|^2)\cdot(|\nabla\alpha|^2+|\nabla\beta|^2).
\end{equation}
We claim that the following identity holds:
\begin{equation}\label{claimedid}
|d\xi|^2=|\mathrm{Im}(\bar{\alpha}\nabla\alpha+\bar{\beta}\nabla\beta)|^2=|\mathrm{Im}(\bar{\alpha}\nabla\alpha-{\beta}\overline{\nabla\beta})|^2.
\end{equation}
In fact, we have from equation (\ref{xi}) that at $p$ the identity
\begin{equation*}
|\nabla\xi|^2=|\mathrm{Re}(\bar{\alpha}\nabla\alpha-{\beta}\overline{\nabla\beta})|^2+|\mathrm{Im}(\bar{\alpha}\nabla\beta+(\nabla\bar{\alpha})\beta)|^2+|\mathrm{Re}(\bar{\alpha}\nabla\beta+(\nabla\bar{\alpha})\beta)|^2,
\end{equation*}
holds, where we used that the framing is syncronous. Therefore
\begin{equation*}
|\nabla\xi|^2+|d\xi|^2=|\bar{\alpha}\nabla\alpha-{\beta}\overline{\nabla\beta}|^2+|\bar{\alpha}\nabla\beta+(\overline{\nabla{\alpha}})\beta|^2.
\end{equation*}
Expanding, the mixed terms cancel out and we are left with (\ref{finalid}).
\\
\par
We now show that (\ref{claimedid}) holds.
The Dirac equation $D_B\Psi=0$ is equivalent to the system
\begin{align*}
i\nabla_1\alpha-\nabla_2\beta+i\nabla_3\beta&=0\\
-i\nabla_1\beta+\nabla_2\alpha+i\nabla_3\alpha&=0.
\end{align*}
Using the identity
\begin{equation*}
d\xi=\sum_i e^i\wedge \nabla_{e_i}\xi,
\end{equation*}
we can compute the $-ie^1\wedge e^2$ component of $d\xi$ (recalling that $\mathrm{Re}(iz)=-\mathrm{Im}(z)$) evaluated at $p$ (using again that the framing is syncronous) as follows:
\begin{align*}
=&-\mathrm{Re}((\nabla_2\alpha)\bar{\alpha}-(\nabla_2\beta)\bar{\beta})-\mathrm{Im}((\nabla_1\bar{\alpha})\beta+\bar{\alpha}(\nabla_1\beta))\\
=&-\mathrm{Re}((\nabla_2\alpha)\bar{\alpha}-(\nabla_2\beta)\bar{\beta})+\mathrm{Re}(i({\nabla_1{\bar{\alpha}}})\beta+i\bar{\alpha}(\nabla_1\beta))\\
=&-\mathrm{Re}((\nabla_2\alpha)\bar{\alpha}-(\nabla_2\beta)\bar{\beta})+\mathrm{Re}(-i(\nabla_1\alpha)\bar{\beta}+i\bar{\alpha}(\nabla_1\beta))\\
=&+\mathrm{Re}(\bar{\alpha}(-(\nabla_2\alpha)+i(\nabla_1\beta)))+\mathrm{Re}(\bar{\beta}((\nabla_2\beta)-i(\nabla_1\alpha)))\\
=&+\mathrm{Re}(\bar{\alpha}i(\nabla_3\alpha)))+\mathrm{Re}(\bar{\beta}i(\nabla_3\beta))\\
=&-\mathrm{Im}((\nabla_3\alpha)\bar{\alpha}+(\nabla_3\beta)\bar{\beta})
\end{align*}
where we used the Dirac equation and the fact that for the standard derivative $\nabla_1{\bar{\alpha}}=\overline{\nabla_1\alpha}$. Hence the $ie^3$ component of $\ast d \xi$ is $\mathrm{Im}((\nabla_3\alpha)\bar{\alpha}+(\nabla_3\beta)\bar{\beta})$. The computation for the remaining two components is analogous, and identity (\ref{claimedid}) follows.\end{proof}

\vspace{0.3cm}
Finally, we can discuss how the refinement of Theorem \ref{spectralnew} works in light of this estimate

\begin{proof}[Proof of Theorem \ref{spectralnew}]
Let us quickly review the proof of the slightly weaker inequality of \cite{Lin} (we refer the reader to the paper for more details). We assume for simplicity of notation that the metric is hyperbolic, so that the Ricci curvature is constantly $-2$; in particular, we will prove the statement of Theorem \ref{Thm3}. Given a solution $(B,\Psi)$ to the Seiberg-Witten equations, the Weitzenb\"ock formula and the equation involving the curvature $F_{B^t}$ imply the identity
\begin{equation*}
\Delta|\Psi|^2=2\langle \Psi, \nabla_B^*\nabla_B \Psi\rangle -2 |\nabla_B\Psi |^2=-|\Psi|^4+3|\Psi|^2-2|\nabla_B\Psi|^2.
\end{equation*}
Multiplying this by $|\Psi|^2$, and integrating over the manifold, we obtain by Green's identity
\begin{equation}\label{green}
\int |\Psi|^6-3|\Psi|^4+2|\Psi|^2|\nabla_B\Psi|^2=-\int |\Psi|^2\Delta|\Psi|^2=-\int \lvert d |\Psi|^2\rvert^2\leq 0.
\end{equation}
The Bochner formula states that on $1$-forms:
\begin{equation*}
(d+d^*)^2=\nabla^*\nabla+\mathrm{Ric}.
\end{equation*}
The curvature $F_{B^t}$ is closed by the Bianchi identity, and therefore our form $\xi=\rho^{-1}(\Psi\Psi^*)_0=-\frac{1}{2}\ast F_{B^t}$ is coclosed. Hence we have
\begin{equation*}
\|\nabla\xi\|^2_{L^2}= \|d\xi\|^2_{L^2}+2\|\xi\|^2_{L^2}\geq (2+\lambda_1^*)\|\xi\|^2_{L^2}
\end{equation*}
where we used the variational definition of $\lambda_1^*$ in the last inequality. Hence, the weak inequality (\ref{weak}) implies
\begin{equation*}
\int|\Psi|^2|\nabla_B\xi|^2\geq  \|\nabla\xi\|^2_{L^2}\geq (2+\lambda_1^*)\|\xi\|^2_{L^2}=\frac{1}{4}(2+\lambda_1^*)\|\Psi\|^4_{L^4},
\end{equation*}
where we used the pointwise identity $|\xi|^2=\frac{1}{4}|\Psi|^4$. Combining this with $(\ref{green})$ we get
\begin{equation*}
\int |\Psi|^6+\frac{1}{2}(\lambda_1^*-4)|\Psi|^4\leq 0
\end{equation*}
so that if $\lambda_1^*\geq4$, $\Psi$ is identically zero, i.e. the Seiberg-Witten equations have no irreducible solutions.
\par
Let us now show how to refine the inequality. Using the identity in Proposition \ref{equal} we obtain
\begin{align*}
\int|\Psi|^2|\nabla_B\xi|^2&= \|d\xi\|^2_{L^2}+ \|\nabla\xi\|^2_{L^2}=\\
&=2\|d\xi\|^2_{L^2}+2\|\xi\|^2\geq (2+2\lambda_1^*)\|\xi\|^2_{L^2}=\frac{1}{2}(\lambda_1^*+1)\|\Psi\|^4_{L^4}.
\end{align*}
Combining this with (\ref{green}), we see that the inequality
\begin{equation*}
\int |\Psi|^6+(\lambda_1^*-2)|\Psi|^4\leq 0
\end{equation*}
holds, so that if $\lambda_1^*\geq 2$, $\Psi$ is identically zero. Finally, under the assumption $\lambda_1^*>2$, the estimates continue to hold when we look at small perturbations of the equations, and the result follows.
\end{proof}

Let us point out that as a direct consequence of our discussion, if $Y$ is a hyperbolic rational homology sphere with $\lambda_1^*>2$, then it is an $L$-space. The converse of this is not true. For example, consider $K$ to be the $(-2,3,7)$-pretzel knot. This is a hyperbolic knot, and it is well known that it admits a lens space (hence $L$-space) surgery \cite{FS1}. In particular, for $n$ large enough the manifold $S^3_n(K)$ obtained by $n$-surgery is an $L$-space \cite{KMOS} and is also hyperbolic by a celebrated result of Thurston. Furthermore, for this family of hyperbolic three-manifolds the diameter goes to infinity while the volume stays bounded above. Then a result of McGowan \cite{McG} implies that $\lambda_1^*(Y_n)$ converges to zero, see also \S \ref{further}.

\vspace{0.3cm}

\vspace{0.5cm}
\section{The trace formula} \label{traceformulanew}

Our basis for estimating $\lambda_1^\ast$ on hyperbolic 3-manifolds is the Selberg trace formula.  Introduced by Selberg \cite{Selberg1}, \cite{Selberg2}, this formula relates geometric data on locally symmetric spaces, e.g. lengths of closed geodesics and their holonomies, to spectral data, e.g. eigenvalues of Laplace operators on forms.  The specialization to hyperbolic 3-manifolds of the trace formula that we will use is the following:

\begin{thm}[Geometric trace formula for coexact 1-forms on hyperbolic 3-manifolds] \label{geometrictraceformulacoexact1formsnew}
Let $H$ be any even, compactly supported, smooth $\R$-valued function on $\R.$  Then the equality
\begin{align*}
&{} \sum_{\lambda^{\ast} = \text{coexact 1-form eigenvalue}} \frac{1}{2} \cdot \widehat{H} \left( \sqrt{\lambda^\ast} \right) + \left( \frac{1}{2} b_1 \left( Y \right) - \frac{1}{2}\right) \widehat{H}(0) \\
&=  \frac{\vol(Y)}{2\pi} \cdot \left(  H(0) - H''(0) \right) \\
&+ \sum_{[\gamma] \neq 1} \ell(\gamma_0) \cdot \left( |1 - e^{\C \ell(\gamma)} | \cdot |1 - e^{-\C \ell(\gamma)}| \right)^{-1} \cdot H(\ell(\gamma)) \cdot \cos( \mathrm{hol}(\gamma))
\end{align*}
holds. In the above formula,
\begin{itemize}
\item
Every eigenvalue $\lambda^\ast$ is summed with multiplicity equal to the dimension of the $\lambda^\ast$-eigenspace for the Laplacian acting on coexact 1-forms on $Y.$

\item
$\widehat{H}(t) := \int_\mathbb{R} H(x) e^{-ix \cdot t} dx$ is the Fourier transform of $H.$  

\item
$[\gamma] \neq 1$ ranges over non-trivial closed geodesics.  

\item
$\gamma_0$ is a \emph{primitive} closed geodesic some multiple of which equals $\gamma.$
\end{itemize}
\end{thm}
We prove this result in Appendix \ref{traceformula}. By a limiting argument described in Appendix \ref{limitargument}, Theorem \ref{geometrictraceformulacoexact1formsnew} can be bootstrapped to a larger space of test functions:

\begin{thm} \label{geometrictraceformulallimitnew}
Let $\delta > 5/2.$  Let $H$ be an even, compactly supported $\R$-valued test function satisfying 
\begin{equation*}
\int_{\mathbb{R}} \left(  \left| \widehat{H}(t) \right|^2 + \left|  \widehat{H}'(t) \right|^2 \right) \left( \sqrt{1 + t^2} \right)^{2\delta} < \infty.
\end{equation*}
Then the trace formula from Theorem \ref{geometrictraceformulacoexact1formsnew} is valid for $H.$  
\end{thm}

Our computations in \S \ref{bookersec} use test functions from the larger space in Theorem \ref{geometrictraceformulallimitnew}.
\begin{remark}\label{notcompsupported}
In fact, one can extend the result to include also functions which are not compactly supported, but decay at infinity fast enough (e.g. a Gaussian). This is a delicate extension of Theorem \ref{geometrictraceformulallimitnew}, and will not be needed for the main results in the paper.
\end{remark}

\subsection{Comments for non-experts in the trace formula}
The trace formula is an unconventional tool in Floer homology, and we have therefore crafted our exposition so that in order to understand the basic methods of this paper, one could treat Theorems \ref{geometrictraceformulacoexact1formsnew} and \ref{geometrictraceformulallimitnew} as black boxes. The most important feature to keep in mind is the following: for nice test functions $H$, these formulas express the eigenvalue spectrum for co-exact $1$-forms, sampled using $\widehat{H},$ in terms of explicitly computable geometric quantities sampled using $H$.
\\
\par
As we expect this paper to be mostly interesting for an audience at the intersection of gauge theory and low-dimensional topology, we provide in Appendix \ref{traceformulaintroduction} an introduction to trace formulas in a manner that we hope will be approachable for our readers. The (somewhat technical) proof of Theorem \ref{geometrictraceformulacoexact1formsnew} can be found in Appendix \ref{traceformula}. We put significant effort into making the exposition suitable for a reader with only some basic background on Lie groups and hyperbolic geometry, and who has read through Appendix \ref{traceformulaintroduction}. Furthermore, we provide at each stage motivation and intuition behind the computations we perform.

\subsection{Comments for experts in the trace formula}  
Some complications made it impossible to simply refer to the literature for Theorem \ref{geometrictraceformulacoexact1formsnew}.
\begin{itemize}
\item[(a)]
The literature, e.g. Selberg's original treatments, is strongly biased towards specializations of the trace formula to functions on $\Gamma \backslash G / K$ as opposed to more general functions on $\Gamma \backslash G$ (e.g. differential forms).  

\item[(b)]
Specializations in the literature of the trace formula to hyperbolic 3-manifolds are unfortunately plagued with errors ``in the constants".  Since we calculate the geometric side of the trace formula on a computer, the answers (and their interpretation) would be meaningless if parts of the formula were off by innocent-seeming factors like 2 or $\frac{1}{2\pi}.$    

\item[(c)]
We needed the trace formula in \emph{completely explicit form} to facilitate the computer calculations in \S \ref{computations}, and this required several steps. 
First of all, the Plancherel measure needed to be calculated precisely, as normalized by the standard hyperbolic metric of curvature $-1$. Then, the precise relationship between Casimir eigenvalues and corresponding coexact 1-form eigenvalues for the standard hyperbolic metric needed to be worked out. Finally, there is a contribution from the trivial representation to the spectral side of the trace formula from Theorem \ref{geometrictraceformulacoexact1formsnew}, which surprised us, since our formula is meant to isolate coexact 1-forms.
\end{itemize}

Unfortunately, we found no reference meeting our needs for (a),(b),(c).  In Appendix \ref{traceformula}, we specialize the trace formula to coexact 1-forms on hyperbolic 3-manifolds ourselves, keeping careful track of all constants and checking their consistency with known asymptotic statements that may be derived via the trace formula, e.g. Weyl's law.  If there is any novelty at all in our specialization of the trace formula, it lies in applying the main Theorem of Bouaziz \cite{Bouaziz}; Bouaziz's theorem characterizes, for a semisimple real group $\mathbb{G},$ which functions on the space of (semisimple) conjugacy classes of $G = \mathbb{G}(\mathbb{R})$ may be expressed as orbital integrals of smooth, compactly supported functions on $G.$

\vspace{0.5cm}
\section{Methods for ruling out small eigenvalues}\label{bookersec}

Suppose we have available a trace formula which expresses a spectral sum  
\begin{equation} \label{spectralsum}
\sum_j \widehat{H}(t_j)
\end{equation}
in explicitly computable terms for every nice test function $H.$  For example, the Selberg trace formula for the 0-form spectrum of hyperbolic surfaces (\ref{tracesurface}) has the above form. In that context, $t_j = \sqrt{ \lambda_j - \frac{1}{4}}$ for the eigenvalues of the Laplacian on the hyperbolic surface $\Gamma \backslash \mathbb{H}^2,$ and the trace formula expresses the spectral sum \eqref{spectralsum} in terms of $H$ sampled at lengths of closed geodesics. 
\par
In our case of interest, the trace formula from Theorem \ref{geometrictraceformulacoexact1formsnew} and Theorem \ref{geometrictraceformulallimitnew} for the coexact 1-form spectrum of hyperbolic 3-manifolds has the above form.  In that context, $t_j = \sqrt{\lambda_j^\ast}$ for the eigenvalues $\lambda_j^\ast$ of the Laplacian acting on coexact 1-forms on the hyperbolic 3-manifold $\Gamma \backslash \mathbb{H}^3,$ and the trace formula expresses the spectral sum \eqref{spectralsum} in terms of $H$ sampled at the lengths of closed geodesics (weighted by their holonomy).
Let us denote the list  of $t_j$ (allowing repetition) by $\mathrm{Spec}'$. The following simple observation underlies the most effective method we know for proving that $t \not\in \mathrm{Spec}'$ :

\begin{lemma} \label{exclusioncertificate}
Let $H$ be a nice test function for which the trace formula computing \eqref{spectralsum} applies.  Suppose that $\widehat{H} \geq 0$ and that 
$$\widehat{H}(t) > \sum_j \widehat{H}(t_j).$$
Then $t\not\in \mathrm{Spec}'.$  
\end{lemma}

\begin{proof}
If $t = t_j,$ then $\widehat{H}(t)$ is one summand in the full spectral sum.  Because $\widehat{H} \geq 0,$ it must be less than the full spectral sum.
\end{proof}

Call a test function $H$ \emph{admissible} if $\widehat{H} \geq 0$ and if the trace formula computing \eqref{spectralsum} is valid for the test function $H.$  
Define
\begin{equation} \label{optimizationsolution}
I_{R,t} := \inf_{\substack{\mathrm{supp}(H) \subset [-R,R]\\ \widehat{H} \text{ admissible}\\  \widehat{H}(t) = 1} } \sum_j \widehat{H}(t_j).
\end{equation}
If $I_{R,t} < 1,$ then $t\not\in \mathrm{Spec}'$; a test function which nearly realizes the infimal value $I_{R,t}$ is a witness to that fact that $t$ is not among the $t_j$ by Lemma \ref{exclusioncertificate}.

\vspace{0.3cm}
\subsection{Excluding eigenvalues: the method of Booker and Strombergsson}\label{bookermethod}

While the proofs of the two main results of the paper, Theorems \ref{Thm1} and \ref{Thm2}, both involve proving restrictions on the value of $\sqrt{\lambda_1^\ast(Y)},$ their nature is rather different.  
\begin{itemize}
\item
In Theorem \ref{Thm2}, the entire interest lies in finding a narrow window in $[0,\sqrt{2}]$ in which $\sqrt{\lambda_1^\ast(Y)}$ certifiably lies.   

\item
In Theorem \ref{Thm1}, we need only show that $\sqrt{\lambda_1^\ast(Y)} \notin [0,\sqrt{2}].$  To demonstrate the latter, there is no specific need to find narrow windows in $(\sqrt{2},\infty)$ in which $\lambda_1^\ast(Y)$ certifiably lies.  However, localizing the value of $\sqrt{\lambda_1^\ast(Y)}$ gives independently interesting information about $Y.$  
\end{itemize}

Both problems can be attacked with the method of Booker and Strombergsson \cite{BS}.  But for completeness, we next describe a cruder approach yielding examples for Theorem \ref{Thm1}, i.e $Y$ for which $\sqrt{\lambda_1^\ast(Y)} > \sqrt{2}.$  
\\
\par
To find examples for Theorem \ref{Thm1}, it is natural to apply the trace formula to admissible test functions $H_0$ for which $\widehat{H_0}$ looks like the indicator function $\mathbf{1}_{[-\sqrt{2},\sqrt{2}]}$ (or any $H_0$ for which $\widehat{H}_0$ is large on $[-\sqrt{2},\sqrt{2}]$ and decays quickly to 0); such $H_0$ might allow us to use Lemma \ref{exclusioncertificate} effectively. Regarding the evaluation of $\sum_j \widehat{H_0}(t_j)$ via the trace formula, recall the prime geodesic theorem on closed, hyperbolic 3-manifolds $Y$ \cite{Sarnak}: 
$$\# \{\text{primitive closed geodesics on } Y \text{of length}\leq R    \} \sim \frac{e^{2R}}{2R}.$$ 
To evaluate the test function $H_0$ on that many (complex) lengths and sum them is exponentially difficult in $R.$  For this reason, it is only possible, in practice, to evaluate the spectral side of the trace formula, via the geometric trace formula from Theorem \ref{geometrictraceformulacoexact1formsnew} and Theorem \ref{geometrictraceformulallimitnew} for admissible function $H_0$ supported on $[-R_0,R_0],$ for some relatively small $R_0.$  See \S \ref{computations} for discussion of practical choices for $R_0$. Of course, by the uncertainty principle, restricting the support of $H_0$ makes it difficult to localize $\widehat{H_0}.$
\par
\medskip

We applied the above approach with $H_0(x)=\frac{2}{5}\beta\ast\beta(2x/5)$ where
\begin{equation*}
\beta(x)=
\begin{cases}e^{-1/(1-x^2)}\text{ if }|x|<1\\
0\text{ otherwise,}
\end{cases}
\end{equation*}
is a cutoff function and $\ast$ denotes the convolution. Recall that the convolution of $f$ and $g$ is defined to be
\begin{equation*}
(f\ast g)(x)=\int_{\mathbb{R}}f(t)g(x-t)dt.
\end{equation*}
The key property for our purposes is that the Fourier transform of the convolution is the product of the Fourier transforms, i.e.
\begin{equation*}
\widehat{f\ast g}=\hat{f}\cdot \hat{g},
\end{equation*}
so that in particular the Fourier transform of $H_0$ is a non-negative function. The function $H_0$ is supported in $[-5,5]$, and we accordingly sampled the geodesics in that range. This approach leads to a proof Theorem \ref{Thm1}. This is because for the manifolds in Table \ref{table1} the inequality $\sum_j \widehat{H_0}(t_j)<0.017$ holds, and $\widehat{H_0}(t)\leq0.0176...$ for $t\in[0,\sqrt{2}]$. Furthermore, because smallness of $\sum_j \widehat{H_0}(t_j)$ correlates strongly with largeness of $\sqrt{\lambda_1^\ast}$, the size of the latter spectral sum provides heuristic information about the distribution of $\sqrt{\lambda_1^\ast}$ in our sample of census manifolds; we refer the reader to \S\ref{further}, and in particular Figure \ref{LnonL}, for a more detailed discussion of this.
\\
\par
We emphasize, however, that $I_{R_0,t}$ provides more specific and interesting information about the location of the $t_j$:  
\begin{itemize}
\item $t = \pm t_j$ implies that $I_{R_0,t} \geq 1$
\item
The pointwise limit of $I_{R,t}$ picks out the eigenvalues of the Laplacian on coexact 1-forms.  More precisely,  
\begin{equation} \label{BSlargesupportlimit}
\lim_{R \to \infty} I_{R,t} = \begin{cases} \dim \left(t_j^2 \text{-eigenspace of the Laplacian on coexact $1$-forms}\right) & \text{ if } t = \pm t_j   \\ 0 & \text{otherwise.}  \end{cases}
\end{equation}
\end{itemize}
In particular, one might hope that if ``$Y$ is small relative to $R = R_0,$" e.g. if $\mathrm{inj}(Y)$ is significantly less than $\frac{1}{2} R,$ the function $t \mapsto I_{R_0,t}$ approximates the characteristic function of $\{\sqrt{\lambda_1^\ast}, \sqrt{\lambda_2^\ast}, \ldots \}$ (allowing repetition). 
Furthermore, $t \mapsto I_{R_0,t}$ potentially does better at excluding eigenvalues, via Lemma \ref{exclusioncertificate}, than any fixed admissible function $H_0$ supported on $[-R_0,R_0]$ because 
$$I_{R_0,t} \leq \frac{\sum \widehat{H_0}(t_j)}{\widehat{H_0}(t)}.$$  
We do not know how to compute the function $t \mapsto I_{R_0,t}$ for any $R_0$ on any hyperbolic 3-manifold $Y.$  However, the method of Booker and Strombergsson \cite{BS} finds an upper bound $J_{R_0,t} \geq I_{R_0,t}$ which is explicitly computable via the trace formula.  They applied their method to exclude eigenvalues on (congruence arithmetic) hyperbolic surfaces less than $\frac{1}{4},$ but their method is equally applicable whenever a trace formula is available in the sense of \eqref{spectralsum}.  Their method runs as follows:

\begin{enumerate}
\item Let $h_0,\ldots, h_n$ be even, $\R$-valued functions on $\R$ supported in $[-\frac{R_0}{2}, \frac{R_0}{2}]$ for which $S := \{  h \ast h: h = \sum_{i=0}^n x_i h_i  \text{ for } x_i \in \mathbb{R} \}$ consists entirely of admissible functions for the trace formula \eqref{spectralsum}.  Define 
\begin{align*}
J_{R_0,t} &:= \inf_{H = h \ast h \in S, \widehat{h}(t) = 1} \sum_j \widehat{H}(t_j) \\
&= \inf_{\sum x_i \widehat{h_i}(t) = 1} \sum_j \left( \sum_{i = 0}^n x_i \widehat{h_i}(t)\right)^2 \\
&= \inf_{\sum x_i \widehat{h_i}(t) = 1} \sum_{a,b = 0}^n x_a x_b \sum_j \widehat{h_a}(t_j) \widehat{h_b}(t_j) \\
&= \inf_{ \langle c_t, x \rangle = 1 } \langle Ax,x \rangle,  
\end{align*}
where $\langle \cdot, \cdot \rangle$ denotes the standard dot product on $\mathbb{R}^n, A$ is the matrix with entries
$$A_{a,b} := \sum_j \widehat{h_a \ast h_b} (t_j),$$
and $c_t$ is the vector
$$c_t = \begin{bmatrix} \widehat{h_1}(t) \\ \vdots \\ \widehat{h_n}(t) \end{bmatrix}.$$
\item Clearly we have
$$I_{R_0,t} \leq J_{R_0,t},$$
because $I$ is the infimum of the same quantity over a larger space of functions. 
\medskip

\item $J_{R_0,t}$ is explicitly computable.  It is the minimum of a (positive definite) quadratic form on $\R^n$ subject to a linear constraint.  We calculate by Lagrange multipliers: 
$$J_{R_0,t} = \frac{1}{\langle A^{-1} c_t,c_t \rangle}.$$
The matrix $A$ is explicitly computable using the trace formula \eqref{spectralsum}. 
\end{enumerate}

\vspace{0.5cm}
\section{Computations} \label{computations}

\subsection{Main computation} \label{maincomputation}
We discuss the proof of Theorem \ref{Thm1}, our main result.  We restricted our investigation to manifolds $Y$ from the Hodgson-Weeks census with labels from 0, \ldots, 49.  The Hodgson-Weeks census consists of $11,031$ closed, orientable hyperbolic 3-manifolds which is a good approximation to the (finite) set of hyperbolic three-manifolds with volume at most $6.5$ and injectivity radius at least $0.15.$  Volume increases with the census label, and the first manifold in the list (the Weeks manifold) is known to be the compact hyperbolic three-manifold with the least volume.  Census manifolds include many of the least complex closed, hyperbolic 3-manifolds, and those census manifolds with labels 0, \ldots, 49 are among the least complex of them.  For reasons we will explain in \S \ref{further}, the smallest coexact 1-form eigenvalue $\lambda_1^\ast(Y)$ tends to be small for complex $Y.$  So, we limited our search for $\lambda_1^\ast(Y) > 2$ to the simplest $Y$ we could find.  \emph{Our computations make essential use of the 3-manifold software SnapPy developed by Culler, Dunfield, and Goerner.}  \medskip

We applied the method of \S \ref{bookermethod} to the trace formula from Theorem \ref{geometrictraceformulacoexact1formsnew} (more precisely: the slightly broader version from Theorem \ref{geometrictraceformulallimitnew}).  More specifically, we used the same shape of test functions as in \cite{BS} with the following parameters (notation from \S \ref{bookermethod}):
 
\begin{itemize}
\item
$h := \left( \frac{1}{2\delta} \mathbf{1}_{[-\delta,\delta]} \right)^{\ast 2}.$ 
\item
$h_k(x) := \frac{1}{2} \left( h(x - k\delta) + h(x + k\delta) \right).$  For these choices,
$$\widehat{\sum x_k h_k}(t) = \left( \frac{\sin (\delta t)}{\delta t} \right)^2 \sum x_k \cos(k \delta t).$$ 
\item
$\delta$ satisfying $\delta \cdot (2n+4) \leq R.$
\item
$n = 19.$
\end{itemize}

It is straightforward to check that the functions $h_a \ast h_b$ satisfy the smoothness hypothesis of Theorem \ref{geometrictraceformulallimitnew} and hence are admissible for the trace formula therein.  The test function $h_{x_0,\ldots,x_n} = \left( \sum x_k h_k \right)^{\ast 2}$ is supported on $[-(2n+4)\delta,(2n+4) \delta].$  Hence, the constraint $\delta \cdot (2n+4) \leq R$ guarantees that every $h_{x_0,\ldots,x_n}$ is supported on $[-R,R].$  \medskip

For the closed, hyperbolic 3-manifold $Y,$ the only way we know to compute the matrix $A$ from \S \ref{bookermethod} is to compute $\vol(Y)$ and the full complex length spectrum of $Y$ up to real length $R$ and sample these complex lengths via the test functions $h_a \ast h_b, 0 \leq a,b \leq n+1,$ per the geometric side of the trace formula from Theorem \ref{geometrictraceformulallimitnew}, to recover the spectral side.  Conveniently, SnapPy has built-in functions in the main class Manifold to compute the volume and the complex length spectrum up to a specified real length cutoff. \medskip

\subsubsection{The particular choice of test functions $h_k$}
The test functions to which we apply the trace formula are linear combinations of shifts, by integer multiples of $\delta$ of the convolution fourth power $H := \left( \frac{1}{2\delta} \mathbf{1}_{[-\delta,\delta]} \right)^{\ast 4}.$  Functions of this type are convenient for several reasons:
\begin{itemize}
\item
The function $H$ is a bump function centered at 0 and supported on $[-4\delta,4\delta].$ 

\item
For $\delta$ small, any nice function supported on $[-R,R]$ can be well approximated by such a linear combination of shifts.  Indeed:
\begin{itemize}
\item
The function $H$ is a good approximation to $\delta_0,$ the delta distribution supported at 0.  Thus, for nice functions $f,$ $H \ast f$ is a good approximation to $f.$    

\medskip

\item
Replacing $f$ by a step function $f_{\mathrm{step}}$ ‘interpolating $f$’ with step length $\delta,$ the function
$H \ast f_{\mathrm{step}}$ approximates $H \ast f$ well.  But the function $H \ast f_{\mathrm{step}}$ is a linear combination of shifts, by multiples of $\delta,$ of the function $H.$  
\end{itemize} 
\medskip

As such, one might hope to find a linear combination of shifts of $H$ (by integer multiples of $\delta$) which approximates well the solution to the optimization problem defining $I_{R,t}$ in \eqref{optimizationsolution}, whose domain is the entire space of admissible test functions supported on $[-R,R].$
\medskip

\item
$H$ is an explicit piecewise polynomial function.  So for every value $\ell,$ linear combinations of shifts of $H$ evaluated at $\ell$ are rapidly and accurately computable.  
\medskip

\item
$\widehat{H}(t) = \left(\frac{\sin(\delta t)}{\delta t} \right)^4.$  Shifting by $k \delta$ multiplies the Fourier transform by $e^{ik \delta \cdot t}.$  So the Fourier transforms of linear combinations of shifts of $H$ are rapidly and accurately computable. 
\end{itemize}

One drawback to these test functions, however, is that they are not $C^\infty.$  It is for this reason that we did the extra work to prove Theorem \ref{geometrictraceformulallimitnew}.

\subsubsection{Choosing $R$} \label{choosingR}
Heuristically, we expect 
  
\begin{align*}
H_t &:= h \left( = \sum x_k h_k \right) \text{ minimizing } \sum \widehat{h \ast h}(t_n) \text{ subject to } \widehat{h}(t) = 1 \\
&\approx  \inf_{\text{admissible } H, H(t) = 1, \mathrm{supp}(H) \subset [-R,R]} \sum \widehat{H}(t_n) \\
&=: J_{R,t}.
\end{align*}

Evidently, $J_{R,t}$ decreases with $R.$  So in principle, one would obtain the most useful information by taking $R$ as large as possible.  However, enumerating the complex length spectrum up to real length $R$ is prohibitively computationally intensive even for moderately large $R.$  Indeed, it is known that the number of primitive geodesics of length at most $R$ is approximately $e^{2R}/2R$ \cite{Sarnak}; in practice, the time needed to compute the spectrum seems to be around $Ce^{6R}$ (see Table \ref{timegrowth}).  For practical purposes, $R = 6.5$ seemed to be a reasonable cutoff. For most of the manifolds we tested, this computation took between 20 and 30 minutes  (even though in some special cases, including those of Table \ref{timegrowth}, it took much longer), and we expect the computation for $R=7$ to typically take about a couple of hours.  Of course, this time constraint limits the applicability of our method.

\begin{table}[h!]
 \begin{tabular}{|c| c| c|} 
 \hline
 Cutoff & Census $0$ & Census $1$ \\ 
 \hline\hline
 4.0 & 0.07 & 0.06  \\ 
 \hline
 4.5 & 0.45 & 0.32   \\
 \hline
 5.0 & 4.66  & 3.40 \\
\hline
 5.5 & 86.79 &  62.32 \\
\hline
 6.0 & 1290.23 & 1127.64  \\
\hline
 
\end{tabular}
\caption{The time (in seconds) needed to compute the spectrum at cutoff $R$ for the manifolds Census $0$ and $1$ (on an 3.1 GHz Intel Core i5).}
\label{timegrowth}
\end{table}

\subsubsection{Choosing $n$} \label{choosingn}
For the particular choice of test functions $h_k$ in the previous section, the functions $h_a \ast h_b$ we will use satisfy 
$$(h_a \ast h_b)(x) =  \frac{1}{4} \sum_{\lambda,\mu \in \{ \pm 1 \}} (h \ast h)(x + (\lambda a + \mu b)\delta).$$
To compute $A,$ we calculated the geometric side of the trace formula for the $\sim \mathrm{constant} \cdot n$ different test functions $h \ast h (x + k \delta), k = -2n,\ldots,2n.$  Computing the matrix $A^{-1}$ then requires inverting the $(n+1) \times (n+1)$-matrix $A.$  \medskip

To balance information gained with computational complexity, the specific choice $n = 19$ suited our purposes.

\vspace{0.3cm}
\subsection{Proofs of Theorem \ref{Thm1} and \ref{Thm2}.}
We present the results of our main computation. We computed function $J_{R_0,t},$ described in general terms in \S \ref{bookermethod}, relative to the functions $h,h_k$ described in \S \ref{maincomputation} and the parameters $(R_0,n,\delta) = (6.5,19, \frac{6.5}{2 \cdot 19 + 4})$ explained in \S \ref{choosingR}, \ref{choosingn}.  Recall that if $t^2$ is an eigenvalue of the Laplacian acting on coexact 1-forms on $Y,$ then $J_{R_0,t}(Y) \geq 1.$  
\vspace{0.3cm}
\subsubsection{Proof of Theorem \ref{Thm1} (Examples of hyperbolic, minimal $L$-spaces)}
\begin{proof}
In Table \ref{table3}, we record 
$$\mathrm{PossibleSmallSpectrum}(Y) := \{ t \in [0,4]: J_{R_0,t}(Y) \geq 1 \}$$ 
for several small census manifolds.

For all $Y$ listed in Table \ref{table3}, $\mathrm{PossibleSmallSpectrum}(Y)$ is disjoint from $[0,\sqrt{2}].$  If $\sqrt{\lambda_1^\ast(Y)}$ lies in $[0,\sqrt{2}]$ at all, then it necessarily lies in $\mathrm{PossibleSmallSpectrum}(Y).$  Because \newline $\mathrm{PossibleSmallSpectrum}(Y)$ is disjoint from $[0,\sqrt{2}]$ for all $Y$ tabulated above, it follows that $\sqrt{\lambda_1^\ast(Y)} > \sqrt{2}$ for all $Y$ from Table \ref{table3}. 
\end{proof}

\begin{remark}
For every entry in the above table, it is in principle possible that $\sqrt{\lambda_1^\ast}$ does not belong to $\mathrm{PossibleSmallSpectrum}(Y)$, which would mean that $\sqrt{\lambda_1^\ast} > 4.$  The following heuristic approach suggests that this cannot be the case. We applied the trace formula to test functions of the shape
\begin{equation*}
H_a = \left( \frac{d^2}{dx^2} + a^2 \right) \cdot e^{-x^2/2}
\end{equation*}
for various $0 < a < 4$ (see Remark \ref{notcompsupported}).  The Fourier transform $\widehat{H}$ is a constant multiple of $\left( -t^2 + a^2 \right) e^{-t^2/2}$ and hence is positive if $|t| \leq a$ and negative otherwise.  In particular, if $\sum  \widehat{H_a}(t_n) > 0,$ then $\sqrt{\lambda_1^\ast} \leq a.$ For various $a = a(Y),$ chosen near troughs of the graph of $J_{R_0, t}(Y),$ the approximate value of $\sum_n \widehat{H_a}(t_n),$ as computed via the trace formula from Theorem \ref{geometrictraceformulallimitnew} truncated at $R_0 = 6.5,$ was ``quite positive".  One could estimate the size of the tail (beyond our cutoff $R_0 = 6.5$) to rigorously prove positivity, but we will not attempt to do so here.    
\end{remark}

\begin{table}[h!]
 \begin{tabular}{|c| c| c| c|} 
 \hline
 Census label & Volume & Injectivity radius & $\mathrm{PossibleSmallSpectrum}(Y)$ \\ 
 \hline\hline
 0 & 0.94270 \ldots &  0.29230 \ldots & $[2.962, 3.124]$ \\ 
 \hline
 2 & 1.01494 \ldots & 0.41572 \ldots & $[3.086,3.302] \cup [3.977,4]$ \\
 \hline
 3 & 1.26371 \ldots & 0.28753 \ldots & $[2.145, 2.222] \cup [3.617,4]$ \\
 \hline
 8 & 1.42361 \ldots  & 0.17618 \ldots & $[2.031, 2.263] \cup [3.234,4]$\\
 \hline
 12 & 1.54356\ldots & 0.16768\ldots & $[1.658, 1.686] \cup [2.478,2.778] \cup [3.720,4]$ \\
 \hline
 13 & 1.54356\ldots& 0.28903\ldots & $[1.520,1.672] \cup [2.108,2.213] \cup [3.140,4]$ \\
 \hline
14 & 1.58316 \ldots & 0.27889 \ldots & $[2.018,4]$ \\
 \hline
15 & 1.58316 \ldots &  0.38874 \ldots & $[2.396,2.595] \cup [3.248,4]$ \\
 \hline
16 & 1.58864\dots & 0.26727\ldots & $[1.809, 1.847] \cup [2.519, 3.013] \cup [3.221,4]$ \\
\hline
22 & 1.83193\ldots & 0.26532 \ldots & $[1.680,1.721] \cup [2.48,4]$ \\
\hline
25 & 1.83193 \ldots &  0.26531 \ldots & $[2.323, 2.597] \cup [3.283,4]$\\
\hline
28 & 1.88541\ldots& 0.29230\ldots & $[1.659, 1.689] \cup [2.543,4]$\\
\hline
29 & 1.88541\ldots & 0.19853\ldots & $[1.540,1.934] \cup [2.247, 3.554] \cup [3.951,4]$ \\
\hline
30 & 1.88541\ldots& 0.19853\ldots & $[1.541,1.704] \cup [2.156,4]$\\     
\hline
31 & 1.88541 \ldots & 0.29230 \ldots & $[2.172,3.015] \cup [3.864,4]$ \\
\hline
32 & 1.88591\ldots & 0.20593\ldots & $[1.740,1.794] \cup [2.491,4]$\\
\hline
33 & 1.91084\ldots & 0.22107\ldots & $[1.710,1.799] \cup [2.214, 2.731] \cup [3.012,4]$ \\
 \hline
39 & 1.96273 \ldots & 0.21576 \ldots & $[2.108,2.780] \cup [3.061,4]$\\ 
\hline
40 & 1.96274\ldots& 0.28904\ldots & $[1.842,1.855] \cup [2.829,3.365] \cup [3.634,4]$\\
\hline
42 & 2.02395\ldots& 0.17922\ldots & $[1.779,4]$\\
\hline
44 & 2.02988 \ldots & 0.43127 \ldots &$[2.717,4]$\\
\hline
46 & 2.02988\ldots& 0.27177\ldots & $[1.992,4]$\\
\hline
49 &2.02988\ldots& 0.21564\ldots & $[1.681,1.894] \cup [2.681,4]$\\
\hline
\end{tabular}
\caption{The hyperbolic manifolds of Theorem \ref{Thm1}.}
\label{table3}
\end{table}

\subsubsection{Proof of Theorem \ref{Thm2} (Narrow $\lambda_1^\ast$-intervals for non-L-spaces)}
\begin{proof}
In Table \ref{table4}, we record the same information as in Table \ref{table3} for some small census manifolds previously proven to be non-$L$-spaces. In particular, Dunfield has determined exactly which manifolds in the Hodgson-Weeks census are $L$-spaces in the setting of Heegaard Floer homology; in his approach, many of the spaces in the census are shown to be $L$-spaces via surgery exact triangles, using the fact that they are obtained by Dehn filling on cusped manifolds which admit lens space fillings. More generally, most spaces in the census arise as branched double covers of links in $S^3$, hence their Floer homology can be computed using software developed in the setting of bordered Heegaard Floer homology \cite{Zhan}. Via the isomorphism proved in \cite{KLT}, \cite{CGH}, and the subsequent papers, this also provides a list of which manifolds in the Hodgson-Weeks census are $L$-spaces in the setting of monopole Floer homology.
\begin{table}[h!]
 \begin{tabular}{|c| c| c| c| } 
 \hline
 Census label  & Volume & Injectivity radius & $\mathrm{PossibleSmallSpectrum}(Y)$ \\ 
 \hline\hline
 1 &  0.98136 \ldots & 0.28904 \ldots & $[0.580,0.583] \cup [3.163,4]$ \\ 
 \hline
 4 &  1.28448 \ldots & 0.24015 \ldots & $[0.784,0.804] \cup [2.220,3.403] \cup [3.964,4]$ \\
 \hline
 6  &  1.41406 \ldots & 0.39706 \ldots & $[0.765,0.776] \cup [2.305,3.383]$\\
\hline
 7 &  1.41406 \ldots & 0.18244 \ldots & $[0.528,0.532] \cup [3.346,4]$\\
\hline
9 &1.44069 \ldots & 0.36152 \ldots &  $[0.660,0.988] \cup [3.348, 4]$\\
\hline 
 19 &  1.75712 \ldots & 0.35268 \ldots & $[0.826,0.908] \cup [1.987,4]$\\
\hline
 23 & 1.83193 \ldots & 0.24060 \ldots & $[0.709,0.718] \cup [2.391,2.797] \cup [3.045,4]$\\
\hline
 24 &  1.83193 \ldots & 0.26531 \ldots & $[0.561,0.566] \cup [3.043,4]$\\
\hline
34  &  1.91221 \ldots & 0.24958 \ldots & $[0.036,0.074] \cup [3.049,4]$\\
\hline
45  &  2.02988 \ldots & 0.27176 \ldots & $[0.777,0.878] \cup [1.925,4]$\\
\hline
47  &  2.02988 \ldots & 0.21563 \ldots & $[0.608,0.617] \cup [2.637,4]$\\
\hline
48  &  2.02988 \ldots & 0.27176 \ldots & $[0.534,0.539] \cup [3.121,4]$\\
\hline
\end{tabular}
\caption{Bounds for $\sqrt{\lambda_1^\ast(Y)}$ for the hyperbolic manifolds $Y$ of Theorem \ref{Thm2}.}
\label{table4}
\end{table}

\medskip

As for non $L$-spaces, the Seiberg-Witten equations admit irreducible solutions, so Theorem \ref{Thm3} implies that $\sqrt{\lambda_1^\ast(Y)} \leq \sqrt{2}$ for every non-$L$-space $Y.$  Thus,
$$\sqrt{\lambda_1^\ast(Y)} \in [0,\sqrt{2}] \cap \mathrm{PossibleSmallSpectrum}(Y) \text{ for every non-L-space } Y.$$
In particular, per Table \ref{table4}, $\sqrt{\lambda_1^\ast(\mathrm{Census}_1)}$ belongs to $[0.580,0.583]$, and the analogous conclusion holds for the other entries.
\end{proof}

\vspace{0.3cm}
\subsection{Pictures bounding $\mathrm{PossibleSmallSpectrum}(Y)$}
Recall that $J_{R_0,t}(Y)$ is designed to approximate
$$\lim_{R \to \infty} I_{R,t} = \begin{cases} \dim \left(t_j^2 \text{-eigenspace of the Laplacian on coexact $1$-forms}\right) & \text{ if } t = \pm t_j \\ 0 & \text{otherwise.}  \end{cases}$$
(see \S \ref{bookermethod} for further discussion).  This bears out in pictures.  We include pictures of the graphs of $t \mapsto J_{R_0,t}(\mathrm{Census}_i)$ for $i = 0,1,2.$     

\begin{figure}
  \includegraphics[width=0.8\linewidth]{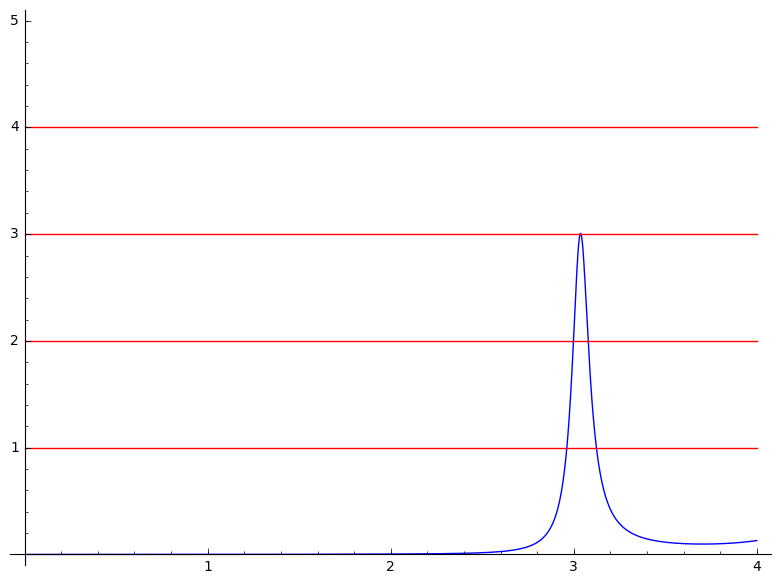}
  \caption{The graph of $t \mapsto J_{R_0,t}(\mathrm{Census}_0)$ for $t \in [0,4].$}
  \label{booker0graph}
\end{figure}

\begin{figure}
  \includegraphics[width=0.8\linewidth]{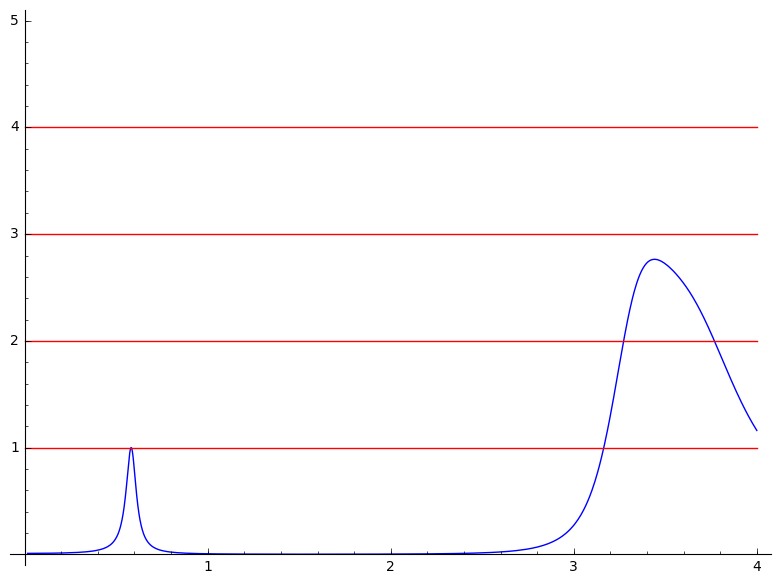}
  \caption{The graph of $t \mapsto J_{R_0,t}(\mathrm{Census}_1)$ for $t \in [0,4].$}
  \label{booker1graph}
\end{figure}

\begin{figure}
  \includegraphics[width=0.8\linewidth]{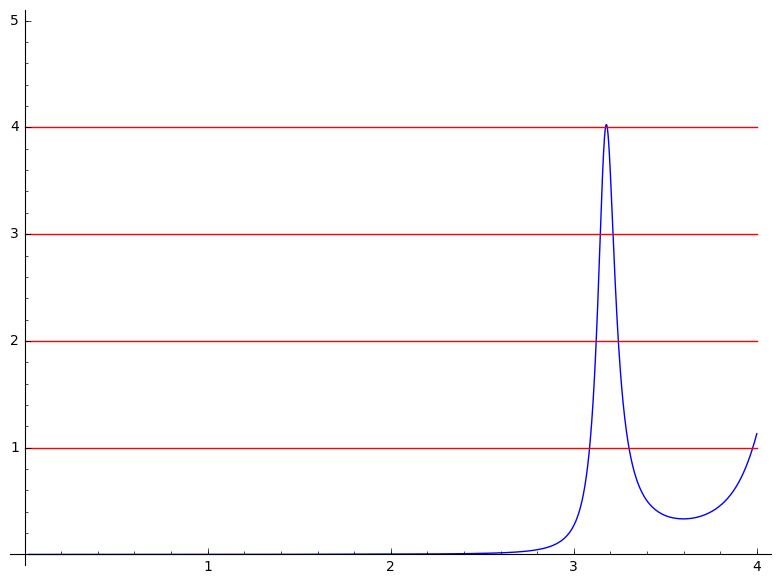}
  \caption{The graph of $t \mapsto J_{R_0,t}(\mathrm{Census}_2)$ for $t \in [0,4].$}
  \label{booker2graph}
\end{figure}

\medskip

In all three pictures, we expect the first peak of the graph to occur near 
$$(\sqrt{\lambda_1^\ast}, \text{dimension of the } \lambda_1^\ast \text{-eigenspace for coexact 1-forms}).$$ 
Indeed, the fact that the vertical coordinate just barely exceeds a positive integer is a non-trivial check on our computations.  To compute the intervals from $\mathrm{PossibleSmallSpectrum}(Y),$ we solved for $J_t(Y) = 1$ (up to tolerance $10^{-6}$) via bisection. \medskip

When the graph of $J_t(Y)$ is peaked just barely above vertical coordinate $m$ for some integer $m > 1,$ the eigenvalue windows are likely much narrower than we claim.  For example, note that  
$$\{ t \in [0,4]: J_t(\mathrm{Census}_0) \geq 3\} \subset [3.036,3.040].$$
So if the $\lambda_1^\ast$-eigenspace for Census 0 really is 3-dimensional, as Figure \ref{booker0graph} suggests, then $\sqrt{\lambda_1^\ast(\mathrm{Census}_0)} \in [3.036,3.040].$  Likewise, if the $\lambda_1^\ast$-eigenspace for Census 2 is actually $4$-dimensional as Figure \ref{booker2graph} suggests, then $\sqrt{\lambda_1^\ast(\mathrm{Census}_2)} \subset [3.177,3.183].$

\begin{remark}
The trace formula is unable to distinguish between two parameters $t_n, t_{n+1}$ which are very close versus equal on the nose.  We do not know, in general or even in the particular examples of Census 0 and Census 2, how to compute the multiplicity of an eigenvalue having multiplicity greater than 1.
\end{remark}

\vspace{0.5cm}
\section{Limitations and further directions}\label{further}
Even though our results can be seen a first step toward understanding the relation between Floer homology and hyperbolic geometry in dimension three, our approach has some significant limitations; we now discuss these and also some natural questions and problems these lead to.
\subsection{Regarding the structure of the collection of hyperbolic minimal $L$-spaces}
While our test was successful when studying small manifolds in the census, it can be seen that as the volume grows, the proportion of manifolds with $\lambda_1^*\leq2$ increases. This should be contrasted with the computations of Dunfield \cite{Dun}, which imply that a very large part of the manifolds in the census are $L$-space. This observation leads the obvious question of whether there are infinitely many manifolds with $\lambda_1^*>2$, or the following more general question:

\begin{quest}
Fix $\epsilon > 0.$  Does the set 
$$S_\epsilon = \{ \text{closed, hyperbolic } Y: H^\ast(Y, \mathbb{Q}) = H^\ast(S^3, \mathbb{Q}) \text{ and } \lambda_1^\ast > \epsilon \}$$ 
have any discernable structure?  In particular, is $S_\epsilon$ always a finite set?  
\end{quest}

While we do not have a completely satisfactory answer to the above question, there are some clear restrictions on the local geometry of the elements in $S_{\epsilon}$. The discussion which follows is inspired by the work of McGowan \cite{McG} (which in fact provides more refined estimates regarding the number of small eigenvalues, provided upper bounds on the volume). \bigskip

\par
Recall that a \textit{hyperbolic tube} $T$ with complex length $\ell e^{i\vartheta_0}$ is obtained by quotienting the cylinder
\begin{equation*}
\{(r,t,\vartheta) \lvert 0\leq r\leq R, 0\leq T\leq \ell,\theta\in S^1\}
\end{equation*}
equipped with the hyperbolic metric
\begin{equation*}
dr^2+\mathrm{cosh}^2r dt^2+\mathrm{sinh}^2r d\vartheta^2
\end{equation*}
via the identification
\begin{equation*}
(r,0,\vartheta)\sim (r,\ell,\vartheta+\vartheta_0).
\end{equation*}
We refer to $R$ as the radius of the tube. The subset $r=0$ is a geodesic called the core geodesic.
Consider now on a tube $T$ of radius $R$ a $1$-form of type $\alpha=f(r)dt$. A form of this kind is always coclosed. Furthermore, we have
\begin{equation*}
d\alpha=f'(r)dr\wedge dt.
\end{equation*}
Now, $|dr\wedge dt|=1/\mathrm{cosh}(r)$. Choosing $f$ to be a standard pyramid shaped function on $[0,R]$, we see that the Rayleigh quotient of $\alpha$ is approximatively
\begin{equation}\label{rayleigh}
\frac{\int_T|d\alpha|^2}{\int_T|\alpha|^2}\approx \frac{\int_0^R |f'|^2dr}{\int_0^R|f|^2dr},
\end{equation}
which converges to zero for $R$ going to infinity. Hence, given $\epsilon>0$, there is a universal upper bound of the diameter of a tube $T\subset Y$ for a hyperbolic rational homology sphere with $\lambda_1^*>\epsilon$. Using this, we have the following.

\begin{prop}
Let $Y$ be a hyperbolic 3-manifold.  There exists $R, \delta >0$ satisfying:
\begin{itemize}
\item
if $Y$ contains an embedded ball of radius $R \geq R_0,$ then $\lambda_1^\ast(Y) \leq \epsilon.$
\item
if $\mathrm{inj}(Y) < \delta,$ then $\lambda_1^\ast(Y) \leq \epsilon.$
\end{itemize}
In particular, $S_{\epsilon}$ is contained in the set of all $Y$ for which the local injectivity radius function has range contained in $[\delta, R].$ 
\end{prop}
\begin{proof}
For the lower bound, we invoke \cite[Theorem 3.2]{GMM1} which says: if there is an embedded geodesic $\gamma$ of real length $\ell$, then $\gamma$ is the core of an embedded tube with radius $r(\ell)$ with $r(\ell)\rightarrow \infty$ as $\ell \rightarrow 0$.  Because $\lambda_1^\ast > \epsilon$ imposes a universal upper bound on the diameter of an embedded tube, the latter implies that $\ell$ is bounded below. Thus, the injectivity radius, which equals half the length of the shortest closed geodesic (\cite{Mar}, Proposition $4.3.2$), must be bounded below too.
\par
To see that there is an upper bound on the local injectivity radius, parametrize the hyperbolic ball of radius $R$ as $(0,R)\times S^2$ equipped with the metric $dr^2+\mathrm{sinh}^2(r)g_{S^2}$, where $g_{S^2}$ is the metric on the unit sphere in $\mathbb{R}^3$. Consider then for a fixed non-zero coclosed $1$-form $\beta$ on $S^2$ the forms of the type $g(r)\beta$. This is a coclosed form, and a computation analogous to (\ref{rayleigh}) shows that its Rayleigh quotient only depends on the Rayleigh quotient of $g$. In particular, when $R$ goes to infinity, this can be made to go to zero.
\end{proof}

\begin{cor}For every $\epsilon,V>0$ there exists only finitely many hyperbolic three-manifolds with $\lambda_1^*>\epsilon$ and $\mathrm{vol}<V$.
\end{cor}
\begin{proof}
This follows directly from the previous proposition combined with the fact that there are only finitely many manifolds with volume bounded above and injectivity radius bounded below \cite{Gro}.  
\end{proof}

One is then led to ask where do the limitations of our approach stem from. Aside from the applicability of the Booker-Strombergsson method to provide effective computations of $\lambda_1^*$, the main problem is that the bound we are using, i.e. $\lambda_1^\ast \leq 2$ when the Seiberg-Witten equations admit irreducible solutions, is rather crude. In particular:
\begin{itemize}
\item it does not use the hyperbolic metric in an essential way. In fact, Theorem \ref{spectralnew} shows that $\lambda_1^\ast(Y) \leq 2$ provided $Y$ is a Riemannian $3$-manifold for which the Seiberg-Witten equations on $Y$ admit irreducible solutions and that $\tilde{s}(Y) = -4$;
\item more importantly, in the proof of Theorem \ref{spectralnew}, we use the estimate $||d\xi||^2_{L^2}\geq \lambda_1^*\|\xi\|^2_{L^2}$. While this holds for \textit{any} coclosed $1$-form $\xi$ on $Y$, one could expect that a sharper estimate holds when $\xi$ arises from a solution to the Seiberg-Witten equations.
\end{itemize}

For example, we just saw that the smallness of $\lambda_1^*$ for manifolds with large embedded balls or short geodesics is caused by $1$-forms of a special kind; it would be interesting to understand if forms of small Rayleigh quotient on a tube or a ball can arise from the solutions to the Seiberg-Witten equations.
More generally, we have the following.
\begin{quest}
Suppose $Y$ is a closed, hyperbolic rational homology sphere.  Can one improve upon the upper bound $\lambda_1^\ast \leq 2$, which holds for all Riemannian 3-manifolds $Y$ satisfying $\widetilde{s}(Y) = -4$, using explicit and computable geometric data arising from the hyperbolic geometry of $Y$ (e.g. the injectivity radius)?  
\end{quest}

In fact, even though our methods are conclusive only in some examples, there seems to be an intriguing correlation between the size of $\lambda_1^*$ and the property of being $L$-spaces (see Figure \ref{LnonL}). A better understanding of this experimental observation could lead to interesting geometric characterizations of hyperbolic $L$-spaces in terms of explicit quantities of interest in hyperbolic geometry.
\begin{figure}
\includegraphics[width=0.8\linewidth]{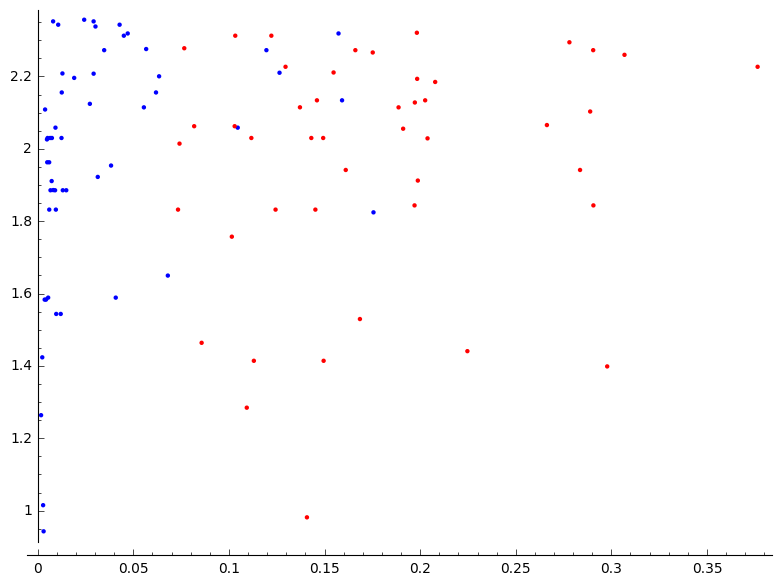}
\caption{We have plotted, among the first $100$ manifolds in the Hodgson-Weeks census, the $L$-spaces in blue and the non-$L$-spaces in red. The $y$-axis records the volume, while the $x$-axis records the value of the spectral sum $\sum \hat{H}(t_j)$ obtained by using $H_0(x)=\frac{2}{5}\beta\ast\beta(2x/5)$ where $\beta(x)=e^{-1/(1-x^2)}$ is a cutoff function (see the discussion of the naive attempt in \S\ref{bookermethod}). The function $H_0$ is supported in $[-5,5]$, and we accordingly need as input the length spectrum with cutoff $R=5$. Heuristically, the graph should be interpreted as follows: a low value of the spectral sum suggests a big value for $\lambda_1^*$; in particular, the manifolds with spectral sum $<0.017$ have $\lambda_1^*>2$.}
  \label{LnonL}
\end{figure}

\vspace{0.3cm}
Regarding the limitations of our methods, the following is also a natural question.
\begin{quest} \label{minimalLspacequestion}
Is there an $L$-space $Y$ which is not a minimal $L$-space? In other words, is there an $L$-space $Y$, such that for each choice of metric, the Seiberg-Witten equations admit irreducible solutions?
\end{quest}

By contrast, the construction of \cite{Fro} shows that there is always a metric for which the equations admit irreducible solutions.

\subsection{Comparison to notions of minimality in other Floer homology theories}
In a different direction, one could try to compare our notion of minimal $L$-space with the analogous ones of \textit{strong $L$-space} in the setting of Heegaard Floer homology \cite{GreLev} and the one of \textit{$\mathrm{SU}_2$-cyclic manifold} in instanton theory (see for example \cite{SivZen} and references therein). These are roughly speaking the spaces for which the relevant Floer chain complex is as simple as possible. As our understanding of minimal $L$-spaces is too limited to even formulate reasonable questions about the relationships between these notions (see Question \ref{minimalLspacequestion}), we will focus here on pointing out some interesting examples.

\subsubsection{Comparison to strong $L$-spaces for Heegaard Floer homology}
It is shown in \cite{Gre} that the branched double cover $\Sigma(L)$ over a non-split alternating link is a strong $L$-space. In \cite{GreLev}, the authors ask whether every strong $L$-space arises in this manner, and they provide evidence towards a positive answer. By contrast, the Weeks manifold, which we have shown to be a minimal $L$-space in Theorem \ref{Thm1}, is not the branched double cover over any alternating knot (as it follows via geometrization from \cite{MedVes}). While the Conjecture in \cite{GreLev} would suggest that the Weeks manifold is not a strong $L$-space, whether this is actually the case is currently an open problem.
\\
\par
Conversely, it is easy to find examples of alternating knots whose branched double cover is hyperbolic and for which our methods strongly suggest that $\lambda_1^*\leq2$ (e.g. $10_{108}$); in particular we cannot determine whether these examples are minimal or not. Nevertheless, browsing through small crossing alternating knots, one can find several examples for which $\lambda_1^*>2$ can be proved using our methods at a cutoff $R=6.5$. For example, the double branched covers of the alternating knots $9_{40}$, $10_{100}$, $10_{102}$, $10_{103}$, $10_{104}$ and $10_{109}$ all satisfy $\lambda_1^\ast > 2$ and are not among the examples covered already in Theorem \ref{Thm1} (as they have larger volume). 
\medskip

As the topology and geometry of an alternating knot is deeply connected with the combinatorics of its alternating diagram (see for example \cite{Lac}) we ask the following.
\begin{quest}
Suppose $L$ is an alternating link in $S^3$ for which the branched double cover is a hyperbolic rational homology sphere. Can one provide explicit lower bounds on $\lambda_1^*$ of $\Sigma(L)$ in terms of an alternating diagram of $L$?
\end{quest}

\subsubsection{Comparison to $\mathrm{SU}_2$-cyclic manifolds for instanton Floer homology}
Regarding the class of $\mathrm{SU}_2$-cyclic manifolds, our understanding is even more limited. Let us point out that the hyperbolic manifolds $\Sigma(8_{18})$ (which is the example labeled $44$ in Table \ref{table1}) and $\Sigma(10_{109})$, which were shown to be $\mathrm{SU}_2$-cyclic in \cite{SivZen}, can be shown to have $\lambda_1^*>2$ using our methods, and are therefore minimal $L$-spaces. 
\medskip

In a different direction, the Weeks manifold $Y$ (which we have shown to be a minimal $L$-space) does admit a non-cyclic (indeed, faithful) $\mathrm{SU}_2$-representation. This fact is well-known to experts, and can be seen directly from the arithmetic description of $\iota:\pi_1(Y){\hookrightarrow} \mathrm{PSL}_2(\mathbb{C})$ given in Section $9.8.2$ of \cite{MacReid} as follows (we refer the reader to Section $8.2$ of \cite{MacReid} for the relevant notions). The inclusion $\iota$ is defined over a cubic field with exactly one complex place; by taking its Galois conjugate corresponding to the real place (at which the corresponding quaternion algebra is ramified), we obtain a new embedding $\pi_1(Y){\hookrightarrow} \mathrm{PSU}_2=\mathrm{SO}_3$. As $H^2(Y,\mathbb{Z}/2\mathbb{Z})=0$, this embedding lifts to an embedding into $\mathrm{SU}_2.$

\subsection{$L$-spaces, integer homology spheres, and a conjecture of  Ozsv\'ath and Szab\'o}
An intriguing conjecture of Ozsv\'ath and Szab\'o states that the only irreducible $L$-spaces which are integral homology spheres are $S^3$ and the Poincar\'e homology sphere. In particular, their conjecture predicts that no hyperbolic integer homology sphere is an $L$-space. While Dunfield has already determined that none of the $150$ integral homology spheres in the Hodgson-Weeks census in an $L$-space, it is still interesting to look at these examples from our perspective. Referring to Figure \ref{LnonL}, the four integral homology spheres (which have census label $5$, $34$, $77$ and $79$) all have spectral sum $>0.25$, which is very large compared to the manifolds in our sample. The case of Census $34$ (see
Figure \ref{booker34graph}) is emblematic: our computations (see Table \ref{table2}) imply that 
\begin{equation*}
0.001<\lambda_1^*<0.005.
\end{equation*}
This value is \textit{extremely small} compared to the other manifolds with similar volume. In light of the conjecture of Ozsv\'ath and Szab\'o, we ask the following.
\begin{quest}
For hyperbolic integer homology spheres $Y,$ what are the best upper bounds one can prove for $\lambda_1^\ast(Y)$ in terms of the hyperbolic geometry of $Y$?  In particular, is it always true that $\lambda_1^\ast(Y) \leq 2$?
\end{quest}

\begin{figure}
  \includegraphics[width=0.8\linewidth]{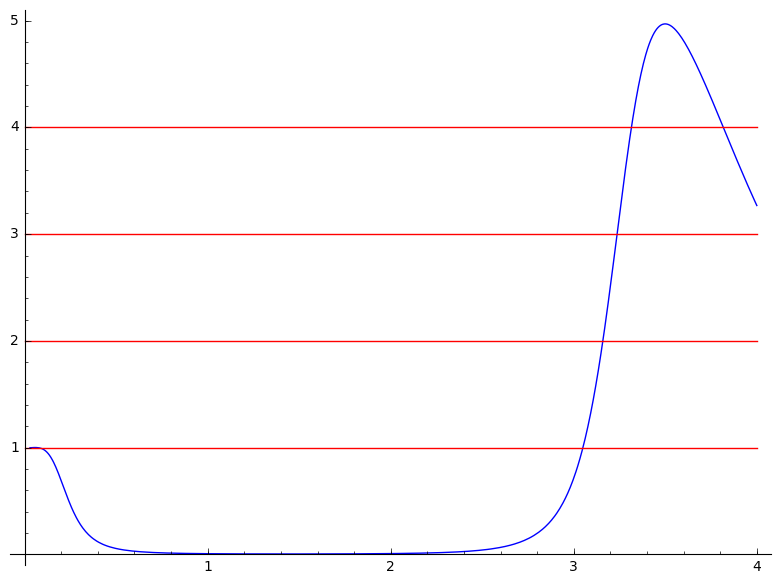}
  \caption{The graph of $t \mapsto J_{R_0,t}(\mathrm{Census}_{34})$ for $t \in [0,4].$}
  \label{booker34graph}
\end{figure}
%
%

\subsection{Isospectrality. }A direct consequence of the Selberg trace formula (when applied to suitable test functions), is that two hyperbolic three-manifolds with the same complex length spectrum have the same eigenvalues of the Laplacian on coexact $1$-forms (and the same volume and $b_1$). A pair of such manifolds is called \textit{length isospectral}; we refer the reader to Section $12.4$ of \cite{MacReid} for a detailed discussion and for some explicit examples. The following is an intriguing question.
\begin{quest}\label{questoiso}Is there a pair of length isospectral hyperbolic rational homology three-spheres for which $\lambda_1^*>2$?
\end{quest}
Unfortunately, all known examples of length isospectral hyperbolic three-manifolds have volume significantly larger than the ones we have considered in this paper \cite{LinVoi}. On the other hand, a positive answer to the question would provide us a pair of non-homeomorphic hyperbolic three-manifolds which, from our perspective, are $L$-spaces for the same geometric reason.
\par
In fact, our techniques produce the same bounds on $\lambda_1^*$ for manifolds with the same volume and complex length spectrum up to a constant $R$, provided we restrict ourselves to test functions $H$ supported in $[-R,R]$. Examples of such pairs of manifolds, where the cutoff $R$ is comparable to the volume and for which admit many geodesics with length at most $R$, are constructed in \cite{FutMil}, and it would be also interesting to find a pairs where both manifolds have $\lambda_1^*>2$ in this more general context.

\vspace{0.5cm}
\begin{appendix}

\section{An introduction to Selberg trace formulas} \label{traceformulaintroduction}

Our goal is to provide an introduction to the main ideas behind the Selberg trace formula, crafted for someone working at the intersection of gauge theory and low-dimensional topology. In particular, we assume only some basic familiarity with the heat kernel on compact manifolds and the trace of operators, as in Chapters $7$ and $8$ of \cite{Roe}.
\begin{remark}\label{convergence}
To streamline the exposition, we will not concern ourselves in Appendices \ref{traceformulaintroduction} and \ref{traceformula} with technical aspects involving convergence and smoothness issues. Of course, we \emph{do} concern ourselves with these issues in Appendix \ref{limitargument}, where we generalize the results to certain classes of non-smooth functions.
\end{remark}
\subsection{The basic idea}
The various incarnations of the Selberg trace formula are obtained by computing the trace of certain convolution operators in two ways, one involving spectral data and the other involving geometric data. In fact, such formulas are far reaching generalizations of the well-known fact that the trace of a matrix can be computed both as the sum of its eigenvalues and as the sum of its diagonal entries.

\subsection{The trace of the heat kernel} \label{heatkernel} Let us first discuss a very specific instance of the trace formula for surfaces (\ref{tracesurface}), due to McKean \cite{McKean}. Given $X=\Gamma \backslash \mathbb{H}^2$ a compact hyperbolic surface, denote by $0=\lambda_0<\lambda_1\leq\lambda_2\leq\dots$ the eigenvalues of the Laplacian on functions. Then, for any $t\in\mathbb{R}^{>0}$, the identity
\begin{equation}\label{mckean}
\sum_{n=0}^{\infty} e^{-\lambda_nt}=\frac{\mathrm{vol}(X)}{4\pi} e^{-t/4} \int_{-\infty}^{\infty} r e^{-t r^2} \tanh(\pi r) \; dr + \frac{1}{(4\pi t)^{1/2}}  e^{-t/4} \cdot \sum_{\gamma \neq 1}\frac{\ell(\gamma_0)}{\mathrm{sinh}(\ell(\gamma)/2)}e^{-\ell^2(\gamma)/4t}
\end{equation}
holds, where we follow the notation specified in the paragraph below Equation \eqref{tracesurface}. In fact, this follows from Equation \eqref{tracesurface} (which also holds for non-compactly supported functions that decay fast enough at infinity) by taking $g(x) = e^{-t/4} \cdot \frac{1}{\sqrt{4\pi t}} \cdot e^{-x^2 / 4t}$, for which $\widehat{g}(r)=e^{-t(r^2+1/4)}.$
\par
The identity \eqref{mckean} is obtained by equating two different computations of the trace of the heat kernel on $X$. We discuss the main points behind the proof, and refer the reader to Chapters $7$ and $9$ of  \cite{Bus} for more details.

\bigskip

Recall that the solution $f_t$ at time $t$ of the heat equation
\begin{align*}
\frac{d}{ds}f(s,y)+\Delta f(s,y)&=0\\
f(0,y)&=f_0(y)
\end{align*}
on $X$ is obtained by taking the convolution of $f_0$ with the heat kernel $K_t(x,y)$ on $X$, i.e.
\begin{equation*}
f_t(y)=\int_X f_{{0}}(x)K_t(y,x) dx.
\end{equation*}
Furthermore, the map $e^{-t\Delta}$ from $L^2(X)$ to itself sending $f_0$ to $f_t$ is trace-class, and we have
\begin{equation*}
\mathrm{trace} (e^{-t\Delta})=\int_X K_t(x,x) dx.
\end{equation*}
We now discuss two explicit expressions for the heat kernel $K_t(x,y)$, which will allow us to compute this quantity in two different ways. 
\vspace{0.3cm}
\subsubsection{Spectral computation of the heat trace} First of all, if $\{\phi_n\}$ denotes an $L^2$-orthonormal basis of eigenvectors for $\Delta$ (where $\phi_n$ has eigenvalue $\lambda_n$), we have
\begin{equation*}
K_t(x,y)=\sum e^{-\lambda_n t}\phi_n(x)\overline{\phi_n(y)},
\end{equation*}
and we can explicitly compute
\begin{align*}
\mathrm{trace}(e^{-t\Delta}) &= \int_X K_t(x,x) dx \\
&=\int_X\sum e^{-\lambda_t}| \phi_n(x) |^2dx \\
&=\sum e^{-\lambda_n t}
\end{align*}
as $\int_X | \phi_n(x)|^2 dx = 1$ for every $n.$ We therefore obtain the left hand side of Equation \eqref{mckean}.  We regard the latter expression for the trace of the heat operator as the ``spectral side" of McKean's formula, which loosely corresponds to computing the trace of a matrix by summing its eigenvalues. 
\vspace{0.3cm}
\subsubsection{Geometric computation of the heat trace}\label{geomheat}Let $k_t(\cdot,\cdot)$ be the heat kernel on the hyperbolic plane $\mathbb{H}^2$. An explicit expression for $k_t(\cdot,\cdot)$ can be found in \cite[Section $7.4$]{Bus}, but for our introductory discussion all we need is the fact that it only depends on the distance between the two points, i.e. it has the form
\begin{equation*}
k_t(\cdot,\cdot)=\tilde{k}_t(d(\cdot,\cdot))
\end{equation*}
where $d(\cdot,\cdot)$ denotes the hyperbolic distance. This follows directly from the fact that the isometry group of $\mathbb{H}^2$ acts transitively on the set of pairs of points with a given distance.
\par
The identity
\begin{equation}\label{kernelexpr2}
K_t(x,y)=\sum_{\gamma\in\Gamma} k_t(\tilde{x},\gamma\cdot\tilde{y})
\end{equation}
holds, where $\tilde{x},\tilde{y}$ are any preimages of $x,y$ in $\mathbb{H}^2$ \cite[Section $7.5$]{Bus}. Here, the right hand side is manifestly bi-invariant for the action of $\Gamma$, and therefore descends to $X\times X$. Then (\ref{kernelexpr2}) readily follows from the characterization \cite[Chapter 7]{Roe} of $K_t(x,y)$ as the unique (sufficiently smooth) time dependent function on $X\times X$ satisfying the properties
\begin{align*}
\left( \frac{d}{dt}+\Delta_x \right)K_t(x,y)&=0\\
K_t(x,y)&\rightarrow \delta(x,y)\text{ as }t\rightarrow 0.
\end{align*}
These can be checked directly using the fact that for all $\gamma$, $k_t(\tilde{x},\gamma\cdot \tilde{y})$ satisfies
\begin{align*}
\left( \frac{d}{dt}+\Delta_{\tilde{x}} \right)k_t(\tilde{x},\gamma\cdot \tilde{y})&=0\\
k_t(\tilde{x}, \gamma\cdot\tilde{y})&\rightarrow \delta(\tilde{x},\gamma\cdot \tilde{y})\text{ as }t\rightarrow 0,
\end{align*}
because $k_t$ is the heat kernel on $\mathbb{H}^2$. We therefore also have
\begin{align*}
\mathrm{trace}(e^{-t\Delta}) &=\int_X K_t(x,x) dx \\
&=\sum_{\gamma\in\Gamma} \int_F k_t(\tilde{x},\gamma\cdot\tilde{x})d\tilde{x}
\end{align*}
where $F$ is a fundamental domain for the action of $\Gamma$ on $\mathbb{H}^2$. We are left to compute the above integrals. The $\gamma=1$ term corresponds to
\begin{align*}
\int_F{k}_t(\tilde{x},\tilde{x})d\tilde{x} &= \int_F\tilde{k}_t(0)d\tilde{x} \\
&=\tilde{k}_t(0)\cdot \mathrm{vol}(X),
\end{align*}
which corresponds to the first term in the right hand side of \eqref{mckean} after taking into account the explicit expression of the heat kernel. Regarding the other terms, let us set $\{\Gamma\}'$ to be the set of non-trivial conjugacy classes of $\Gamma$; this set is in bijection with the set of closed geodesics on $X$, because each free homotopy class of loops contains exactly one geodesic representative \cite[Section $4.1.5$]{Mar}. Denoting by $\Gamma_{\gamma}$ the centralizer of $\gamma$ in $\Gamma$, we have
\begin{align*}
\sum_{\gamma\in\Gamma\setminus \{1\}} \int_F k_t(\tilde{x},\gamma\cdot\tilde{x})d\tilde{x}&=\sum_{[\gamma]\in\{\Gamma\}'}\sum_{\delta\in \Gamma_\gamma\setminus\Gamma} \int_F k_t(\tilde{x},\delta^{-1}\gamma\delta \cdot\tilde{x})d\tilde{x} \\
&=\sum_{[\gamma]\in\{\Gamma\}'}\sum_{\delta\in \Gamma_\gamma\setminus\Gamma} \int_F k_t(\delta\tilde{x},\gamma\delta \cdot\tilde{x})d\tilde{x}\\
&=\sum_{[\gamma]\in\{\Gamma\}'}\sum_{\delta\in \Gamma_\gamma\setminus\Gamma} \int_{\delta F} k_t(\tilde{x},\gamma \cdot\tilde{x})d\tilde{x} \\
&=\sum_{[\gamma]\in\{\Gamma\}'}\int_{F'_{\gamma}} k_t(\tilde{x},\gamma \cdot\tilde{x})d\tilde{x}
\end{align*}
where $F'_{\gamma}=\bigcup_{\delta\in \Gamma_{\gamma}\setminus\Gamma} \delta F $ is a fundamental domain for $\Gamma_{\gamma}$.
\par
The key observation here is that to compute each integral we can choose any fundamental domain for $\Gamma_{\gamma}$. In our case, $\Gamma$ is a torsion-free cocompact lattice in $\mathrm{Isom}^+(\mathbb{H}^2)=\mathrm{PSL}(2,\mathbb{R})$, and therefore each non-trivial element is hyperbolic. Writing $\gamma=\gamma_0^n$ for a primitive geodesic $[\gamma_0]$ and $n\in \mathbb{N}$, we have that $\Gamma_{\gamma}$ is the infinite cyclic group generated by $\gamma_0$, cf. \cite[Lemma $4.2.2$]{Mar}. Working in the upper half plane model, after conjugation we can assume $\gamma_0$ is the hyperbolic element corresponding to the dilation
\begin{align*}
\gamma_0: \mathbb{H}^2\rightarrow \mathbb{H}^2\\
z\mapsto\lambda z
\end{align*}
with $\lambda=e^{\ell(\gamma_0)}$. Here, using coordinates $z=x+iy$, we can take the strip $1<y<\lambda$ as the fundamental domain for $\Gamma_{\gamma}$. Therefore, we have
\begin{equation}\label{geomint}
\int_{F'_\gamma} k_t(\tilde{x},\gamma \cdot\tilde{x})d\tilde{x}=\int_1^{\lambda}\int_{-\infty}^{+\infty} \tilde{k}_t(d(z,\lambda^n z))\frac{dxdy}{y^2}.
\end{equation}
Such an integral clearly only depends on $\lambda$ and $n$. To compute it, notice that for fixed $x,y$, the map $w\mapsto (w-x)/y$ is an isometry sending $z$ to $i$ and $\lambda^nz$ to $b+i\lambda^n$, where $b=(\lambda^n-1)x/y$. A simple substitution in the integral above then shows that
\begin{equation*}
\int_{F'_\gamma} k_t(\tilde{x},\gamma \cdot\tilde{x})d\tilde{x}=\frac{\log\lambda}{\lambda^n-1}\int_{-\infty}^{\infty} \tilde{k}_t(d(i,b+i\lambda^n))db.
\end{equation*} 
Following \cite[Section $9.2$]{Bus}, the last integral is essentially the \textit{Abel transform} of $\tilde{k}_t$ \cite[Section 7.3]{Bus}: this is a classical integral transform taking as input a radial function (such as $\tilde{k}_t(d(i,\cdot))=k_t(i,\cdot)$) and converting it into a function which is constant on each horocycle $y=\text{const}$ (in our case $y=\lambda^n=e^{\ell(\gamma)}$) by suitably integrating it over it. The key point is that the Abel transform admits an \textit{explicit} inverse transform. This allows to calculate the desired integral using the \textit{explicit} form of the heat kernel $\tilde{k}_t$. In fact, the determination of $\tilde{k}_t$ also relies on the (inverse) Abel transform, \cite[Section $7.4$]{Bus}.
\par
While we will not pursue the complete computation here, and refer again the reader to Chapter $7$ and $9$ of \cite{Bus} for details, the final answer gives us the corresponding term in the sum on the right hand side of \eqref{mckean}. We regard this second determination of the trace of the heat operator as the ``geometric side" of McKean's formula, which loosely corresponds to computing the trace of a matrix by summing its diagonal entries.   

\vspace{0.3cm}

\subsection{The representation theoretic generalization} \label{traceformulareptheory}
In our discussion of McKean's formula, we observed: equating two expressions for the trace of convolution with the heat kernel on a compact hyperbolic surface leads to a deep relationship between its hyperbolic geometry and its spectral geometry. This paradigm can be greatly generalized as follows (see Section $3$ of \cite{White}). Consider a Lie group $G$ with Haar measure $dg$ and a discrete cocompact subgroup $\Gamma\subset G$. Denote by $K\subset G$ a maximal compact subgroup. In the concrete situation of a hyperbolic surface $X$, $G= \mathrm{PSL}_2(\mathbb{R})$, $\Gamma=\pi_1(X)$ is cocompact and torsion-free, and $K=\mathrm{PSO}_2 \subset \mathrm{PSL}_2(\mathbb{R})$. Note that the hyperbolic plane arises as the locally symmetric space $\mathbb{H}^2=G/K$, where we use the upper half plane model for $\mathbb{H}^2$, $G$ acts via linear fractional transformations, and $K$ is the stabilizer of $i$. Furthermore, we can identify $X=\Gamma\setminus G/K$.
\begin{remark}\label{unimodular}
For the purposes of this section, following \cite{White}, we will make the simplifying (but inessential, cf. Remark \ref{notuni} below) assumption that the group $G$ and the relevant subgroups we consider are \textit{unimodular}, i.e. left and right Haar measures coincide. Recall in general that one can define the \textit{modular function} of a Lie group $G$ as
\begin{equation}\label{modfun}
\Delta_G(g)=|\det \mathrm{Ad}(g)|.
\end{equation}
This function measures the failure of a right Haar measure $d_r(\cdot)$ to be left invariant, i.e. 
\begin{equation*}
d_r(g\cdot)=\Delta(g)d_r(\cdot),
\end{equation*}
and therefore a group $G$ is unimodular iff its modular function is identically $1$. See \cite[Chapter VIII.2]{Knapp} for basic facts about modular functions and unimodular groups. In the present section we only need the fact that semisimple (e.g. $\mathrm{PSL}_2(\mathbb{R})$), abelian and discrete groups are unimodular \cite[Corollary 8.31]{Knapp}. The relevance of this notion is the following. Let $H$ be a closed unimodular subgroup of a unimodular group $G$ with Haar measure $dg$. Then for any Haar measure $dh$ on $H$, there is a unique $G$-invariant measure $dx$ on $H\backslash G$ such that for all $f\in C^{\infty}_c$ the identity
\begin{equation*}
\int_G f(g)dg=\int_{H\backslash G}\left(\int_H f(hx)dh\right)dx
\end{equation*}
holds, see \cite[Theorem 8.36]{Knapp}. The measure $dx$ depends on $dh$; we will denote it by $\frac{dg}{dh}$, and write
\begin{equation*}
\int_G f(g)dg=\int_{H\backslash G}\left(\int_H f(hg)dh\right)\frac{dg}{dh}.
\end{equation*}
Similarly, if $K$ is a closed unimodular subgroup of $H$ with Haar measure $dk$, the identity
\begin{equation*}
\int_{K\setminus G} f(g)\frac{dg}{dk}=\int_{H\backslash G}\left(\int_{K\backslash H} f(hg)\frac{dh}{dk}\right)\frac{dg}{dh}
\end{equation*}
holds. In this section we will always work with $K$ a discrete subgroup, and choose $dk$ to be the counting measure. In this case for simplicity we will denote $\frac{dg}{dk}$ and $\frac{dh}{dk}$ by $dg$ and $dh$ respectively.
\end{remark}

Fix a Haar measure $dg$ on $G$, and consider the Hilbert space $L^2(\Gamma \backslash G)$. Notice that in this new setup the functions $L^2(X)$ on $X=\Gamma\setminus G/K$ correspond to the functions $L^2(\Gamma \backslash G)$ which are invariant under the action of $K$ by right translation. Given $f \in C^{\infty}_c(G)$, we can use it to define the (right) convolution operator $R(fdg)$ whose value on $\phi\in L^2(\Gamma \backslash G)$ is given by
\begin{equation*}
(R(fdg)\phi)(h)=\int_G f(g)\phi(hg)dg.
\end{equation*}

Again, we equate two expressions for the trace of such convolution operator $R(fdg)$; these loosely correspond to expressing its trace, on the one hand, by summing its diagonal matrix entries (the ``geometric side") and expressing it, on the other hand, as the sum of its eigenvalues (the ``spectral side") . 

\bigskip

Regarding the geometric side, notice that
\begin{align*}
(R(fdg)\phi)(h) &= \int_G f(g)\phi(hg)dg \\
&=\int_G f(h^{-1}g)\phi(g)dg \\
&= \int_{\Gamma \backslash G}\sum_{\gamma\in \Gamma}f(h^{-1}\gamma g)\phi(\gamma g)dg \\
&=\int_{\Gamma \backslash G} \left( \sum_{\gamma\in \Gamma}f(h^{-1}\gamma g) \right) \phi(g)dg
\end{align*}
so that $R(f)$ is an integral operator on $\Gamma\setminus G$ with kernel $K_f(h,g)=\sum_{\gamma\in\Gamma} f(h^{-1}\gamma g),$ cf. Equation (\ref{kernelexpr2}). We can then compute the trace in terms of geometric data as outlined in the previous paragraph. In this more general setup, if we set $\{\Gamma\}$ to be the set of \textit{all} conjugacy classes of $\Gamma$, and still denote the centralizer of $\gamma\in \Gamma$ by $\Gamma_{\gamma}$, we get
\begin{align*}
\mathrm{trace} \; R(fdg) &=\int_{\Gamma\setminus G}K_f(g,g) \; dg \\
&=\int_{\Gamma\setminus G} \sum_{\gamma\in\Gamma} f(g^{-1}\gamma g) \; dg\\
&=\int_{\Gamma\setminus G} \sum_{\gamma\in\{\Gamma\}}\sum_{\delta\in\Gamma_{\gamma}\setminus\Gamma} f(g^{-1}\delta^{-1}\gamma \delta g) \;dg\\
&= \sum_{[\gamma]\in\{\Gamma\}}\int_{\Gamma\setminus G}\sum_{\delta\in\Gamma_{\gamma}\setminus\Gamma} f(g^{-1}\delta^{-1}\gamma \delta g) \; dg\\
&=\sum_{[\gamma]\in\{\Gamma\}}\int_{\Gamma_{\gamma}\setminus G} f(g^{-1}\gamma g) \; dg.
\end{align*}
For any $\gamma\in\Gamma$, we denote its centralizer in $G$ by $G_{\gamma}$, and assume that it is unimodular. After choosing a Haar measure $dg_{\gamma}$ on $G_{\gamma}$, we have via Remark \ref{unimodular} that
\begin{align*}
\int_{\Gamma_{\gamma}\setminus G} f(g^{-1}\gamma g) \; dg&=\int_{G_{\gamma}\setminus G} \left( \int_{\Gamma_{\gamma}\setminus G_{\gamma}} f((g_{\gamma}g)^{-1}\gamma g_{\gamma}g) dg_{\gamma}\right) \frac{dg}{dg_{\gamma}}\\
&=\int_{G_{\gamma}\setminus G} \left(f(g^{-1}\gamma g) \int_{\Gamma_{\gamma}\setminus G_{\gamma}} dg_{\gamma}\right) \frac{dg}{dg_{\gamma}}\\
&=\mathrm{vol}(\Gamma_{\gamma}\setminus G_{\gamma},dg_{\gamma})\int_{G_{\gamma}\setminus G} f(g^{-1}\gamma g)\frac{dg}{dg_{\gamma}}
\end{align*}
where we used that $g_{\gamma}$ commutes with $\gamma$. We therefore obtain
\begin{equation*}
\mathrm{trace} \; R(fdg)=\sum_{[\gamma]\in \{\Gamma\}}\mathrm{vol}(\Gamma_{\gamma}\setminus G_{\gamma},dg_{\gamma})\int_{G_{\gamma}\setminus G} f(g^{-1}\gamma g)\frac{dg}{dg_{\gamma}}.
\end{equation*}
This is a far reaching generalization of the right hand side of \eqref{mckean}. For example, given a torsion-free cocompact lattice $\Gamma\subset\mathrm{PSL}_2(\mathbb{R})$ and a non-trivial conjugacy class $[\gamma]$, for natural choices of Haar measures $\mathrm{vol}(\Gamma_{\gamma}\setminus G_{\gamma},dg_{\gamma})$ specializes to the more familiar quantity $\ell( \gamma_0).$ This follows because in this case $\gamma$ is hyperbolic, $G_{\gamma}$ is the copy of $\mathbb{R}$ consisting of the identity and the hyperbolic elements having the same axis as $\gamma$ (in particular $G_{\gamma}$ is unimodular), and $\Gamma_{\gamma}$ is the cyclic group generated by $\gamma_0$ (see \cite[Section $4.2.1$]{Mar}).

\bigskip

We also introduce the notation
\begin{equation*}
O_{\gamma}\left(f\frac{dg}{dg_{\gamma}}\right)=\int_{G_{\gamma}\setminus G} f(g^{-1}\gamma g) \; \frac{dg}{dg_{\gamma}}
\end{equation*}
and refer to this quantity as an \textit{orbital integral}.

\bigskip

The spectral side of McKean's formula \eqref{mckean}, which involves the spectral theory of the Laplacian, is generalized in terms of the representation theory of $G$. Namely, $L^2(\Gamma\backslash G)$ carries a natural unitary action of $G$ by right translation. One can show that, as a unitary representation of $G$, it decomposes as an orthogonal direct sum of irreducible unitary representations
\begin{equation}\label{decrep}
L^2(\Gamma\setminus G) = \bigoplus_{\pi\in\widehat{G}}m_\Gamma(\pi) \cdot \pi,
\end{equation}
where $\widehat{G}$ denotes the set of equivalence classes of unitary irreducible representations, and each multiplicity $m_{\Gamma}(\pi)$ is finite \cite[Theorem $3.16$]{White}. Notice also that the number of representations with $m_{\Gamma}(\pi)\neq0$ is countable, because $L^2(\Gamma\setminus G)$ is a separable Hilbert space (and therefore every orthonormal basis is countable). Consider again the operator $R(fdg)$ for $f\in C_c^\infty(G)$. This preserves the decomposition into irreducibles $(\ref{decrep})$, and so 
\begin{equation*}
\mathrm{trace} \; R(fdg)=\sum_{\pi\in\widehat{G}}m_\Gamma(\pi) \cdot \mathrm{trace} \; (\pi(fdg)),
\end{equation*}
where, for a representation $\pi$ of $G,$ $\pi(fdg)$ is defined as
\begin{equation*}
v\mapsto \int_G f(g)(\pi(g)v)dg.
\end{equation*}
Putting everything together proves the following (see also Section $3.5$ in \cite{White}): 


\begin{thm} \label{traceformulagrouplevel}
Given a unimodular group $G$, consider a cocompact lattice $\Gamma$, and assume that all centralizers $G_{\gamma}$ are unimodular. Then, for every smooth compactly supported function $f$ on $G,$ there is an equality
\begin{equation} \label{traceformulaequation}
\sum_{\pi}  m_{\Gamma}(\pi) \cdot \mathrm{trace}(\pi(fdg) ) = \sum_{[\gamma]\in\{\Gamma\}} \vol(\Gamma_{\gamma} \backslash G_{\gamma}, dg_{\gamma}) \cdot O_{\gamma}\left(f\frac{dg}{dg_{\gamma}}\right).
\end{equation}
The right side of \eqref{traceformulaequation} is called the \emph{geometric side of the trace formula}.  The left side of \eqref{traceformulaequation} is called the \emph{spectral side of the trace formula}.
\end{thm}  
\begin{remark}
Stated in this generality, the theorem does not require the cocompact lattice $\Gamma$ to be torsion-free. On the other hand, we will only work with torsion-free lattices in what follows.\end{remark}

\begin{remark}\label{notuni}
As mentioned in Remark \ref{unimodular}, the assumptions on unimodularity can be dropped, see \cite{Selberg1} \cite{Selberg2} \cite{Tamagawa}. In fact, one shows that the mere existence of a cocompact lattice $\Gamma\subset G$ implies that $G$ and all centralizers $G_{\gamma}, \gamma \in \Gamma,$ are unimodular.
\end{remark}

\vspace{0.3cm}

\subsection{Specializing the general Selberg trace formula} \label{specializingtraceformula}
The expression from Theorem \ref{traceformulagrouplevel}, as written, is too general for practical use.  It contains information about the spectra of all $G$-invariant differential operators on $\Gamma\setminus G/K$ and bundles thereon. One can obtain equations such as \eqref{tracesurface} and the formula in Theorem \ref{Thm4} by looking at the trace of suitable $f\in C_c^\infty(G)$ to isolate a much smaller subset of these spectra. While the process of specializing the trace formula to concrete examples is somewhat technical, our main goal in this subsection is to make the basic idea behind it transparent.
\vspace{0.3cm}

Consider the case of functions on a hyperbolic surface $X=\Gamma\setminus \mathbb{H}^2$. We refer the reader to Section $4$ of \cite{White} for a more detailed discussion. Functions on $X$ correspond to functions in $L^2(\Gamma \backslash G)$ which are invariant under the action of $K$ by right translation. Therefore, if we wish to study spectrum of functions on $X$, we only really care about representations $\pi\in\widehat{G}$ which admit non-trivial $K$-invariant vectors, i.e. $\pi^K\neq0$. These representations can be isolated in Theorem \ref{traceformulagrouplevel} by taking convolution with test functions $f\in C_c^\infty(G)$ with are invariant under the actions of $K$ by both right and left translations. In fact, if $f$ is such a function, denoting by $R_k$ the right translation by $k\in K$ on $L^2(\Gamma \backslash G),$ we have $R(fdg)=R(fdg)\circ R_k$, and
\begin{align*}
R(fdg) &= \int_K R(fdg)\circ R_kdk \\
&= R(fdg)\int_K R_k dk.
\end{align*}
But acting through a representation $\pi$, the operator $\int_K R_k dk$ corresponds to the orthogonal projection onto $\pi^K$. Hence, if $\pi^K=0$, then $\pi(f)$ has trace zero. Let us point out that there are plenty of $K$-bi-invariant functions on $G$; in fact, they correspond one-to-one with even functions on $\mathbb{R}$ via the so-called \textit{Harish-Chandra transform}. Here we identify $\mathbb{R}$ with the subgroup $A\cong \mathbb{R}^{>0}\subset G$ of positive diagonal matrices (via the exponential), and define for a $K$-bi-invariant function $f$
\begin{equation}\label{harishchandra}
Hf\left(\begin{array}{cc}a&0\\0 & a^{-1}\end{array}\right)=a\cdot \int_{\mathbb{R}}f\left(\begin{array}{cc}a&ax\\0 & a^{-1}\end{array}\right)dx,
\end{equation}
see \cite[Section $4.5$]{White}. In fact, this is essentially the translation in the language of representation theory of the Abel transform we discussed in Subsection \ref{geomheat}. First of all, $K$ bi-invariant functions on $G$ correspond to radial functions on $\mathbb{H}^2$ of the form $f(d(i,\cdot))$. Furthermore, in the quotient $G/K=\mathbb{H}^2$, the argument of $f$ in the integral \eqref{harishchandra} corresponds to the point $a^2x+ia^2$. We are therefore integrating $f$ along the horocycle $y=a^2$, and we can think of $Hf$ as a function constant on such horocycles. In this language, the fact that the Harish-Chandra transform is a bijection corresponds to the fact that the Abel transform admits an explicit inverse.
\\
\par
Furthermore, we have a complete understanding of the (non-trivial) unitary representations $\pi$ of $\mathrm{PSL}_2(\mathbb{R})$ for which $\pi^K\neq0$. These are often denoted by $\pi_s$, where $s\in i\mathbb{R}\cup [-1,1]$, and are all infinite dimensional. The ones corresponding to an imaginary parameter are called \textit{(unitary) principal series representations}, while those corresponding to a real parameter are called \textit{complementary series representations}. We have in these cases that $\pi_s^K$ is one dimensional, and via the correspondence between $K$-fixed elements and functions on $X,$ a basis element of this one dimensional space $\pi_s^K \subset L^2(\Gamma \backslash X)$ corresponds to an eigenvector of the Laplacian of eigenvalue $(1-s^2)/4$. To conclude, by taking $f$ to be $K$-bi-invariant, the left side (the ``spectral side") of Theorem \ref{traceformulagrouplevel} reduces to a non-trivial sum involving the eigenvalues of the Laplacian on $X$, while the right side (the ``geometric side") reduces to a sum involving the lengths of the geodesics of $X$. To massage this formula into its final form (\ref{tracesurface}), one needs to perform a series of non-trivial computations (see Section $4$ of \cite{White}).

\vspace{0.3cm}
Finally, an analogous approach works in the context of the present paper.  In our context, $G=\mathrm{PGL}_2(\mathbb{C})$ and $K=\mathrm{PU}(2) \cong \mathrm{SO}_3$. As we are interested in $1$-forms, instead of looking at representations $\pi\in\widehat{G}$ with $\pi^K\neq 0$, we look at those containing copies of the standard representation of $K=\mathrm{SO}_3$ on $\mathbb{R}^3.$ An additional complication in this case is that the spectrum on $1$-forms contains both the spectrum on exact forms (i.e. the non-trivial spectrum on functions) and coexact forms (which is what we are really interested in). Using test functions of a very particular type, whose existence is afforded to us by a theorem of Bouaziz \cite{Bouaziz}, we are nonetheless able to isolate the contribution of coexact $1$-forms.  We do so in \S \ref{traceformula}.

\begin{remark}
As in the case of $\mathrm{PSL}_2(\mathbb{R})$, the group $G=\mathrm{PGL}_2(\mathbb{C})$ is semisimple hence unimodular. Furthermore, we will see that in our setup all relevant centralizers are also unimodular, so that Theorem \ref{traceformulagrouplevel} can be applied. On the other hand, when describing explicitly the unitary irreducible representations of both $\mathrm{PSL}_2(\mathbb{R})$ (which are the $\pi_s$ mentioned above) and $\mathrm{PGL}_2(\mathbb{C})$, one is naturally led to consider the subgroups of upper triangular matrices, which are not unimodular. We will discuss in detail the construction of the unitary irreducible representations of $\mathrm{PGL}_2(\mathbb{C})$ in Subsection \ref{unitarydual} of Appendix $B$, while the case of $\mathrm{PSL}_2(\mathbb{R})$ can be found in \cite[Section $4.4$]{White}.
\end{remark}

\vspace{0.5cm}
\section{The trace formula specialized to coexact 1-forms on hyperbolic 3-manifolds} \label{traceformula}

Consider a closed oriented hyperbolic $Y$ three-manifold, and fix a smooth compactly supported even test function $H$ on $\mathbb{R}.$  Our goal is to explain how, by specializing the general trace formula
\begin{equation}\label{traceappendixform}
\sum_{\pi}  m_{\Gamma}(\pi) \cdot \mathrm{trace}(\pi(fdg) ) = \sum_{[\gamma]\in\{\Gamma\}} \vol(\Gamma_{\gamma} \backslash G_{\gamma}, dg_{\gamma}) \cdot O_{\gamma}\left(f\frac{dg}{dg_{\gamma}}\right)
\end{equation}
of Theorem \ref{traceformulagrouplevel} appropriately for the semisimple group $G=\mathrm{PGL}(2,\mathbb{C})$ and the torsion-free cocompact lattice $\Gamma\cong \pi_1(Y)\subset G$, one can obtain the identity of Theorem \ref{geometrictraceformulacoexact1formsnew}:
\begin{align*}
&{} \sum_{\lambda^\ast = \text{coexact 1-form eigenvalue}} \frac{1}{2} m_\Gamma(\lambda^\ast) \cdot \widehat{H} \left( \sqrt{\lambda^\ast} \right) + \left( \frac{1}{2} b_1 \left( Y \right) - \frac{1}{2}\right) \widehat{H}(0) \\
&=  \frac{\vol(Y)}{2\pi} \cdot \left(  H(0) - H''(0) \right) + \sum_{[\gamma] \neq 1} \ell(\gamma_0) \frac{ \cos( \mathrm{hol}(\gamma))}{ |1 - e^{\C \ell(\gamma)} | \cdot |1 - e^{-\C \ell(\gamma)}| } \cdot H(\ell(\gamma)). 
\end{align*}
Our discussion is structured as follows:
\begin{itemize}
\item in Subsection \ref{aboutG}, we discuss some relevant preliminary notions about $G$ needed for our computation. In particular, we review the identification of $G$ with the isometry group of $\mathbb{H}^3$, fix conventions for the Haar measures, and describe the classification of its irreducible unitary representations.
\item in Subsection \ref{shc}, we introduce the so-called Satake-Harish-Chandra transform of $f$. This is a function $Sf$ defined on the maximal torus $T$ of $G$, and is a generalization of the Harish-Chandra transform \eqref{harishchandra}.
\item In Subsections \ref{geomshc} and \ref{specshc} we begin our specialization process by showing how the terms in equation \eqref{traceappendixform} can all be expressed very concretely in terms of the Satake-Harish-Chandra transform $Sf$ of $f$. In particular, this will greatly simplify the orbital integral on the geometric side and the trace computation on the spectral side.
\item in Subsection \ref{noshc} we invoke a result of Bouaziz to rewrite the trace formula directly in terms of a given function $F$ on $T$, rather than the transform $Sf$ of $f$.
\end{itemize}
The formula obtained this way (see Corollary \ref{preliminarygeometrictraceformula} below) is still too general for our application, as it takes into account all unitary irreducible representations of $G$ (hence, in some sense, all natural differential operators on $Y$). To remedy the situation, we proceed as follows:
\begin{itemize}
\item  In Subsections \ref{rep1form} and \ref{HodgeLaplacian}, we identify precisely which unitary irreducible representations of $G$ are relevant for our purposes, i.e. contain information about the spectrum on coexact $1$-forms.
\item Finally, in \ref{iso1form} and \ref{theend}, we complete the proof by choosing suitable functions $F$ on $T$ that isolate the contribution of the spectrum on coexact $1$-forms.
\end{itemize}

\vspace{0.3cm}
\subsection{Preliminaries on the Lie group $G$.}\label{aboutG} We begin by recalling some background notions regarding our group, fixing notations and Haar measures, and discussing its unitary irreducible representations.

\subsubsection{The isometry group of $\mathbb{H}^3$.}\label{GH3}We review some notions of hyperbolic geometry which are relevant for our purposes, and refer the reader to \cite[Chapter 1]{MacReid} for more a detailed discussion. While for our purposes it will be more convenient to deal with the group $G=\mathrm{PGL}_2(\mathbb{C})$, the relationship with hyperbolic geometry is more transparent when working with $\mathrm{PSL}_2(\mathbb{C})$ (which is clearly isomorphic to $G$). We will always work with the upper half-space model 
\begin{equation*}
\mathbb{H}^3=\mathbb{C}\times \mathbb{R}^{>0}
\end{equation*}
with coordinates $(z,t)$ and metric tensor $g_{\mathbb{H}^3}=\frac{1}{t^2}g_{\mathrm{eucl}}$. The group of orientation-preserving isometries of $\mathbb{H}^3$ is isomorphic to $\mathrm{PSL}_2(\mathbb{C})$, and the action of
\begin{equation*}
\left( \begin{array}{cc} a & b \\ c & d \end{array} \right)\in\mathrm{PSL}_2(\mathbb{C})
\end{equation*}
on $\mathbb{H}^3$ is given by
\begin{equation}\label{upperaction}
\left( \begin{array}{cc} a & b \\ c & d \end{array} \right)\cdot (z,t)=\left(\frac{(az+b)(\bar{c}\bar{z}+\bar{d})+a\bar{c}t}{|cz+d|^2+|c|^2t^2}, \frac{t}{|cz+d|^2+|c|^2t^2}\right)
\end{equation}
see \cite[Section $1.1$]{EGM}.
\par
A compact hyperbolic three-manifold $Y$ corresponds to a quotient $\Gamma\setminus \mathbb{H}^3$ with $\Gamma\subset \mathrm{PSL}
_2(\mathbb{C})$ a torsion-free cocompact lattice; in this case, every non-trivial element in $\Gamma$ is loxodromic. Recall that a loxodromic element $\gamma$ is an element whose action on $\mathbb{H}^3$ has exactly two fixed points, both at infinity; this is equivalent to $\mathrm{tr} \; \gamma \in \mathbb{C}\setminus[-2,2]$. Geometrically it corresponds to a screw motion translating by $\ell(\gamma)>0$ along the geodesic connecting the two fixed points at infinity (called the \textit{axis} of $\gamma$), and simultaneously rotating by $\mathrm{hol}(\gamma)$ around it. There is a bijection between non-trivial conjugacy classes in $\Gamma$ and oriented closed geodesics in $Y$ (\cite[Lemma $4.1.5$]{Mar}), and under this bijection the complex length $\C \ell(\gamma)=\ell(\gamma)+i \; \mathrm{hol}(\gamma)$ of the class of an element $\gamma$ conjugate to
\begin{equation*}
\left( \begin{array}{cc} w & 0 \\ 0 & w^{-1} \end{array} \right),\quad |w|>1
\end{equation*}
is $\mathbb{C}\ell(\gamma)=2\log w$ \cite[Section 12.1]{MacReid}.
\begin{remark}\label{lenPGL}
When thinking of $\mathrm{PGL}_2(\mathbb{C})$, any loxodromic element $\gamma$ is conjugate to 
\begin{equation*}
\left( \begin{array}{cc} z & 0 \\ 0 & 1 \end{array} \right),\quad |z|>1
\end{equation*}
for some $z$, and we have $\mathbb{C}\ell(\gamma)=\log z$.
\end{remark}
We conclude by discussing the various centralizers that will appear, see \cite[Section 4.2.1]{Mar} for details. For a given loxodromic element $\gamma$, the centralizer $G_{\gamma}$ consists of the identity element, all the loxodromic elements which translate with rotation along the axis of $\gamma$ (possibly in the opposite direction) and all the elliptic elements fixing pointwise the axis of $\gamma$. It is therefore a copy of $\mathbb{R} \times S^1\cong \mathbb{C}^\times.$ In particular it is abelian, hence unimodular (cf. Remark \ref{unimodular}), so that the general trace formula of Theorem \ref{traceformulagrouplevel} applies.
\par
Similarly, the centralizer $\Gamma_{\gamma}$ of $\gamma$ in the cocompact torsion-free lattice $\Gamma\cong\pi_1(Y)$ is the cyclic group generated by a primitive loxodromic element $\gamma_0$ with $\gamma=\gamma_0^n$. At the level of geodesics, $\gamma_0$ corresponds to a prime geodesic of which $\gamma$ is a multiple.

\vspace{0.3cm}
\subsubsection{Subgroups of $G = \mathrm{PGL}_2(\C)$.}\label{notation} We now introduce some distinguished subgroups of $\mathrm{PGL}_2(\C)$ (defined via subgroups on $\mathrm{GL}_2(\C)$ under the quotient map):
\begin{itemize}
\item $B$, the subgroup of upper triangular matrices;
\item $K = \mathrm{PU}_2\cong \mathrm{SO}_3$; this is a maximal compact subgroup of $G$ corresponding to the stabilizer of $(0,1)\in \mathbb{H}^3$, and $\mathbb{H}^3=G/K$.
\item $T$, the subgroup of diagonal matrices (the unique maximal torus, up to conjugation);
\item $A$, the subgroup of diagonal matrices with real entries;
\item $M$, the maximal compact subgroup of $T$ (diagonal unitary matrices);
\item $N$, the subgroup of upper triangular matrices with both diagonal terms equal to $1$.
\end{itemize}
Recall that the Iwasawa decomposition \cite[Chapter VI.4]{Knapp} implies that the multiplication map $N\times A\times K\rightarrow G$ is a diffeomorphism. In particular, every element in $G$ can be written uniquely in the form $nak$ where $n,a,k$ are in $N,A,K$ respectively. For our purposes, it will be useful to use a slightly different decomposition, see Remark \ref{measnot} below. Notice that $nak=a (a^{-1}na)k$, and that $A$ normalizes $N$. Therefore the multiplication map
\begin{equation}\label{iwasawa}
A\times N\times K\rightarrow G
\end{equation}
is also a diffeomorphism.
\\
\par
Furthermore, we will use the following notation:
\begin{itemize}
\item We denote by $W = N(T) / T,$ the Weyl group of $T$. Here $N(T)$ denotes the normalizer of $T$. The group $W$ consists of two elements.
\item We will denote the Lie algebra of a given Lie group with the corresponding gothic letter. For example $\mathfrak{g}$ is the Lie algebra obtained from the matrix Lie algebra $\mathfrak{gl}_2(\C)$ by quotienting by multiples of the identity matrix. 
\item For $S \subset G$ consisting of semisimple elements, let $S_{\mathrm{reg}}$ denote those elements of $S$ which are \textit{regular}, i.e. the elements for which the centralizer of $S$ is a maximal torus. In our context, every non-trivial element of $\Gamma$ is loxodromic and therefore regular; also, $T_{\mathrm{reg}}$ consists of all elements except those of order at most two.
\end{itemize}

\subsubsection{Haar measures}
While the trace formula is valid for any choice of measure on $G$ and the centralizers $G_{\gamma}$, for explicit computations it is convenient to fix concrete measures. Notice that all distiguished groups above are unimodular with the only exception of $B$. In the latter case, the modular function (\ref{modfun}) is given by
\begin{equation}\label{modcharB}
\delta\left( \begin{array}{cc} a_1 & \ast \\ 0 & a_2 \end{array} \right):=\Delta_B\left( \begin{array}{cc} a_1 & \ast \\ 0 & a_2 \end{array} \right)=|a_1/a_2|^2.
\end{equation}
For the rest of this Appendix, we will make the following choices of (bi-invariant) Haar measures:
\begin{itemize}
\item
$dk$ denotes the volume $1$ Haar measure on $K.$

\item
$da = du,$ where $A = \left\{ \left( \begin{array}{cc} e^u & 0 \\ 0 & 1 \end{array} \right): u \in \mathbb{R} \right\}$ and $du$ is standard Lebesgue measure on $\mathbb{R}.$ 

\item 
$dm=\frac{1}{2\pi}d\theta$ where $M = \left\{ \left( \begin{array}{cc} e^{i\theta} & 0 \\ 0 & 1 \end{array} \right): \theta \in \mathbb{R} / 2\pi \mathbb{Z}  \right\}$ with $d\theta$ the standard Lebesgue measure. This measure has volume $1$.

\item
$dt = \frac{1}{2\pi}d\theta \; du,$ where $T =AM= \left\{ \left( \begin{array}{cc} e^{u + i\theta} & 0 \\ 0 & 1 \end{array} \right): u \in \mathbb{R}, \theta \in \mathbb{R} / 2\pi \mathbb{Z} \right\}$.  

\item
$dn$ is the standard Euclidean measure $dx \; dy$ on $N = \left\{ \left( \begin{array}{cc} 1 & x + iy \\ 0 & 1 \end{array} \right): x,y \in \mathbb{R} \right\}.$
\end{itemize}
Via the decomposition \eqref{iwasawa}, these can be combined to define the measure $dg=da\;dn \;dk$ on $G$, meaning that
\begin{equation}\label{haarG}
 \int_G f(g)dg=\int_A\int_N\int_K f(ank) da\;dn \;dk.
\end{equation}
The right side of \eqref{haarG} does indeed define a Haar measure on $G$. This follows directly by applying the result \cite[Theorem $8.32$]{Knapp} on decompositions of Haar measures first to the product subgroup $H=AN\subset G$ and then to the product $HK=G$. The only non-trivial observation is that the modular function $\Delta_H$ of $H$ is trivial on its second factor $N$.
\begin{remark}\label{measnot}
This is why we use the decomposition $ANK$ instead of $NAK$: the modular function of $NA$ is \textit{not} trivial on $A$.
\end{remark}
\vspace{0.3cm}
\subsubsection{The classification of irreducible unitary representations of $\mathrm{PGL}_2(\C)$} \label{unitarydual}
We follow \cite[II \S 4]{Knapp2}, to which we refer for additional details. The trivial representation $\mathbf{1}$ is clearly unitary and irreducible. The other such representations are all infinite dimensional, and are parametrized by $n \in \mathbb{Z}, s \in \C$ as in the following discussion. Let $\chi_{s,n}: B \rightarrow \C^\times$ denote the character
\begin{equation}\label{chisn}
\chi_{s,n}: \left( \begin{array}{cc} a & \ast \\ 0 & d \end{array} \right) \mapsto \left|  a/d \right|^s \cdot \left( \frac{a/d}{\left| a/d \right|} \right)^n =:  \chi_s  \left( \begin{array}{cc} a & \ast \\ 0 & d \end{array} \right) \cdot \chi_n  \left( \begin{array}{cc} a & \ast \\ 0 & d \end{array} \right).
\end{equation}
Denote by $\pi_{s,n}$ the induced representation
\begin{equation*}
\pi_{s,n} := \mathrm{Ind}_B^G (\chi_{s,n}\cdot\delta^{1/2}),
\end{equation*}
where $\delta$ is the modular function of $B$ in Equation (\ref{modcharB}). Very explicitly, this is defined as follows.\footnote{Recall that given groups $K\subset H$, and a representation $\pi$ of $K$ on $V$, the (algebraic) induced representation $\mathrm{Ind}_K^H \pi$ is the set of functions $f:H\rightarrow V$ for which $f(kh)=\pi(k)f(h)$, with the action of $H$ by right translation. In our setup, one needs a little more attention in order to define a Hilbert space structure.} Consider the space of functions
\begin{equation}\label{funcsn}
V_{s,n}=\{ f:G\rightarrow \mathbb{C} \lvert \text{ $f$ is smooth and } f(bg)=\chi_{s,n}(b)\cdot \delta(b)^{1/2}f(g)\text{ for all }b\in B\}.
\end{equation} 
Given that $G=BK$, a function $f$ in $V_{s,n}$ is determined by the restriction $f\lvert_K$, and we set $\pi_{s,n}$ to be completion of $V_{s,n}$ with the respect to the $L^2$-norm
\begin{equation*}
\|f\|_{V_{s,n}}= \| f\lvert_K\|_{L^2(K)}.
\end{equation*}
The action of $G$ on $\pi_{s,n}$ is via right translations.
\par
We can also describe of $\pi_{s,n}$ more concretely as follows. For $G=\mathrm{PSL}_2(\C)$, we can identify
\begin{equation*}
N\setminus G\cong \left(\C^2\setminus{(0,0)}\right)/\sim\quad \text{where }(x,y)\sim (-x,-y).
\end{equation*}
Here we think of $\C^2$ as the space of row vectors, with the right action of $\mathrm{SL}_2(\C)$ (so that $N$ is the stabilizer of $(0,1)$). As $\delta$ and $\chi_{s,n}$ are trivial on $N$, a function $f$ in $V_{s,n}$ descends to a function $\bar{f}$ on $N\setminus G$ such that
\begin{equation*}
\bar{f}(\lambda(x,y))=|\lambda|^{-2s-2}\left(\frac{\lambda}{|\lambda|}\right)^{-2n}\bar{f}(x,y)
\end{equation*}
for all $\lambda\in\mathbb{C}^\times$. One readily checks that this construction defines a bijection between $V_{s,n}$ and functions on $N\setminus G$ with this homogeneity property.
\\
\par
All irreducible unitary representations of $G,$ besides the trivial one, are of the form $\pi_{s,n}.$  However, the condition that $\pi_{s,n}$ is unitary (i.e. it admits a $G$-invariant inner product) severely restricts the possible $s,n$. Indeed, there are only two classes of such representations:
\begin{itemize}
\item
For all $n\in\mathbb{Z}$ and $s\in i\mathbb{R}$, $\chi_{s,n}$ is a unitary character. In this case $\pi_{s,n}$ is irreducible and one can check that the inner product associated to $\|\cdot\|_{V_{s,n}}$ is in fact $G$-invariant. The representations $\pi_{s,n}, s \in i \R,$ are called \emph{unitary principal series} representations.
\item
For $s \in [-1,1] \backslash \{0\},$ the representations $\pi_{s,0}$ are all irreducible and admit a strange $G$-invariant inner product. These representations are known as \emph{complementary series} representations.  
\end{itemize}
Finally, the only coincidences among the representations $\pi_{s,n}$ are 
$$\pi_{s,n} \cong \pi_{\bar{s}, -n},$$
see also Remark \ref{isomorphicinducedreps} below for an explanation of these coincidences in terms of traces. We sum up our discussion as follows.
\begin{prop}[\cite{Knapp2}, II $\S 4$]\label{allirrep}
Every unitary irreducible representation of $G$ is isomorphic to one of the following:
\begin{itemize}
\item the trivial representation $\mathbf{1}$;
\item $\pi_{s,n}$ where $n\geq 1$ and $s\in i\mathbb{R}$;
\item $\pi_{s,0}$ where $s\in i\mathbb{R}^{\geq0}\cup [0,1]$.
\end{itemize}
Furthermore, two distinct representations in this list are not isomorphic.
\end{prop}

\vspace{0.3cm}
\subsection{The Satake-Harish-Chandra transform.}\label{shc}
As discussed in the previous section, the group $G$ and the relevant centralizers are all unimodular, so that the trace formula of Theorem \ref{traceformulagrouplevel} applies. Our first goal in specializing the trace formula is to show how all the terms appearing in \eqref{traceappendixform} can be expressed in terms of a special integral transform of $f$. This is the natural generalization of the Harish-Chandra (or equivalently Abel) transform used in Section \ref{heatkernel} to functions which are not necessarily $K$-bi-invariant. We will need this more general type of functions in order to describe differential forms on $Y$ from a representation-theoretic viewpoint.
\\
\par
Given a compactly supported smooth function $f$ on $G$, we define its \emph{Satake-Harish-Chandra transform} to be the function
\begin{align*}
Sf&:T\rightarrow \mathbb{C}\\
t&\mapsto \delta(t)^{1/2}  \int_N \int_K f(k^{-1}t n k) dk dn 
\end{align*}
where $\delta(t)$ is the modular function defined in  \eqref{modcharB} evaluated at $t \in T\subset B$. The function $Sf(t)$ is cleary smooth and compactly supported; it is also invariant under the action of the Weyl group
\begin{equation*}
(u,\theta)\rightarrow (-u,-\theta)\text{ or, equivalently, }t\rightarrow t^{-1},
\end{equation*}
cf. Proposition \ref{orbitalsatake} below. Notice that if $f$ is $K$-bi-invariant, and $a$ is positive, the integral simplifies to
\begin{equation*}
Sf\left(\begin{array}{cc}a&0\\0 & a^{-1}\end{array}\right)=a^2\cdot \int_{\mathbb{C}}f\left(\begin{array}{cc}a&an\\0 & a^{-1}\end{array}\right)dn.
\end{equation*}
which is the direct analogue of \eqref{harishchandra}. On the other hand, as we are interested in coexact $1$-forms, we will need to consider functions which are not necessarily $K$-bi-invariant. A key result we will need in our discussion (generalizing the fact that the Harish-Chandra transform is a bijection), is a suitable surjectivity statement for the Satake-Harish-Chandra transform, see Subsection \ref{noshc} below.

\vspace{0.3cm}
\subsection{The geometric side of the trace formula and the Satake-Harish-Chandra transform}\label{geomshc}
In this section we begin to specialize the geometric side of the general trace formula \eqref{traceappendixform} to our specific case of interest, and show how it can be expressed in terms of $Sf$. In particular, we need to compute for each conjugacy class the orbital integral
\begin{equation*}
O_{\gamma}\left(f\frac{dg}{dg_{\gamma}}\right)=\int_{G_{\gamma}\setminus G} f(g^{-1}\gamma g) \; \frac{dg}{dg_{\gamma}}
\end{equation*}
and the covolume of the centralizer
\begin{equation*}
\vol(\Gamma_{\gamma} \backslash G_{\gamma}, dg_{\gamma}).
\end{equation*}
We will see that the former can be expressed in terms of the Satake-Harish-Chandra transform of $f$, while the latter admits a direct geometric meaning in terms of translation length. The computation for the trivial conjugacy class is very different from that of a loxodromic conjugacy class, and we begin with the latter.

\subsubsection{Orbital integrals for loxodromic classes}\label{orbint} Because $\Gamma=\pi_1(Y)$ is torsion-free and cocompact, every $1 \neq \gamma \in \Gamma$ is loxodromic, hence regular (see Subsection \ref{notation}). In particular,
\begin{equation*}
h^{-1} \gamma h = t_{\gamma} \in T_{\mathrm{reg}}\text{ for some }h \in G.
\end{equation*}
This choice of $h$ is unique up to right multiplication by $N(T).$  We can define a specific Haar measure $dg_{\gamma}$ on the centralizer $G_{\gamma}$ by
\begin{equation*}
dg_{\gamma} = (\text{conjugation by } h)_\ast dt.
\end{equation*}
Because the Haar measure $dt$ is invariant under $N(T),$ the above specification of $dg_{\gamma}$ is well-defined.  In particular,
\begin{equation*}
O_{\gamma} \left(f \frac{dg}{dg_\gamma} \right) = O_{t_\gamma} \left( f \frac{dg}{dt} \right).
\end{equation*}
We now discuss how to compute the latter integral. We have the following.
\begin{prop}\label{orbitalsatake}
For every element $t\in T_{\mathrm{reg}}$, there is an equality
\begin{equation*}
O_{t} \left( f \frac{dg}{dt} \right)=| D(t^{-1})|^{-1/2} Sf(t)
\end{equation*}
where
\begin{equation*}
D(t) := |\det( 1-\mathrm{Ad}(t) \lvert_{\mathfrak{t} \backslash \mathfrak{g}}  )|
\end{equation*}
is the Weyl discriminant, and $Sf$ is the Satake-Harish-Chandra transform.
\end{prop}
Very explicitly, if $t = \left( \begin{array}{cc} z & 0 \\ 0 & 1 \end{array} \right) \in \mathrm{PGL}_2(\C)$, we readily calculate that 
\begin{equation}\label{weyldiscr}
D(t) = |(1 - z)^2(1 - z^{-1})^2|.
\end{equation}
\begin{proof}
For a given $t \in T_{\mathrm{reg}}$ (for which $G_t=T$), this corresponds to the integral
\begin{equation*}
\int_{T\setminus G} f(g^{-1}tg)\frac{dg}{dt}.
\end{equation*}
Using the integration formulas in Remark \ref{unimodular}, we have, setting for the sake of notation $dy=\frac{dt}{da}$:
\begin{align*}
\int_{A\setminus G}f(g^{-1}tg)\frac{dg}{da}&=\int_{T\setminus G}\left(\int_{A\setminus T}f(g^{-1}y^{-1}t y g)\;dy\right )\;\frac{dg}{dt}\\
&=\int_{T\setminus G}\left(f(g^{-1}t  g)\int_{A\setminus T}1\;dy\right )\;\frac{dg}{dt}\\
&=\int_{T\setminus G} f(g^{-1}tg)\frac{dg}{dt}\\
&=O_t\left(f\frac{dg}{dt}\right),
\end{align*}
where we use that $y$ commutes with $t$ and that $A\setminus T=M$ has volume $1$. Using Equation \eqref{haarG}, we see then
\begin{align*}
O_t \left( f \frac{dg}{dt} \right)  &=\int_{A\setminus G}f(g^{-1}tg)\frac{dg}{da}\\
&= \int_N\int_K  f(k^{-1}n^{-1} t n k) dk dn \\
&= \int_N\int_K f(k^{-1}t (t^{-1} n^{-1} t n) k) dk dn. \\
\end{align*}
The Jacobian of the change of variables $t^{-1} n^{-1} t n \leftrightarrow n$ is the constant $\delta(t)^{-1/2} | D(t^{-1})|^{1/2}$, hence 
\begin{align*}
O_t \left( f \frac{dg}{dt} \right)  &= | D(t^{-1})|^{-1/2} \delta(t)^{1/2}  \int_N \int_K f(k^{-1}t n k) dk dn, \\
\end{align*}
and the result follows.\end{proof}

\vspace{0.3cm}
\subsubsection{Covolumes of centralizers of loxodromic elements} \label{covolume}

Let $1 \neq \gamma \in \Gamma.$  As recalled in Subsection \ref{GH3}, the centralizer $\Gamma_\gamma$ equals $\langle \gamma_0 \rangle,$ where $\gamma_0 \in \Gamma$ is primitive and $\gamma=\gamma_0^n$. Therefore,
\begin{align*}
\vol( \Gamma_\gamma \backslash G_\gamma, dg_{\gamma} ) &= \vol( \langle \gamma_0 \rangle \backslash G_\gamma, dg_{\gamma}) \\
&= \vol( \langle t_{\gamma_0} \rangle \backslash T , dt).
\end{align*} 
Suppose $$t_{\gamma_0} = \left( \begin{array}{cc} z_0 & 0 \\ 0 & 1 \end{array} \right) \in \mathrm{PGL}_2(\C),\quad |z_0|>1.$$  
With respect to our chosen Haar measures,
\begin{equation*}
 \vol( \langle t_{\gamma_0} \rangle \backslash T , dt) = \log |z_0| = \ell(\gamma_0),
\end{equation*}
see Remark \ref{lenPGL}.
\vspace{0.3cm}
\subsubsection{The identity element}
In the case of $1\in\Gamma$, $G_1=G$ and $\Gamma_1=\Gamma$. Therefore, with respect to our chosen Haar measure $dg = da \; dn \; dk,$
$$\vol(\Gamma \backslash G, dg) = \vol(Y),$$
and the contribution of the identity term to the trace formula equals
$$\vol(Y) \cdot f(1).$$
The next proposition expresses $f(1)$ in terms of the Satake-Harish-Chandra transform $Sf(t).$ It is a special case of the general \textit{Plancherel formula}, which expresses the value of a function at the identity element in terms of suitable integral transforms. For example, in the case the group is $\mathbb{R}$, it simply states the following familiar consequence of the Fourier inversion formula:
\begin{equation*}
H(0)=\frac{1}{2\pi}\int_{\mathbb{R}} \widehat{H}(t)dt.
\end{equation*} 
\begin{prop} \label{plancherel}
There is a constant $c > 0$ for which the identity
$$f(1) = -c \left(  \frac{d^2}{du^2} + \frac{d^2}{d\theta^2} \right) Sf |_{(u,\theta)=(0,0)}$$
holds for every smooth compactly supported $f$. 
\end{prop}

\begin{proof}
This is Lemma $11.1$ in \cite{Knapp2}, where in their notation $F^{T}_f$ is the Satake-Harish-Chandra transform, $\partial(\alpha)=\frac{d}{du}-i\frac{d}{d\theta}$ and $\partial(\bar{\alpha})=\frac{d}{du}+i\frac{d}{d\theta}$.
\end{proof}

\begin{remark}
While in principle the constant $c$ can be determined directly, we will instead derive it at the end of our computations via the Weyl law for the asymptotic number of coexact eigenvalues.
\end{remark}

\vspace{0.3cm}
\subsection{The spectral side of the trace formula and the Satake-Harish-Chandra transform}\label{specshc} We now discuss how to compute the spectral side of \eqref{traceappendixform} in terms of $Sf$.  We need to compute $\mathrm{trace}(\pi(fdg))$ for every irreducible unitary $\pi$ of $G$ (classification recalled in Subsection \ref{unitarydual}).  Our goal is to express all of these traces in terms of the Satake-Harish-Chandra transform $Sf$.

\vspace{0.3cm}
\subsubsection{The trivial representation} In this case the trace is simply
\begin{equation*}
\mathrm{trace}( \mathbf{1}(f dg) ) = \int_G f(g) dg.
\end{equation*}
We have the following.
\begin{prop} \label{tracetrivialrep}
Given our choices of Haar measures, the identity
$$\mathrm{trace}( \mathbf{1}(f dg) ) = \frac{1}{|W|} \int_T D(t^{-1})^{1/2} \cdot Sf(t)  dt$$
holds. Here $W$ is the Weyl group of $T,$ hence $|W|=2$.
\end{prop}

\begin{proof}
This follows from the following computation
\begin{align*}
\int_G f(g) dg&= \frac{1}{|W|} \int_T | D(t^{-1})| \int_{T \backslash G} f(g^{-1} t g) \frac{dg}{dt} dt \\
&= \frac{1}{|W|} \int_T |D(t^{-1})|^{1/2} \cdot \left( |D(t^{-1})|^{1/2} O_t \left( f \frac{dg}{dt} \right) \right) dt \\
&= \frac{1}{|W|} \int_T |D(t^{-1})|^{1/2} \cdot Sf(t)  dt. 
\end{align*}
In the first line we used the Weyl integration formula for non-compact groups \cite[Theorem $8.64$]{Knapp}, where we set $r=1$ and $H_1=T$ because in our case $T$ is the unique maximal torus of $G$ up to conjugation. In fact, this last observation implies that the proof of the Weyl integration formula for compact groups \cite[Theorem $8.60$]{Knapp} applies directly.\end{proof}

\begin{remark}\label{mult1}
Notice that a copy of the trivial representation $\mathbf{1}\subset L^2(\Gamma\setminus G)$ corresponds to a $G$-invariant function; therefore $m_{\Gamma}(\mathbf{1})=1$, with $\mathbf{1}\subset L^2(\Gamma\setminus G)$ consisting of the subset of constant functions.
\end{remark}

\vspace{0.3cm}
\subsubsection{The representations $\pi_{s,n}$} Recall that given a compactly supported function
\begin{equation*}
F:T\rightarrow \C,
\end{equation*}
one can define its Fourier transform
\begin{equation*}
\widehat{F}:\widehat{T}\rightarrow \C,
\end{equation*}
where $\widehat{T}$ is the unitary dual, and
$$\widehat{F}(\chi) = \int_T F(t) \chi^{-1}(t) dt.$$ In our case, very explicitly, we identify $\widehat{T}=i\mathbb{R}\times \mathbb{Z}$, where $(r,m)$ corresponds to the $U(1)$-character
\begin{equation*}
(u,\theta)\mapsto e^{ru}e^{im\theta}.
\end{equation*}
The Fourier transform is then
\begin{equation*}
\widehat{F}(r,m)=\frac{1}{2\pi}\int_T F(u,\theta) e^{-ru}e^{-im\theta} du d\theta.
\end{equation*}
Of course, this expression also makes sense for $r$ real (because $F$ is compactly supported). We can therefore evaluate the Fourier transform $\widehat{Sf}$ of $Sf: T\rightarrow \mathbb{C}$ at all the characters $\chi_{s,n}$ in Equation \eqref{chisn} parametrizing the irreducible representations of $G$ (thought as elements of $\widehat{T}$ via restriction). Given this, we have the following computation.

\begin{prop} \label{traceinducedrep}
With respect to the Haar measure $dg = da \; dn \; dk,$
$$\mathrm{trace}(\pi_{s,n}(f dg) ) = \widehat{Sf}(\chi_{s,n}^{-1}),$$
where $\widehat{\cdot}$ denotes the Fourier transform
\end{prop}

\begin{proof}
Translating into our notation, Equation $(10.21)$ from \cite{Knapp2} says
\begin{align*}
\mathrm{trace}(\pi_{s,n}(f dg) )&=\frac{1}{2\pi}\int_T\int_N\int_K \delta(t)^{1/2}f(ktnk^{-1})e^{su}e^{in\theta}d\theta \;du\; dn\; dk\\
&= \frac{1}{2\pi} \int_TSf(t)e^{su}e^{in\theta}d\theta \;du\\
&=\widehat{Sf}(\chi_{s,n}^{-1}).
\end{align*}
In their notation, the parameter $(\sigma,\nu)$ corresponds to $(n,s)$, $a=e^u$ and $e^{\rho \log a}=\delta(t)^{1/2}$.
\end{proof}

\vspace{0.3cm}
\subsection{Getting rid of the Satake-Harish-Chandra transform}\label{noshc}
So far, we have succeeded in our first goal of expressing all the summands in the trace formula in terms of the Satake-Harish-Chandra transform $Sf$:
\begin{itemize}
\item combining Proposition \ref{traceinducedrep} and Proposition \ref{tracetrivialrep}, the spectral side of the trace formula for $f dg$ equals 
\begin{equation*} \label{spectralsidesatake}
\text{spectral side}(f dg) = \sum_{s,n} m_\Gamma(\pi_{s,n}) \cdot \widehat{Sf}(\chi_{s,n}^{-1}) + \frac{1}{|W|} \int_T|D(t^{-1})|^{1/2} \cdot Sf(t)  dt. 
\end{equation*}
\item combining the calculation of regular orbital integrals from Proposition \ref{orbitalsatake}, the computation of Subsection \ref{covolume}, and Proposition \ref{plancherel}, the geometric side of the trace formula for $f dg$ equals
\begin{equation*}\label{geometricsidesatake}
\text{geometric side}(f dg) = -c \left(  \frac{d^2}{du^2} + \frac{d^2}{d\theta^2} \right) Sf |_{t = 1} + \sum_{[\gamma] \neq 1} \ell(\gamma_0)\cdot |D(t_\gamma^{-1})|^{-1/2} \cdot Sf(t_\gamma),
\end{equation*}
where the sum runs over all non-trivial conjugacy classes in $\Gamma.$  The constant $c$ is the same as in the statement of Proposition \ref{plancherel}.  
\end{itemize}

Having expressed all terms of the trace formula, applied to $f dg,$ in terms of $Sf,$ it is essential to understand the image of the Satake-Harish-Chandra transform.  This was answered by Bouaziz \cite{Bouaziz} for all real semisimple groups $G.$  We state Bouaziz's theorem only in the special case $G = \mathrm{PGL}_2(\C).$

\begin{thm}[\cite{Bouaziz}] \label{bouazizimage}
For $G = \mathrm{PGL}_2(\C),$ every smooth, compactly supported, $W$-invariant function on $T$ is of the form $Sf$ for some smooth, compactly supported function $f$ on $G.$
\end{thm}  

In particular, this allows us to rephrase our computations purely in terms of a function $F:T\rightarrow \C$. Recall that in our setup the Weyl group $W$ consists of two elements and is generated by $(u,\theta)\mapsto (-u,-\theta)$.

\begin{cor}[Preliminary geometric trace formula] \label{preliminarygeometrictraceformula}
Let $F$ be any smooth, compactly supported function on $T$ for which $F(u,\theta)=F(-u,-\theta)$. The equality
\begin{align*}
&{} \sum_{s,n} m_\Gamma(\pi_{s,n}) \cdot \widehat{F}(\chi_{s,n}^{-1}) + \frac{1}{|W|} \int_T |D(t^{-1})|^{1/2} \cdot F(t)  dt \\
&=  -c \cdot \vol(Y) \cdot \left(  \frac{d^2}{du^2} + \frac{d^2}{d\theta^2} \right) F |_{t = 1} + \sum_{[\gamma] \neq 1} \ell(\gamma_0) \cdot |D(t_\gamma^{-1})|^{-1/2} \cdot F(t_\gamma), 
\end{align*}
holds, where $c$ is the constant from Proposition \ref{plancherel}.
\end{cor}

\begin{remark}\label{isomorphicinducedreps}The result of Bouaziz can also be used to understand coincidences between pairs irreducible unitary representations $\pi_{s,n}$, as mentioned in Subsection \ref{unitarydual}. This is because the representations $\pi_{s,n}$ and $\pi_{s',n'}$ are isomorphic iff they have equal traces, i.e. $\mathrm{trace}( \pi_{s,n}(f dg) ) = \mathrm{trace}(\pi_{s',n'}(f dg))$ for all smooth, compactly supported functions $f$ on $G.$  Equivalently by Proposition \ref{traceinducedrep}, for all smooth compactly supported $f$ we have
\begin{equation*}
\widehat{Sf}(\chi_{s,n}^{-1}) = \mathrm{trace}( \pi_{s,n}(f dg) ) = \mathrm{trace}( \pi_{s',n'}(f dg) )=\widehat{Sf}( \chi_{s',n'}^{-1}).
\end{equation*}
By Theorem \ref{bouazizimage}, the latter is equivalent to
$$\widehat{H}(\chi_{s,n}^{-1}) = \widehat{H}(\chi_{s',n'}^{-1})$$
for all $W$-invariant, compactly supported functions $H$ on $T.$  But this is only possible if $(s',n') = (s,n)$ or $(\bar{s}, -n)$.
\end{remark}

\vspace{0.3cm}
\subsection{Irreducible representations and coclosed $1$-forms.}\label{rep1form}
As written, the formula in Corollary \ref{preliminarygeometrictraceformula} is still too general, as it includes contributions from the eigenvalue spectrum of all natural differential operators on $\Gamma\setminus G/ K$ and not just the coexact 1-form eigenvalue spectrum.  In order to find a trace formula for coexact forms, we first need to understand which representations $\pi_{s,n}$ from Proposition \ref{allirrep} contribute to the spectrum on coclosed $1$-forms, and then we need to choose suitable test functions that isolate their contribution. The goal of this subsection is to tackle the first question, which will be answered in Proposition \ref{isococlosed} below.
\\
\par
We begin by discussing the representation theoretic interpretation of differential forms on $Y=\Gamma\setminus G/K$. We denote by $\mathfrak{p}_0\subset \mathfrak{g}$ the subspace $i\mathfrak{su}_2(\C)$ consisting of $2 \times 2$ traceless hermitian matrices. As $K$ is the stabilizer of $(0,1)\in\mathbb{H}^3$, we have the natural identification of $K$-representations
\begin{equation*}
\mathfrak{p}_0= T_{(0,1)}\mathbb{H}^3.
\end{equation*}
We will use the notation $\mathfrak{p}=\mathfrak{p}_0\otimes \C$, and interpret it as the complexified tangent space to $\mathbb{H}^3$ at $(0,1)$. Notice that the quotient map
\begin{equation*}
\Gamma\setminus G\rightarrow \Gamma\setminus G/K
\end{equation*}
is a principal $K$-bundle. Furthermore, the bundle of complex valued differential forms $\Omega^k(Y,\C)$ is the vector bundle associated to $\mathfrak{p}$ via the natural representation of $K$ on the dual exterior algebra $(\wedge^k\mathfrak{p})^{\vee}$. We therefore have the following.

\begin{lemma}[Matsushima] \label{Matsushima}
For $k = 0,1,2,3,$ there is a natural identification
$$\Omega^k(\Gamma \backslash \mathbb{H}^3,\C) = \mathrm{Hom}_K(\wedge^k \mathfrak{p}, C^{\infty}(\Gamma \backslash G)),$$
hence by (\ref{decrep}) the orthogonal Hilbert space decomposition
\begin{equation*}
L^2 \Omega^k(\Gamma \backslash \mathbb{H}^3,\C) = \bigoplus_{\pi \in \widehat{G}} m_\Gamma(\pi) \cdot \mathrm{Hom}_K(\wedge^k \mathfrak{p}, \pi)
\end{equation*}
holds.
\end{lemma}
The latter should be thought (as it will be made clearer later in \S \ref{HodgeLaplacian} and specifically Proposition \ref{kuga}) as a representation theoretic version of the eigenspace decomposition of the Hodge Laplacian. We will refer to non-zero elements in $\mathrm{Hom}_K(\wedge^k \mathfrak{p}, \pi)$ as $\wedge^k \mathfrak{p}$-\textit{isotypic} vectors in $\pi$. The main result of this subsection is the following.
\begin{prop}\label{isococlosed}
Consider the space of coclosed $1$-forms $\ker d^*\subset  \Omega^1$, and denote by $\overline{\ker d^*}$ its $L^2$-closure. 
Then, under the identification of Lemma \ref{Matsushima}, we have the orthogonal decomposition
\begin{equation*}
\overline{\ker d^*}=\bigoplus_{s\in i\mathbb{R}} m_{\Gamma}(\pi_{s,1})\cdot \mathrm{Hom}_K(\wedge^1 \mathfrak{p},\pi_{s,1}),
\end{equation*}
where the $\{\pi_{s,1}\}_{s\in i\mathbb{R}}$ are  unitary irreducible representations described explicitly in Subsection \ref{unitarydual}. Furthemore, each $\mathrm{Hom}_K(\wedge^1 \mathfrak{p},\pi_{s,1})$ is $1$-dimensional. Finally, as in the decomposition (\ref{decrep}), each multiplicity $m_{\Gamma}(\pi_{s,1})$ is finite and for only countably many $s\in i\mathbb{R}$ we have $m_{\Gamma}(\pi_{s,1})\neq0$.
\end{prop}
Roughly speaking, this proposition says that among all unitary irreducible subrepresentations contained in $L^2(\Gamma\setminus G)$, exactly those of the form $\pi_{s,1}$ contribute to the coclosed 1-form spectrum.  
\\
\par
The rest of the subsection is dedicated to the proof of Proposition \ref{isococlosed}. The first step is to understand which of the representations $\pi_{s,n}$ contain a $\wedge^k \mathfrak{p}$-isotypic vector. In what follows, we denote by $\wedge^k\mathfrak{p}|_M$ the representation of $M$ obtained by restriction (via the inclusion $M\subset K$).
\begin{lemma} \label{frobeniusreciprocity}
For all representations $\pi_{s,n}$ there is a canonical isomorphism of vector spaces
\begin{equation*}
\mathrm{Hom}_K(\wedge^k \mathfrak{p}, \pi_{s,n})\cong \mathrm{Hom}_M(\wedge^k\mathfrak{p}|_M, \chi_n)
\end{equation*}
where we recall that
\begin{equation*}
M=\left\{\left(\begin{array}{cc} e^{i\theta} &0\\0&1\end{array} \right): \theta \in \mathbb{R} / 2\pi \mathbb{Z} \right\}
\end{equation*}
and $\chi_n$ is the $1$-dimensional representation of $M$ with character
\begin{equation*}
\chi_n\left(\begin{array}{cc} e^{i\theta} &0\\0&1\end{array} \right)=e^{in\theta},
\end{equation*}
cf. Equation \eqref{chisn}.
\end{lemma}
\begin{proof}
We claim that there is an isomorphism of $K$-representations
\begin{equation}\label{isoind}
\pi_{s,n}\cong \mathrm{Ind}^K_M\chi_n.
\end{equation}
(This means, in particular, that for fixed $n$ and varying $s,$ the representations $\pi_{s,n}$ are all isomorphic as $K$-representations). Given this, one readily concludes because
\begin{equation*}
\mathrm{Hom}_K(\wedge^k \mathfrak{p}, \pi_{s,n})\cong \mathrm{Hom}_K(\wedge^k \mathfrak{p},\mathrm{Ind}^K_M\chi_n) \cong \mathrm{Hom}_M(\wedge^k\mathfrak{p}|_M, \chi_n)
\end{equation*}
where the second isomorphism is given by Frobenius reciprocity, i.e. the fact that restriction $|_M$ and induction $\mathrm{Ind}^K_M$ are adjoint functors \cite[Theorem $9.9$]{Knapp}.
\par
To see why (\ref{isoind}) holds, notice that $\mathrm{Ind}^K_M\chi_n$ consists of functions $f:K\rightarrow\C$ such that
\begin{equation*}
f(mk)=\chi_n(m)f(k)\text{ for all }m\in M, k\in K.
\end{equation*}
Noticing that $G=BK$, $M=B\cap K$, and the modular function $\delta$ of $B$ is trivial when restricted to $M$, one can associate to such $f$ a well-defined function $\tilde{f}:G\rightarrow \C$ satisfying
\begin{equation*}
\tilde{f}(bk)=\chi_{s,n}(b)\delta^{1/2}(b)\tilde{f}(k).
\end{equation*}
One readily checks that $\tilde{f}$ belongs to $\pi_{s,n}=\mathrm{Ind}_B^G(\chi_{s,n}\cdot\delta^{1/2})$ (see Equation (\ref{funcsn})), and that this assignment is an isomorphism.
\end{proof}

For our specific case, recalling that the action of $K\cong \mathrm{SO}_3$ on $\wedge^1\mathfrak{p}$ is the complexification of the standard representation on $\mathbb{R}^3$, we have the isomorphism
\begin{equation*}
\wedge^1\mathfrak{p}|_M\cong \chi_1\oplus \chi_{-1}\oplus \chi_0
\end{equation*}
as representations of $M\cong \mathrm{SO}_2\subset \mathrm{SO}_3$. Here we use that the complexification of the standard action of $\mathrm{SO}_2$ on $\mathbb{R}^2$ is isomorphic to $\chi_1\oplus \chi_{-1}$. Furthermore, as each $\chi_n$ is $1$-dimensional, we readily check
\begin{equation*}
\dim_{\mathbb{C}} \mathrm{Hom}_M(\chi_n,\chi_m)=
\begin{cases}
1\text{ if }n=m\\
0\text{ otherwise.}
\end{cases}
\end{equation*}
Hence Lemma \ref{frobeniusreciprocity} implies that the representation $\pi_{s,n}$ (where we use the parametrization of Proposition \ref{allirrep}) contains $\wedge^1\mathfrak{p}$-isotypic vectors if and only if $n=0$ or $1$, in which case the space of such vectors is $1$-dimensional. Notice also that for the trivial representation $\mathbf{1}$ we clearly have
\begin{equation*}
\mathrm{Hom}_K(\wedge^1\mathfrak{p},\mathbf{1})=0.
\end{equation*}
Therefore, in the case of $1$-forms Matsushima's Lemma can be simplified to
\begin{equation}\label{mats1}
L^2 \Omega^1 = \bigoplus_{\substack{s\in i\mathbb{R}\\n=0, 1}} m_\Gamma(\pi_{s,n}) \cdot \mathrm{Hom}_K(\wedge^1 \mathfrak{p}, \pi_{s,n}).
\end{equation}
The proof of Proposition \ref{isococlosed} is then completed by the next lemma.
\begin{lemma} \label{exactvscoclosed}
Vectors in $\mathrm{Hom}_K(\wedge^1 \mathfrak{p}, \pi_{s,0})$ correspond to exact $1$-forms, while vectors in $\mathrm{Hom}_K(\wedge^1 \mathfrak{p}, \pi_{s,1})$ correspond to coclosed $1$-forms.
\end{lemma}

Before proving the lemma, we need to discuss the Hodge star $\ast$ and the exterior derivative $d$ in a representation theoretic framework. For an irreducible subrepresentation $\pi \subset L^2(\Gamma \backslash G),$ we have the commutative diagram

$$\begin{CD}
\mathrm{Hom}_K(\wedge^k \mathfrak{p}, \pi) @>>> \mathrm{Hom}_K(\wedge^k \mathfrak{p}, L^2(\Gamma \backslash G)) = L^2 \Omega^k  \\
@V{\ast_\pi}VV @VV{\ast}V \\
\mathrm{Hom}_K(\wedge^{3-k} \mathfrak{p}, \pi) @>>> \mathrm{Hom}_K(\wedge^{3-k} \mathfrak{p}, L^2(\Gamma \backslash G)) = L^2 \Omega^{3-k}
\end{CD}$$
where the horizontal arrows are inclusions, $\ast$ is the Hodge star on $Y$, and $\ast_{\pi}$ is the following: recalling that $\mathfrak{p}$ is identified with $T_{(0,1)}\mathbb{H}^3\otimes\C,$ the operator $\ast_{\pi}$ is defined by precomposition with $\ast$ applied to $T_{(0,1)}\mathbb{H}^3\otimes\C.$ Similarly, we have the diagram

$$\begin{CD}
\mathrm{Hom}_K(\wedge^k \mathfrak{p}, \pi) @>>> \mathrm{Hom}_K(\wedge^k \mathfrak{p}, L^2(\Gamma \backslash G)) = L^2 \Omega^k  \\
@V{d_\pi}VV @VV{d}V \\
\mathrm{Hom}_K(\wedge^{k+1} \mathfrak{p}, \pi) @>>> \mathrm{Hom}_K(\wedge^{k+1} \mathfrak{p}, L^2(\Gamma \backslash G)) = L^2 \Omega^{k+1}
\end{CD}$$
where the horizontal arrows are inclusions, $d$ is the exterior derivative and $d_{\pi}$ is the representation theoretic version of $d$. The explicit formula for $d_{\pi}$ mirrors the invariant formula 
\begin{align*}
d\omega(X_0, \dots , X^k) &= \sum_{i = 0}^k (-1)^i \; X_i \cdot \omega(X_0 , \dots , \widehat{X_i} , \dots , X_q ) \\
&+ \sum_{i < j} (-1)^{i+j} \; \omega( [X_i,X_j] , \dots , \widehat{X_i} , \dots , \widehat{X_j} , \dots ,{X_q})
\end{align*}
for the exterior derivative $d$ evaluated at smooth vector fields $X_0, \ldots, X_q.$  See \cite[Chapter 1, \S 1]{BW} for further details. 

\begin{proof}[Proof of Lemma \ref{exactvscoclosed}]
We have that $\wedge^0\mathfrak{p}\cong\C$ is the trivial representation of $K$, hence $\wedge^0\mathfrak{p}\lvert_M\cong\chi_0$ as representations of $M$. As $m_{\Gamma}(\mathbf{1})=1$ (see Remark \ref{mult1}), we therefore obtain the following analogue of \eqref{Matsushima} for functions:
\begin{equation}\label{mats1}
L^2 \Omega^0 =\mathrm{Hom}_K(\wedge^0\mathfrak{p},\mathbf{1})\oplus \left(\bigoplus m_\Gamma(\pi_{s,0}) \cdot \mathrm{Hom}_K(\wedge^0 \mathfrak{p}, \pi_{s,0})\right),
\end{equation}
where the sum runs over the relevant parameters $s$, and each isotypic summand is $1$-dimensional. In particular, for every $s$, $\mathrm{Hom}_K(\wedge^0 \mathfrak{p}, \pi_{s,0})$ appears in the decomposition of $L^2\Omega^0$ with the same multiplicity as $\mathrm{Hom}_K(\wedge^1 \mathfrak{p}, \pi_{s,0})$ appears in the decomposition of $L^2\Omega^1$.
\par
Now, for a given $\pi$, the operator
\begin{equation*}
d_{\pi}: \mathrm{Hom}_K(\wedge^0 \mathfrak{p}, \pi) \rightarrow \mathrm{Hom}_K(\wedge^1 \mathfrak{p}, \pi)
\end{equation*}
acts on
\begin{equation*}
\omega:\wedge^0 \mathfrak{p}\cong \C\rightarrow\pi
\end{equation*}
in the following way:
\begin{align*}
d_{\pi}(\omega)&:\wedge^1 \mathfrak{p}\rightarrow\pi\\
d_{\pi}(\omega)(X)&= \pi(X)\cdot\omega(1).
\end{align*}
Of course, $d_{\mathbf{1}}$ is the zero map. We claim that $d_{\pi_{s,0}}$ is injective for every representation $\pi_{s,0}$ (this is the representation-theoretic incarnation of the basic fact that the differential of a non-constant function is not identically zero). In fact, $\omega(1)\in\pi_{s,0}$ is a $K$-invariant, hence it is annihilated by $\mathfrak{k}$\footnote{Notice that $\mathfrak{k}$ is a gothic $k$, and denotes the Lie algebra of $K$.}; if we had $d_{\pi_{s,0}}(\omega)=0$, $\omega(1)$ would be annihilated by $\mathfrak{p}$ too, hence by all of $\mathfrak{g}$ since $\mathfrak{g} = \mathfrak{p}_0 \oplus \mathfrak{k}.$ This is a contradiction, because $\C\cdot\omega(1)\subset \pi_{s,0}$ would then be a $G$-subrepresentation. We conclude that $d_{\pi_{s,0}}$ is an isomorphism (as both domain and target space are $1$-dimensional) and therefore all vectors in $\mathrm{Hom}_K(\wedge^1 \mathfrak{p}, \pi_{s,0})$ correspond to exact 1-forms by the commutativity of the diagram relating $d$ and $d_{\pi_{s,0}}.$

For the second part of the statement, since $d^\ast$ equals $\ast d \ast$ up to sign, by the commutativity of the above diagrams it suffices to show that the composition
\begin{equation*}
\mathrm{Hom}_K( \wedge^1 \mathfrak{p}, \pi_{s,1}) \xrightarrow{\ast_{\pi_{s,1}}}  \mathrm{Hom}_K( \wedge^2 \mathfrak{p}, \pi_{s,1}) \xrightarrow{d_{\pi_{s,1}}}  \mathrm{Hom}_K( \wedge^3 \mathfrak{p}, \pi_{s,1}) \xrightarrow{\ast_{\pi_{s,1}}}  \mathrm{Hom}_K( \wedge^0 \mathfrak{p}, \pi_{s,1})  
\end{equation*}
vanishes.  This is indeed true because the target $\mathrm{Hom}_K( \wedge^0 \mathfrak{p}, \pi_{s,1})= \mathrm{Hom}_M( \wedge^0 \mathfrak{p}|_M, \chi_1)$
vanishes.
\end{proof}

\vspace{0.3cm}
\subsection{Irreducible representations and the Hodge Laplacian} \label{HodgeLaplacian}

In this section we discuss the precise sense in which the decomposition of Matsushima's Lemma (Lemma \ref{Matsushima}) can be interpreted in terms of the eigenspace decomposition of the Hodge Laplacian. The main result is the following.

\begin{prop}\label{spectraldecrep}
Referring to the decomposition of Proposition \ref{isococlosed}, each isotypic summand $\mathrm{Hom}_K(\wedge^1\mathfrak{p}, \pi_{s,1})\cong \C$ corresponds to a one dimensional eigenspace of the Hodge Laplacian on coclosed $1$-forms on $Y$ with eigenvalue $-s^2\in\mathbb{R}^{\geq0}$.
\end{prop}

The rest of this subsection is dedicated to its proof.

\vspace{0.3cm}
\subsubsection{Killing forms and the hyperbolic metric.} We set our notation and conventions for the Killing forms.
Recall that the Lie algebra $\mathfrak{g}$ is $\mathfrak{gl}_2(\C)/\mathbb{C}$. We will consider the basis over $\mathbb{C}$ consisting of
\begin{equation*}
H = \left( \begin{array}{cc} 1 & 0 \\ 0 & 0 \end{array} \right), E = \left( \begin{array}{cc} 0 & 1 \\ 0 & 0 \end{array} \right), F = \left( \begin{array}{cc} 0 & 0 \\ 1 & 0 \end{array} \right),
\end{equation*}
which satisfies the commutation relations
\begin{equation}\label{commg}
[H,E]=E,\quad [H,F]=-F,\quad [E,F]=2H.
\end{equation}
\begin{remark}\label{aboutH}
The element $H$ above is not the standard element used for the basis of $\mathfrak{sl}_2(\C)$. Under the natural identification $\mathfrak{g}\cong\mathfrak{sl}_2(\C)$, it corresponds to $\left( \begin{array}{cc} 1/2 & 0 \\ 0 & -1/2 \end{array} \right)$.
\end{remark}
In what follows, we consider $\mathfrak{g}$ as a real Lie algebra; a basis over $\mathbb{R}$ is given by
\begin{equation}\label{rbasis}
H,\; iH,\; E,\; iE,\; F,\; iF.
\end{equation}
Consider the Killing form
$$B(X,Y) := \mathrm{trace}(\mathrm{ad}(X)_\R \circ \mathrm{ad}(Y)_\R).$$
We write the subscript $\R$ to emphasize that we must view $\mathrm{ad}(X),\mathrm{ad}(Y)$ as $\R$-linear transformations of the complex vector space $\mathfrak{g}.$
The Killing form $B$ is non-degenerate (as $B$ is semisimple), and induces a positive-definite inner product on $\mathfrak{p}_0=T_{(0,1)}\mathbb{H}^3$:
$$\langle X,Y \rangle_0 = -B(X,Y).$$
The inner product $\langle \cdot,\cdot \rangle_0$ is $K$-invariant and thus propogates to an invariant metric on all of $\mathbb{H}^3.$  We call this metric $g_{\mathrm{Killing}}.$  Notice that $\frac{1}{4}g_{\mathrm{Killing}}$ equals the standard curvature $-1$ metric on $\mathbb{H}^3.$ Indeed, this is a metric of constant negative curvature \cite[Section 7.G]{Besse}. Furthermore, one can check the normalization by noticing that $H\in\mathfrak{p}_0$ (for which $\|H\|_{\mathrm{Killing}}=2$) generates the one-parameter family
\begin{equation*}
e^{tH}=\left(\begin{array}{cc} e^{t/2} & 0 \\ 0 & e^{-t/2} \end{array} \right)\in\mathrm{PSL}(2,\mathbb{C})
\end{equation*}
(see Remark \ref{aboutH}). Via Equation (\ref{upperaction}), this corresponds in $\mathbb{H}^3=G/K$ (where again $K$ is the stabilizer of $(0,1)$) to the geodesic $(0, e^t)$, whose tangent vector at $t=0$ has length $1$.

\vspace{0.3cm}
\subsubsection{Differential operators and representation theory.}
Recall that the Lie algebra $\mathfrak{g}$ consists of the set of left-invariant vector fields on $G$, and can be therefore thought as the set of left-invariant first order differential operators on $G$. From this viewpoint, the universal enveloping algebra $U(\mathfrak{g})$ of $\mathfrak{g}$ \cite[Section III$.1$]{Knapp} corresponds to the set of of general left invariant differential operators on $G$. Because $\mathfrak{g}$ is semisimple, there is a distinguished element of the center of $U(\mathfrak{g})$ called the \textit{Casimir element} \cite[Section V.4]{Knapp}. It is defined as
\begin{equation*}
C=\sum_{i=1}^6 X_iX_i^{\vee}\in Z(U(\mathfrak{g}))
\end{equation*}
where $\{X_i\}$ is any $\mathbb{R}$-basis of $\mathfrak{g}$, and $\{X_i^\vee\}$ is the dual basis with respect to the Killing form $B$ (which is non-degenerate). Its definition is independent of the choice of basis. The fact that $C$ is in the center of $U(\mathfrak{g})$ has two direct consequences:
\begin{itemize}
\item The corresponding left-invariant differential operator on $G$ is in fact bi-invariant, and descends therefore to a differential operator on $Y=\Gamma\setminus G/K$.
\item By Schur's lemma, $C$ acts on a given irreducible representation $\pi$ of $G$ by a scalar. We denote this value by $C(\pi)$, the \textit{Casimir eigenvalue} of $\pi$.
\end{itemize}

The next result puts these two observations together to show that the decomposition from Matsushima's Lemma \ref{Matsushima} is a refinement of the spectral decomposition of the Hodge Laplacian on $k$-forms.

\begin{prop}[Kuga's Lemma, \cite{BC}, Lemme 1.1.1] \label{kuga}
Given an irreducible representation $\pi \subset L^2(\Gamma \backslash G),$ every $\wedge^k \mathfrak{p}$-isotypic vector in $\pi$ corresponds to a $k$-eigenform of the Hodge Laplacian on $Y$ of eigenvalue $\lambda = -C(\pi)$, the Casimir eigenvalue of $\pi$. Furthermore, every $k$-eigenform of the Hodge Laplacian on $Y$ of eigenvalue $\lambda$ arises in this way. Here, the Hodge Laplacian is the one corresponding to the metric $g_{\mathrm{Killing}}$ on $\mathbb{H}^3$.
\end{prop}

Given this, in order to prove Proposition \ref{spectraldecrep}, all we need to do is compute the Casimir eigenvalue of the representations $\pi_{s,1}$, which can done directly as follows (we will perform the computation for general $(s,n)$, as it is identical).
\par
Calculating from the definition using the $\mathbb{R}$-basis \eqref{rbasis} and the commutation relations \eqref{commg}, we can write
\begin{equation} \label{casimirPGL2C}
C  = \frac{1}{4} H \cdot H - \frac{1}{4} H - \frac{1}{4} (iH) \cdot (iH) - \frac{1}{4} H + E \cdot (\text{extra term}) + (iE) \cdot (\text{extra term}),
\end{equation}
where the extra terms belong to $\mathfrak{g} \subset U(\mathfrak{g}).$  Consider a smooth function $f$ in $\pi_{s,n}$ with $f(1)\neq 0$; recall that this function satisfies
\begin{equation}\label{equivariancesn}
f(bg) = \delta(b)^{1/2} \chi_{s,n}(b) f(g)\text{ for all }b \in B.
\end{equation}
By definition, the Casimir eigenvalue satisfies $Cf(1)=C(\pi_{s,n})f(1)$, so it suffices to evaluate $Cf(1).$

Note that for all $T \in U(\mathfrak{g}),$ the equivariance property \eqref{equivariancesn} implies that
\begin{align*}
(ET \cdot f)(1) &= E (Tf)(1) \\
&= \frac{d}{dt}\lvert_{t = 0} (Tf)( e^{t E}) \\
&= \frac{d}{dt}\lvert_{t = 0} (Tf)(1)=0
\end{align*}
because both $\delta$ and $\chi_{s,n}$ are trivial on the one parameter family
\begin{equation*}
e^{t E}=\left(\begin{array}{cc}1& t \\0&1\end{array}\right),\quad t\in\mathbb{R}.
\end{equation*}
Similarly, $((iE)T \cdot f)(1) = 0.$  By \eqref{casimirPGL2C} we then compute,
\begin{align*}
(Cf)(1) &= \left( \left(\frac{1}{4} H \cdot H - \frac{1}{2} H - \frac{1}{4} (iH) \cdot (iH) \right) \cdot f \right) (1) \\
&= \frac{1}{4} \frac{\partial^2}{\partial u \partial t}|_{(0,0)} f( e^{uH} e^{tH}) - \frac{1}{2} \frac{\partial}{\partial t}|_{0} f(e^{tH}) - \frac{1}{4} \frac{\partial^2}{\partial u \partial t}|_{(0,0)} f(e^{t iH} e^{u iH}) \\
&= \frac{1}{4} [ (s+1)^2 - 2(s+1) + n^2 ] f(1) \\
&= \frac{1}{4}( s^2 + n^2 - 1) f(1).
\end{align*}
Therefore we conclude that the Casimir eigenvalue equals
\begin{equation}\label{casimireigen}
C(\pi_{s,n}) = \frac{1}{4}( s^2 + n^2 - 1).
\end{equation}
In particular, by Kuga's Lemma, the $\wedge^1 \mathfrak{p}$-isotypic vector in $\pi_{s, 1}$ corresponds to a coclosed $1$-eigenform on $\Gamma \backslash \mathbb{H}^3$ of Laplace eigenvalue $- \frac{1}{4} s^2,$ when $\mathbb{H}^3$ is endowed with the metric $g_{\mathrm{Killing}}.$  In the standard curvature $-1$ metric on $\mathbb{H}^3,$ this eigenvalue becomes $-s^2,$ and Proposition \ref{spectraldecrep} is proved.
\begin{remark}
From Equation \eqref{casimireigen}, by setting $n=0$ we also recover the more classical fact that the relevant spectral parameter for the Laplacian on functions on a hyperbolic three-manifold is $1+r^2$, with $r\in \mathbb{R}^{\geq 0}\cup i[0,1]$, see \cite{Sarnak}. Here the number $1$ is the bottom of the $L^2$-spectrum of the Laplacian on functions on $\mathbb{H}^3$, cf. Remark \ref{onequart}.
\end{remark}

\vspace{0.3cm}
\subsection{Choosing test functions to isolate coclosed $1$-forms}\label{iso1form}
We are now in good shape to specialize the trace formula of Corollary \ref{preliminarygeometrictraceformula} to a formula only involving the spectrum on coclosed $1$-forms: we learned from Proposition \ref{isococlosed} that only representations of the form $\pi_{s,1}$, $s\in\mathbb{R}$ contribute to the coclosed 1-form spectrum, and from Proposition \ref{spectraldecrep} that each copy of this representation corresponds to a one dimensional eigenspace of eigenvalue $-s^2$. To isolate the contribution of these representations, the natural candidate are functions of the form
$$F(u,\theta) = H(u) \cos \theta$$
where $H$ an even, compactly supported, \emph{$\R$-valued} function on $\R$. In fact, one readily computes, denoting $s=it$ for $t\in\mathbb{R}$,
\begin{equation*}
\widehat{F}(\chi_{it,n}^{-1}) = \begin{cases} \frac{1}{2}\widehat{H}(-t) & \text{ if } n = \pm 1 \\ 0 & \text{ otherwise.} \end{cases}
\end{equation*}
We will now unravel in terms of this function each term of the formula in Corollary \ref{preliminarygeometrictraceformula}.

\subsubsection{The non-trivial representations}
By Proposition \ref{spectraldecrep} we have
\begin{equation*}
m_{\Gamma}(\pi_{it,1}) + m_\Gamma(\pi_{-it,1}) =: m_\Gamma( t^2 ),
\end{equation*}
the multiplicity of $t^2$ in the spectrum on coclosed $1$-forms. Using that $\widehat{H}$ is an even function, the first term in the formula of Corollary \ref{preliminarygeometrictraceformula} is therefore
\begin{equation*}
\frac{1}{2}\sum m_{\Gamma}(t^2)\widehat{H}(t)
\end{equation*}
where the sum runs over the values for which $t^2$ is a coclosed eigenvalue on $1$-forms.

\subsubsection{The trivial representation}
Setting as usual, $z = e^{u + i \theta},$ the contribution of the trivial representation to the geometric trace formula for the test function $F$ equals
\begin{align*}
&\frac{1}{|W|}\int_T |D(t^{-1})|^{1/2} F(t) dt= \\ 
&=\frac{1}{2} \cdot \frac{1}{2\pi} \int \int|1-z| \cdot |1-z^{-1}|\cdot H(u) \cdot \cos\theta \; du \; d\theta \hspace{0.5cm} \text{because } |W| = 2 \\
&= \frac{1}{2} \cdot \frac{1}{2\pi} \int \int (e^u + e^{-u} - 2 \cos \theta)  \cdot \cos\theta \cdot H(u) \; d\theta \; du  \\
&= -\frac{1}{2} \int H(u) \; du \\
&= - \frac{1}{2} \widehat{H}(0)
\end{align*}
where we computed $|1-z| \cdot |1-z^{-1}|=e^u + e^{-u} - 2 \cos \theta$.

\subsubsection{The identity contribution}
The identity contribution to the trace formula for the test function $F$ is
\begin{align*}
&= -c \cdot \vol(Y) \cdot \left(  \frac{d^2}{du^2} + \frac{d^2}{d\theta^2} \right) F |_{(u,\theta) = (0,0)} \\
&=  -c \cdot \vol(Y) \cdot \left(  \frac{d^2}{du^2} + \frac{d^2}{d\theta^2} \right) ( H(u) \cos \theta) |_{(u,\theta) = (0,0)} \\
&= c \cdot \vol(Y) \cdot (H(0) - H''(0)),
\end{align*}
where $c$ is the constant from Proposition \ref{plancherel}.

\subsubsection{The sum over closed geodesics}
Finally, setting as usual
\begin{equation*}
t_\gamma = \left( \begin{array}{cc} e^{u + i \theta} & 0 \\ 0 & 1 \end{array} \right) \in T
\end{equation*}
where $u+i\theta=\C \ell(\gamma)$, each closed geodesics in the sum contributes
\begin{align*}
&\ell(\gamma_0) \cdot |D(t_\gamma^{-1})|^{-1/2} \cdot F(t_\gamma) \\
=& \ell(\gamma_0) \cdot \left(  |1 - e^{u+i\theta}| \cdot |1 - e^{-(u+i\theta)}| \right)^{-1} \cdot H(u) \cos \theta \\
=& \ell(\gamma_0) \cdot \left( | 1 - e^{\C \ell(\gamma)} | \cdot |1 - e^{- \C \ell(\gamma)} | \right)^{-1} \cdot H(\ell(\gamma)) \cdot \cos( \mathrm{hol}(\gamma) ).
\end{align*}
\vspace{0.3cm}
\subsection{Conclusion of the proof}\label{theend}
Combining all the computations from the previous subsection, and noticing that coclosed $1$-eigenforms with eigenvalue $0$ are harmonic $1$-forms, we obtain the following.
\begin{cor}[Theorem \ref{geometrictraceformulacoexact1formsnew}, up to a costant $c$ to be determined]\label{preliminarygeometrictraceformulacoexact1forms}
Let $H$ be any smooth, compactly supported, even, $\R$-valued function on $\R.$  There is an equality
\begin{align*}
&{} \sum_{\lambda^\ast = \text{coexact 1-form eigenvalue}} \frac{1}{2} m_\Gamma(\lambda^*) \cdot \widehat{H} \left( \sqrt{\lambda^\ast} \right) + \left( \frac{1}{2} b_1\left( Y \right)- \frac{1}{2} \right) \widehat{H}(0) \\
&=  c \cdot \vol(Y) \cdot \left(  H(0) - H''(0) \right) + \sum_{[\gamma] \neq 1} \ell(\gamma_0) \cdot \frac{\cos( \mathrm{hol}(\gamma))}{ |1 - e^{\C \ell(\gamma)} | \cdot |1 - e^{-\C \ell(\gamma)}|} \cdot H(\ell(\gamma)). 
\end{align*}
In the above formula, $c$ is the constant from Proposition \ref{plancherel}.
\end{cor}
To conclude the proof of Theorem \ref{geometrictraceformulacoexact1formsnew}, we only need to evaluate the missing constant $c$. Using Weyl's law, which we now recall, this can be done by choosing suitable functions for which the main contribution from the geometric side comes from the term involving the volume. Denote by $N_{d^\ast}(X)$ the number of coexact 1-form eigenvalues on $Y = \Gamma \backslash \mathbb{H}^3$ satisfying $\sqrt{\lambda^*} \leq X$ (counted with multiplicity). Recall that the spectrum of the Laplacian on $1$-forms is the union of the spectrum on coclosed $1$-forms, and the non-trivial spectrum on functions. Therefore, applying the general Weyl law for vector bundles \cite[Corollary 2.43]{BGV}, we have for $X$ very large
\begin{equation*}
N_{d^\ast}(X)= 2\cdot\frac{ \vol(Y)}{(4\pi)^{3/2} \Gamma(3/2+1)} X^3+o(X^3)
\end{equation*}
where the first coefficient $2=3-1$ is the difference between the dimension of the bundle of $1$-forms and the bundle of functions.
\par
Fix a smooth, compactly supported, real valued even test function $H$ with $H(0) \neq 0$ and positive Fourier transform (see Section \ref{bookersec} for some concrete examples).  Let $H_\nu = H \cdot (e^{i\nu u} + e^{-i\nu u}).$  We integrate the spectral side of the trace formula from Corollary \ref{preliminarygeometrictraceformulacoexact1forms} for the test function $H_\nu$ over $\nu \in [-X,X]$. As $\widehat{H}_{\nu}(t)=\widehat{H}(t-\nu)+\widehat{H}(t+\nu)$, a parameter $\sqrt{\lambda^*}$ contributes
\begin{equation*}
\frac{1}{2}\int_{[-X,X]} \widehat{H}(\sqrt{\lambda^*}-\nu)+\widehat{H}(\sqrt{\lambda^*}+\nu)d\nu
\end{equation*}
to the sum. This is close to $2\pi H(0)$ for $X$ very large compared to $\sqrt{\lambda^*}$ by the Fourier inversion formula
\begin{equation*}
\int_{\mathbb{R}}\widehat{H}(t)dt=2\pi H(0).
\end{equation*}
From this, one obtains with some basic estimates the asymptotic for $X$ very large
\begin{align*}
\int_{\nu \in [-X,X]} \text{spectral side for } H_\nu \; d\nu = 2\pi  H(0)\cdot\underbrace{\frac{2 \cdot \vol(Y)}{(4\pi)^{3/2} \Gamma(3/2+1)} X^3}_{\text{leading term of }N_{d^\ast}(X)}+o(X^3).
\end{align*}
On the other hand, the main contribution from the geometric side clearly comes from integrating $-H''_{\nu}(0)= 2\cdot H(0)\cdot\nu^2$, and we compute for $X$ very large
\begin{align*}
\int_{\nu \in [-X,X]} \text{geometric side for } H_\nu \; d\nu &= c \cdot \vol(Y) \underbrace{\cdot 2 \cdot H(0) \cdot \frac{2X^3}{3}}_{\int_{-X}^X-H''_{\nu}(0)d\nu}+o(X^3).
\end{align*}
Equating the above two asymptotic expansions yields $c = \frac{1}{2\pi}$, and the proof of Theorem \ref{geometrictraceformulacoexact1formsnew} is completed.


\vspace{0.5cm}

\section{Limiting argument: Proof of Theorem \ref{geometrictraceformulallimitnew}} \label{limitargument}

\begin{proof}
Let $\delta > \frac{5}{2}.$  Express $\delta = \alpha + \beta$ with $\alpha \geq 1$ and $\beta > \frac{3}{2}.$  Let $\langle t \rangle := (1 + t^2)^{\frac{1}{2}}.$  Then
\begin{align*}
\left| \sum_n \widehat{H}(t_n) \right| &= \left|  \sum \widehat{H}(t_n) \langle t_n \rangle^{\alpha + \beta} \langle t_n \rangle^{-\alpha} \langle t_n \rangle ^{-\beta} \right| \\
&\leq \sqrt{\sum \left| \widehat{H}(t_n) \langle t_n \rangle^{\alpha + \beta} \right|^2 \langle t_n \rangle ^{-2\alpha} } \cdot \sqrt{ \sum \langle t_n \rangle^{-2\beta}  } \hspace{0.5cm} \text{by Cauchy-Schwarz}. \\ 
\end{align*}
By the local Weyl law (cf. \cite[Lemma 2.3]{DG} and \cite[Lemma 2.2]{Muller}), 
\begin{equation} \label{localweyl} 
\# \{ t_n \in [N, N+1]  \} \leq D N^2
\end{equation}
for some $D > 0.$  Since $2\beta > 3,$ it follows that the second summand above is convergent.  Say it equals $C.$  We continue:

\begin{align*}
&=  \sqrt{C} \cdot \sqrt{\sum \left| \widehat{H}(t_n)\langle t_n \rangle^{\alpha + \beta} \right|^2 \langle t_n \rangle^{-2\alpha} } \\
&\leq  \sqrt{C} \cdot \sqrt{\sum_{N = 0}^{\infty} \sup_{t \in [N, N+1]} \left| \widehat{H}(t)\langle t \rangle^{\alpha + \beta} \right|^2 \cdot \left( \langle N \rangle^{-2\alpha} \cdot \# \left\{ t_n \in [N,N+1] \right\} \right) } \\
&\leq \sqrt{C}  \sqrt{D} \cdot \sqrt{\sum_{N = 0}^{\infty} \sup_{t \in [N, N+1]} \left| \widehat{H}(t)\langle t \rangle^{\alpha + \beta} \right|^2 } \hspace{0.5cm} \text{by \eqref{localweyl} because } \alpha \geq 1.
\end{align*}

There is also the Sobolev inequality
\begin{equation} \label{sobolevinequalityinterval}
\sup_{t \in [a,b] } |G(t)|^2 \leq E \cdot \left( ||G||^2_{L^2[a,b]} + || G' ||^2_{L^2[a,b]} \right)
\end{equation}
for all smooth functions $G$ on $[a,b]$ and some constant $E$ uniform in $b-a$. Applying this to $G=\widehat{H}(t)\langle t \rangle^{\alpha + \beta},$ we have 
\begin{align*}
|| G' ||^2_{L^2[N,N+1]} &\leq || \frac{d}{dt} \widehat{H}(t) \langle t \rangle^{\alpha + \beta} ||^2_{L^2[N,N+1]} + ||  \widehat{H}(t) \frac{d}{dt}\langle t \rangle^{\alpha + \beta} ||^2_{L^2[N,N+1]} \\
&\leq  || \frac{d}{dt} \widehat{H}(t) \langle t \rangle^{\alpha + \beta} ||^2_{L^2[N,N+1]} + F || \widehat{H}(t) \langle t \rangle^{\alpha + \beta} ||^2_{L^2[N,N+1]},
\end{align*}
for some constant $F \geq 1$ independent of $N.$  We continue


\begin{align*}
&\leq \sqrt{C}  \sqrt{D} \sqrt{E} \sqrt{F} \cdot \sqrt{ \sum_{N = 0}^{\infty} || \widehat{H}(t) \langle t \rangle^{\alpha + \beta} ||^2_{L^2[N,N+1]} + || \frac{d}{dt} \widehat{H}(t) \langle t \rangle^{\alpha + \beta} ||^2_{L^2[N,N+1]}  } \\
&\leq  \sqrt{C}  \sqrt{D} \sqrt{E} \sqrt{F} \cdot \sqrt{ \sum_{N = 0}^{\infty} || \widehat{H}(t) \langle t \rangle^{\alpha + \beta} ||^2_{L^2[N,N+1]} + || \frac{d}{dt} \widehat{H}(t) \langle t \rangle^{\alpha + \beta} ||^2_{L^2[N,N+1]}  } \\
&\leq  \sqrt{C}  \sqrt{D} \sqrt{E} \sqrt{F} \cdot \sqrt{ || \widehat{H} \langle t \rangle^{\alpha + \beta} ||^2_{L^2(\mathbb{R})} + || \frac{d}{dt} \widehat{H} \langle t \rangle^{\alpha + \beta} ||^2_{L^2(\mathbb{R})} } \\
&\leq \sqrt{C} \sqrt{D} \sqrt{E} \sqrt{F} \sqrt{2 \delta}  \sqrt{ \int_{\mathbb{R}} \left(  \left| \widehat{H}(t) \right|^2 + \left|  \widehat{H}'(t) \right|^2 \right) \langle t \rangle^{2\delta} }.
\end{align*}
Thus, the linear functional
\begin{equation} \label{regularspectrallinearfunctional}
H \mapsto \sum \widehat{H}(t_n),
\end{equation}
corresponding to the regular spectral contribution to the trace formula for the test function $H,$ is continuous in the (weighted) Sobolev space $\mathcal{S}$ defined by the norm
$$|| H ||_{\mathcal{S}} := \sqrt{ \int_{\mathbb{R}} \left(  \left| \widehat{H}(t) \right|^2 + \left|  \widehat{H}'(t) \right|^2 \right) \langle t \rangle^{2\delta} }.$$
It is also readily checked (using again the Sobolev inequalities) that the linear functionals
\begin{align} \label{traceformulafunctionals} 
H &\mapsto H(0) - H''(0)\nonumber \\
H &\mapsto \widehat{H}(0)\nonumber \\
H &\mapsto \sum_{1 \neq \gamma, \ell(\gamma) \leq R} c_{\gamma} H(\ell(\gamma)),
\end{align}
corresponding respectively to the identity contribution, and the trivial representation contribution, and the regular geometric contribution to the trace formula for the test function $H,$ are continuous on $\mathcal{S}.$  Indeed,

\begin{itemize}
\item
To bound the identity contribution in terms of $|| H ||_{\mathcal{S}},$ begin by noting that 
$H(0) - H''(0) = \frac{1}{2\pi} \int_R \widehat{H}(t) (1 + t^2) dt$ by Fourier inversion.  So,
\begin{align*}
|H(0) - H''(0)| &\leq \frac{1}{2\pi} \int_R | \widehat{H}(t) | \langle t \rangle^2 dt \\
& \leq \frac{1}{2\pi} \int_R | \widehat{H}(t) | \langle t \rangle^{\delta} \langle t \rangle^{2-\delta} dt \\
&\leq \frac{1}{2\pi} \left(\leq \int_R | \widehat{H}(t) |^2 \langle t \rangle^{2\delta} dt \right)^{1/2} \cdot \left( \int_R \langle t \rangle^{4-2\delta} dt \right)^{1/2} \hspace{0.5cm} \text{by Cauchy-Schwarz} \\
&\leq C' \cdot  || H ||_{\mathcal{S}},
\end{align*}
where $C' = \frac{1}{2\pi} \left( \int_R \langle t \rangle^{4-2\delta} dt \right)^{1/2}.$   Note that because $\delta > 5/2,$ the latter integral converges.

\item
To bound the regular geometric contribution in terms of $|| H ||_{\mathcal{S}}$-norm, use Fourier inversion
once again: $H(\ell) = \frac{1}{2\pi} \int_R \widehat{H}(t) e^{it.\ell} dt.$  Then 
\begin{align*}
|H(\ell)| &\leq \frac{1}{2\pi} \int_R |\widehat{H}(t) e^{it.\ell}| dt \\
&= \frac{1}{2\pi} \int_R |\widehat{H}(t)| dt \\
&\leq C' ||H||_W
\end{align*}
for the same constant $C'$ as above by the same argument.  Thus, 
\begin{equation*}
\left| \sum_{\ell(\gamma) <= R}  c_{\gamma} H(\ell(\gamma)) \right| \leq D' ||H||_W,
\end{equation*}
where $D' = C' \cdot \sum_{\ell(\gamma) <= R}  |c_{\gamma}|.$
\end{itemize}

If $H$ is supported on $[-R,R],$ then for every $\epsilon > 0,$ there is a sequence $H_m$ of smooth functions supported on $[-(R+\epsilon), R + \epsilon ]$ converging to $H$ in the $\mathcal{S}$-topology\footnote{The sequence of functions we take: approximate $H$ by $K \ast H$ for some smooth approximation $K$ to $\delta_0,$ which we can take to be supported on $[-\epsilon,\epsilon].$  All of these functions are supported on $[-(R + \epsilon), R + \epsilon],$ and they converge to $H$ as our approximation to $\delta_0$ improves.}.  Take such a sequence $H_m$ with $\epsilon$ chosen so that there are no closed geodesics of length in $(R,R+\epsilon).$  It follows that the trace formula is valid for $H$ by taking the limit of both sides of the trace formula applied to $H_m.$  
\end{proof}

\begin{remark}
It is natural to attempt a limiting argument for test functions more general than those from the statement of Theorem \ref{geometrictraceformulallimitnew}. The main difficulty is controlling the geometric side of the trace formula (\ref{traceformulafunctionals}), which has asymptotic to  $e^{2R}/2R$ summands below the length threshold of $R$ \cite{Sarnak}.  
Though we will not pursue it presently, we expect that the trace formula holds true for every test function $H$ satisfying 
$$|| H ||_{\mathcal{S}} + \int_0^\infty |H(x)| e^x dx < \infty.$$ 
Notably, this class includes Gaussian test functions.
\end{remark}

\end{appendix}

\end{document}